%% file: article3.tex
\documentclass[12pt,a4paper]{amsart}

\usepackage{amsmath}
\usepackage{amsfonts}
\usepackage{latexsym}
\usepackage{graphicx}
\usepackage{amssymb}
\usepackage{amsthm}
\usepackage[margin=2cm]{geometry}
\usepackage{url}
\usepackage{color}
\usepackage{enumerate}
\usepackage[shortlabels]{enumitem} %pour commencer les enumerations a des nombres differents
\usepackage[small]{caption}
\usepackage{cite}       % sort and compress references number in latex

\usepackage[T1]{fontenc}     % Hyphénation des mots accentués
\usepackage{lmodern}         % Polices vectorielles
\usepackage[utf8]{inputenc}  % Codage UNICODE (UTF-8)

\usepackage{accents}    % defines \accentset
\usepackage{bbding}     % scissor symbol \ScissorRightBrokenBottom

\usepackage{todonotes}
\usepackage{hyperref}

\usepackage{tikz}
\usetikzlibrary{shapes} % for diamond node shape
\usepackage{tikz-cd}      % commutative diagrams

\newtheorem{definition}{Definition}[section]
\newtheorem{lemma}[definition]{Lemma}
\newtheorem{proposition}[definition]{Proposition}

\newtheorem{corollary}[definition]{Corollary}
\newtheorem{conjecture}[definition]{Conjecture}
\newtheorem{remark}[definition]{Remark}
\newtheorem{theorem}{Theorem}

\newcommand{\N}{\mathbb{N}}
\newcommand{\Z}{\mathbb{Z}}
\newcommand{\Q}{\mathbb{Q}}
\newcommand{\R}{\mathbb{R}}
\newcommand{\generictorus}{{\boldsymbol{T}}}
\newcommand{\torusI}{\mathbb{T}}
\newcommand{\torusII}{\mathbb{T}^2}
\newcommand{\torus}{\torusII}

\newcommand{\A}{\mathcal{A}}
\newcommand{\T}{\mathcal{T}}
\newcommand{\U}{\mathcal{U}}
\newcommand{\Bcal}{\mathcal{B}}
\newcommand{\Dcal}{\mathcal{D}}
\newcommand{\Fcal}{\mathcal{F}}
\newcommand{\Lcal}{\mathcal{L}}
\newcommand{\Pcal}{\mathcal{P}}
\newcommand{\Xcal}{\mathcal{X}}
\newcommand{\Ycal}{\mathcal{Y}}
\newcommand{\Zcal}{\mathcal{Z}}

\newcommand{\zero}{{\boldsymbol{0}}}
\newcommand{\be}{{\boldsymbol{e}}}
\newcommand{\bk}{{\boldsymbol{k}}}
\newcommand{\bn}{{\boldsymbol{n}}}
\newcommand{\bm}{{\boldsymbol{m}}}
\newcommand{\bp}{{\boldsymbol{p}}}
\newcommand{\bv}{{\boldsymbol{v}}}
\newcommand{\bx}{{\boldsymbol{x}}}
\newcommand{\by}{{\boldsymbol{y}}}
\newcommand{\balpha}{{\boldsymbol{\alpha}}}
\newcommand{\bbeta}{{\boldsymbol{\beta}}}

\newcommand{\dist}{\mathrm{dist}}
\newcommand{\card}{\mathrm{card}}
\newcommand{\SFT}{\textsc{SFT}}
\newcommand{\sctop}{\textsc{top}}
\newcommand{\scbottom}{\textsc{bottom}}
\newcommand{\scleft}{\textsc{left}}
\newcommand{\scright}{\textsc{right}}
\newcommand{\scSymbRep}{\textsc{SymbRep}}
\newcommand{\scConfig}{\textsc{Config}}
\newcommand{\interior}[1]{\accentset{\circ}{#1}}
\newcommand{\defn}[1]{\textbf{#1}}

\newcommand\tile[4]{
    \raisebox{-3mm}{
\begin{tikzpicture}[scale=0.9]
\draw (0, 0) -- (1, 0);
\draw (0, 0) -- (0, 1);
\draw (1, 1) -- (1, 0);
\draw (1, 1) -- (0, 1);
\node[rotate=0,font=\footnotesize] at (0.8, 0.5) {#1};
\node[rotate=0,font=\footnotesize] at (0.5, 0.8) {#2};
\node[rotate=0,font=\footnotesize] at (0.2, 0.5) {#3};
\node[rotate=0,font=\footnotesize] at (0.5, 0.2) {#4};
\end{tikzpicture}}}

\def\p{1.61803398874989}   % phi

\title
[Markov partitions, toral $\mathbb{Z}^2$-rotations, Jeandel-Rao Wang shift, model sets]
{Markov partitions for toral $\mathbb{Z}^2$-rotations\\featuring Jeandel-Rao Wang shift and model sets}

\author[S.~Labb\'e]{S\'ebastien Labb\'e}
\address[S.~Labb\'e]{Univ. Bordeaux, CNRS,  Bordeaux INP, LaBRI, UMR 5800, F-33400, Talence, France}
\email{sebastien.labbe@labri.fr}
\date{\today}

\thanks{The author acknowledges financial support from the Laboratoire
International Franco-Québécois de Recherche en Combinatoire (LIRCO), the Agence
Nationale de la Recherche through the project CODYS (ANR-18-CE40-0007) and the
Horizon  2020 European  Research  Infrastructure  project OpenDreamKit
(676541).}

\begin{document}

\begin{abstract}
    We define a partition $\mathcal{P}_0$ and a $\mathbb{Z}^2$-rotation
    ($\mathbb{Z}^2$-action defined by rotations) on a 2-dimensional torus whose
    associated symbolic dynamical
    system is a minimal proper subshift of the Jeandel-Rao aperiodic Wang shift
    defined by 11 Wang tiles.  We define another partition
    $\mathcal{P}_\mathcal{U}$ and a $\mathbb{Z}^2$-rotation on
    $\mathbb{T}^2$ whose associated symbolic dynamical system is equal to a
    minimal and aperiodic Wang shift defined by 19 Wang tiles.  This proves
    that $\mathcal{P}_\mathcal{U}$ is a Markov partition for the
    $\mathbb{Z}^2$-rotation on $\mathbb{T}^2$.  We prove in both
    cases that the toral $\mathbb{Z}^2$-rotation is the maximal
    equicontinuous factor of the minimal subshifts and that the set of fiber
    cardinalities of the factor map is $\{1,2,8\}$.  The two minimal subshifts
    are uniquely ergodic and are isomorphic as measure-preserving dynamical
    systems to the toral $\mathbb{Z}^2$-rotations.  It provides a
    construction of these Wang shifts as model sets of 4-to-2 cut and project
    schemes. A do-it-yourself puzzle is available in the appendix to illustrate
    the results.

    \smallskip
    \noindent
    \textsc{Résumé}.
    Nous définissons une partition $\mathcal{P}_0$ et une
    $\mathbb{Z}^2$-rotation ($\mathbb{Z}^2$-action définie par des rotations)
    sur un tore 2-dimensionnel dont le système dynamique
    symbolique associé est un sous-décalage propre et minimal du sous-décalage
    apériodique de Jeandel-Rao décrit par un ensemble de 11 tuiles de Wang.
    Nous définissons une autre partition $\mathcal{P}_\mathcal{U}$ et une
    $\mathbb{Z}^2$-rotation sur $\mathbb{T}^2$ dont le système
    dynamique symbolique associé est égal au sous-décalage minimal et
    apériodique défini par un ensemble de 19 tuiles de Wang. On montre que
    $\mathcal{P}_\mathcal{U}$ est une partition de Markov pour la
    $\mathbb{Z}^2$-rotation sur $\mathbb{T}^2$. Nous prouvons dans
    les deux cas que la $\mathbb{Z}^2$-rotation sur le tore est le facteur
    équicontinu maximal des sous-décalages minimaux et que l'ensemble des
    cardinalités des fibres du facteur est $\{1,2,8\}$. Les deux
    sous-décalages minimaux sont uniquement ergodiques et sont isomorphes en
    tant que systèmes dynamiques mesurés à la $\mathbb{Z}^2$-rotation sur le
    tore. Les résultats fournissent une construction de ces sous-décalages de
    Wang en tant qu'ensembles modèles par la méthode de coupe et projection 4
    sur 2. Un puzzle à faire soi-même est disponible en annexe pour illustrer
    les résultats.

\end{abstract}

\maketitle

\setcounter{tocdepth}{1}
\tableofcontents

% POUR AFFICHER LES TODOS
% \makeatletter
% \providecommand\@dotsep{5}
% \makeatother
% \listoftodos\relax

\input{common-content.tex}

\part*{References and appendix}

%%%%%%%%%%%%%%%%
% Bibliographie %
%%%%%%%%%%%%%%%%%
%\bibliographystyle{plain} %numeros
%\bibliographystyle{alpha} %author+year
\bibliographystyle{myalpha} %initials for first names + DOI
%{\footnotesize
\bibliography{../biblio}
%}

\newpage
\section*{Appendix -- A DIY Puzzle to illustrate the results}

We encode the 11 Jeandel-Rao tiles into geometrical shapes, see
Figure~\ref{fig:JR-tile-set-geometric}, where each integer color in
$\{0,1,2,3,4\}$ is replaced by an equal number of triangular or circular bumps.
%\begin{center}
%    \includegraphics[width=.9\linewidth]{../Decoupe_laser/T0_shapes.pdf}
%\end{center}

Print one or more copies of this page and cut each of the 25 tiles shown
in Figure~\ref{fig:5x5-JR-pattern} with scissors.
Use the tiles and the \defn{Universal solver for Jeandel-Rao Wang shift} shown in
Figure~\ref{fig:universal-JR-tiling-solver} to construct every pattern seen in
the proper minimal subshift $\Xcal_{\Pcal_0,R_0}\subsetneq\Omega_0$ 
of the Jeandel-Rao Wang shift.

\begin{figure}[h]
\begin{center}
\input{solution5x5_61725_scale3.tex}
\end{center}
    \caption{A $5\times 5$ pattern with Jeandel-Rao tiles ready to laser cut.
    Tiles should have 3cm size when printed in A4 format.}
    \label{fig:5x5-JR-pattern}
\end{figure}

\begin{figure}
\begin{center}
\input{JR_partition_appendix.tex}
\end{center}
    \caption{The Universal solver for Jeandel-Rao Wang shift.
    Any pattern in the minimal subshift of Jeandel-Rao Wang shift
    is the coding of the orbit of some starting point 
    % in the rectangular fundamental domain 
    by the action of horizontal and
    vertical translations by $1$ unit (3cm when printed in A4 format).}
    \label{fig:universal-JR-tiling-solver}
\end{figure}

\end{document}

%% file: common-content.tex
\section{Introduction}

% sage: n((2-golden_ratio))
% 0.381966011250105
% sage: n((golden_ratio))
% 1.61803398874989
% sage: n(golden_ratio/2/pi)
% 0.257518107400242
% sage: n(golden_ratio^2/2/pi)
% 0.416673050492137

\tikzstyle{B-beads}=[circle,draw=blue,fill=blue,inner sep=0pt,minimum size=6pt]
\tikzstyle{R-beads}=[diamond,very thick,draw=red,fill=red!5,inner sep=0pt,minimum size=7pt]

%\subsection{Codings of irrational rotations on a torus}

We build a biinfinite necklace by placing 
beads at integer positions on the real line:
\begin{center}
\begin{tikzpicture}[scale=2.8]
    \draw (-2.15,0) -- (3.15,0);
    \foreach \a in {-2,...,3}
    \node[circle,draw=black,fill=white,inner sep=0pt,minimum size=6pt,
          label=above:{$\a$}] at (\a,0) {};
\end{tikzpicture}
\end{center}
Beads come in two colors: {\bf\color{red}light red}
\begin{tikzpicture}
\node[R-beads] (a) at (0,0) {};
\end{tikzpicture}
and {\bf\color{blue}dark blue}
\begin{tikzpicture}
\node[B-beads] (a) at (0,0) {};
\end{tikzpicture}.
Given $\alpha>0$,
we would like to place the colored beads 
in
such a way that the relative frequency 
\[
    \frac{\text{number of blue beads in } \{-n,-n+1,\dots,n\}}
         {\text{number of red beads in } \{-n,-n+1,\dots,n\}}
\]
converges to $\alpha$ as $n$ goes to infinity.

A well-known approach is to use coding of rotations on a circle
of circumference $\alpha+1$
whose radius is $\frac{1}{2\pi}(\alpha+1)$.
The coding is given by the
partition of the circle $\R/(\alpha+1)\Z$ into one arc of 
length $\alpha$ associated with dark blue beads
and another arc of 
length~$1$ associated with light red beads.
The two end points of the arcs are associated with red and blue beads respectively
in one way or the other.
Then, we wrap the biinfinite necklace around the circle and each bead is given
the color according to which of the two arcs it falls in.
For example, when $\alpha=\frac{1+\sqrt{5}}{2}$
and if the zero position is assigned to one of the end points,
we get the picture below:
\begin{center}
\begin{tikzpicture}[scale=2.8]
    \def\phi{1.61803398874989}
    \def\angle{360*.381966}
    \def\radius{ 0.416673050492137}

    \draw[very thick,->] (0:2*\radius) arc (0:20:2*\radius);

    %interieur des secteurs
    \fill[blue!20] (0,0) -- (\radius,0) arc (0:360-\angle:\radius) ;
    \fill[red!20]  (0,0) -- (\radius,0) arc (0:-\angle:\radius) ;

    %exterieur des secteurs
    \draw[thick,blue]       (\radius,0) arc (0:360-\angle:\radius) ;
    \draw[thick,red]        (\radius,0) arc (0:-\angle:\radius) ;

    %deux separations des secteurs
    \draw[thick,blue] (0,0) -- (\radius,0);
    \draw[thick,red]  (0,0) -- (-\angle:\radius) ;

    % label dans les secteurs
    \node[red] at (-\angle*.5:\radius*.4) {\textbf{red}};
    \node[blue] at (180-\angle*.5:\radius*.4) {\textbf{blue}};

    \foreach \a in {-2,...,8}
    \node at (\a*\angle:\radius*1.3) {$\a$};
    \foreach \a in {-2,0,1,3,4,6,8}
    \node[B-beads] (a) at (\a*\angle:\radius) {};
    \foreach \a in {-1,2,5,7}
    \node[R-beads] (a) at (\a*\angle:\radius) {};
\end{tikzpicture}
\end{center}
Then, we unwrap the biinfinite necklace and we
get an assignment of colored beads to each integer position
such that the relative frequency between blue and red beads is $\alpha$.
Here is what we get after zooming out a little:
\begin{center}
\begin{tikzpicture}[scale=1]
    \draw (-2.3,0) -- (8.3,0);
    % labels
    \foreach \a in {-2,...,8}
    \node[above=3mm] at (\a,0) {$\a$};
    % blue beads
    \foreach \a in {-2,0,1,3,4,6,8}
    \node[B-beads] (a) at (\a,0) {};
    % red beads
    \foreach \a in {-1,2,5,7}
    \node[R-beads] (a) at (\a,0) {};
\end{tikzpicture}
\end{center}
\newcommand\patternI[2]{
\begin{tikzpicture}[scale=#1]
    \draw (-.5,0) -- (0.5,0);
    \node[#2-beads] at (0,0) {};
\end{tikzpicture}}
\newcommand\patternII[3]{
\begin{tikzpicture}[scale=#1]
    \draw (-.5,0) -- (1.5,0);
    \node[#2-beads] at (0,0) {};
    \node[#3-beads] at (1,0) {};
\end{tikzpicture}}
\newcommand\patternIII[4]{
\begin{tikzpicture}[scale=#1]
    \draw (-.5,0) -- (2.5,0);
    \node[#2-beads] at (0,0) {};
    \node[#3-beads] at (1,0) {};
    \node[#4-beads] at (2,0) {};
\end{tikzpicture}}
We observe that this colored necklace has very few distinct patterns.
The patterns of size 0, 1, 2 and 3 that we see in the necklace are shown in the
table below:
\begin{center}
\begin{tabular}{c|c|c|c}
0 & 1 & 2 & 3\\
\hline
\begin{tikzpicture}
    \node[B-beads,white] at (0,0) {};
    \draw (-.15,0) -- (0.15,0);
\end{tikzpicture}
&\patternI{1}{B}  & \patternII{1}{B}{R} &\patternIII{1}{B}{R}{B}\\
&\patternI{1}{R}  & \patternII{1}{R}{B} &\patternIII{1}{R}{B}{B}\\
&                 & \patternII{1}{B}{B} &\patternIII{1}{B}{B}{R}\\
&                 &                     &\patternIII{1}{R}{B}{R}\\
\end{tabular}
\end{center}
We do not get other patterns of size 1, 2 or 3 in the whole biinfinite necklace
since every pattern is uniquely determined by the position of its first bead on
the circle. For each $n\in\N$ there exists a partition of the circle
according to the pattern associated with the position of its first bead:
\begin{center}
\begin{tabular}{ccc}
\begin{tikzpicture}[scale=3]
    \def\phi{1.61803398874989}
    \def\angle{360*.381966}
    \def\radius{ 0.416673050492137}

    \draw[very thick,fill=black!5] (0,0) circle (\radius);

    %deux separations des secteurs
    \draw[very thick] (0,0) -- (\radius,0);
    \draw[very thick] (0,0) -- (-\angle:\radius) ;

    % label dans les secteurs
    \node[red] at (-\angle*.5:\radius*.5) {\patternI{1}{R}};
    \node[blue] at (180-\angle*.5:\radius*.5) {\patternI{1}{B}};
\end{tikzpicture}
&
\begin{tikzpicture}[scale=3]
    \def\phi{1.61803398874989}
    \def\angle{360*.381966}
    \def\radius{ 0.416673050492137}

    \draw[very thick,fill=black!5] (0,0) circle (\radius);

    % deux separations des secteurs
    \draw[very thick] (0,0) -- (\radius,0);
    \draw[very thick] (0,0) -- (-\angle:\radius) ;
    \draw[very thick] (0,0) -- (-2*\angle:\radius) ;

    % label dans les secteurs
    \node at (-\angle*.5:\radius*.5)             {\patternII{.4}{R}{B}};
    \node at (-\angle*1.5:\radius*.5)            {\patternII{.4}{B}{R}};
    \node[xshift=1mm] at (180-\angle:\radius*.5) {\patternII{.4}{B}{B}};
\end{tikzpicture}
&
\begin{tikzpicture}[scale=3]
    \def\phi{1.61803398874989}
    \def\angle{360*.381966}
    \def\radius{ 0.416673050492137}

    \draw[very thick,fill=black!5] (0,0) circle (\radius);

    % deux separations des secteurs
    \draw[very thick] (0,0) -- (0:\radius);
    \draw[very thick] (0,0) -- (-\angle:\radius) ;
    \draw[very thick] (0,0) -- (-2*\angle:\radius) ;
    \draw[very thick] (0,0) -- (-3*\angle:\radius) ;

    % label dans les secteurs
    \node at (-180-2*\angle:\radius*.5)  {\patternIII{.3}{R}{B}{B}};
    \node at (-\angle*1.5:\radius*.5)    {\patternIII{.3}{B}{R}{B}};
    \node[xshift=1mm] at (180-\angle:\radius*.5)     {\patternIII{.3}{B}{B}{R}};
    \node[xshift=1.5mm] at (180-1.5*\angle:\radius*.5) {\patternIII{.3}{R}{B}{R}};
\end{tikzpicture}
\end{tabular}
\end{center}
When $\alpha$ is irrational, one can prove that the partition of the circle for
patterns of size $n$ is made of $n+1$ parts.
The proof follows from the fact that 
the distance between two consecutive beads on the necklace is equal to the
length of one of the original arc (here, the red arc of length 1).
So the partition at a given level is obtained from the previous one by adding 
exactly one separation which increases the number of patterns by 1.
This shows that the colored necklace is a
Sturmian sequence, that is, a sequence whose pattern complexity is $n+1$,
see \cite{MR1905123}.
When $\alpha=\frac{1+\sqrt{5}}{2}$, this is a construction of the biinfinite
Fibonacci word \cite{berstel_mots_1980}.
Note that it is known that sequences having strictly less than $n+1$ patterns
of length $n$, for some $n\in\N$, are eventually periodic 
\cite{MR1507944}.
Therefore, Sturmian sequences are the simplest aperiodic sequences in terms
of pattern (or factor) complexity \cite{MR2759107}.

What Coven and Hedlund proved in \cite{MR0322838}
based on the initial work 
of Morse and Hedlund \cite{MR0000745} 
on Sturmian sequences 
dating from 1940 
is that a biinfinite sequence 
is Sturmian if and only if it is the coding of an irrational
rotation.
Proving that the coding of an irrational rotation is a Sturmian sequence is
the easy part and corresponds to what we did above.
The difficult part is to prove that a Sturmian sequence can be obtained as the coding of an irrational rotation for some starting point.
The proof is
explained nowadays in terms of $S$-adic development of Sturmian sequences,
Rauzy induction of circle rotations,
the continued fraction expansion of real numbers 
and the Ostrowski numeration system \cite{MR1970385}.
Rauzy discovered that the connexion between Sturmian sequences 
and rotations can be generalized to sequences using three symbols \cite{MR667748}
involving a rotation on a 2-dimensional torus $\torus$.
This result was extended recently for almost all rotations on $\torus$
\cite{MR3986918}, see also \cite{thuswaldner_S-adic_2019}.

%\subsection{Symbolic representation of dynamical systems}
\subsection{From biinfinite necklaces to $2$-dimensional configurations}

In this work, we want to extend the behavior of Sturmian sequences beyond the
1-dimensional case by considering 
$d$-dimensional configurations.
We say that a \emph{configuration} is an assignment of colored beads
from a finite set $\A$ to every coordinate of the lattice $\Z^d$.
Are there rules describing how to place colored beads in a configuration
in such a way that it encodes
rotations on a higher dimensional torus?
% Does there exist a combinatorial characterization of the configurations $\Z^d\to\A$
% which are symbolic representations of rotations on a higher dimensional torus?

\begin{center}
\begin{tikzpicture}[scale=1]
    %\draw (-2.3,0) -- (3.3,0);
    % grid lines
    \foreach \a in {-2,...,3}
    \draw (\a,-1.4) -- (\a,2.4);
    \foreach \b in {-1,...,2}
    \draw (-2.4,\b) -- (3.4,\b);
    % nodes
    \foreach \a in {-2,...,3}
    \foreach \b in {-1,...,2}
    \node[circle,draw=black,fill=white,inner sep=0pt,minimum size=6pt] 
    at (\a,\b) {};
    %\node[above] at (0,0) {$(0,0)$};
    %\node[above] at (1,0) {$(1,0)$};
    %\node[above] at (0,1) {$(0,1)$};
    %\node[R-beads] at (0,0) {};
    %\node[B-beads] at (1,0) {};
%label=above:{$(\a,\b)$}
\end{tikzpicture}
\end{center}
This is related to a question of Adler:
``\textit{how and to what extent can a dynamical
system be represented by a symbolic one}'' \cite{MR1477538}.
The kind of dynamical system we consider are
toral $\Z^d$-rotations, that is, $\Z^d$-actions by rotations on a torus.
When $d=1$, the answer is given in terms of Sturmian sequences and factor complexity.
While Berthé and Vuillon \cite{MR1782038} considered the coding of
$\Z^2$-rotations on the 1-dimensional torus,
we consider $\Z^d$-rotations
on the $d$-dimensional torus.
We show that an answer to the question when $d=2$ can be made in terms
of sets of configurations avoiding a finite set of forbidden patterns known as
\emph{subshifts of finite type} and more precisely in terms of aperiodic
tilings by Wang tiles.
This contrasts with the one-dimensional case, since 
Sturmian sequences can not be described by a finite set of forbidden patterns
(a one-dimensional shift of finite type is nonempty if and only
if it has a periodic point \cite[\S 13.10]{MR1369092}).

\subsection{Jeandel-Rao's aperiodic set of 11 Wang tiles}

The study of aperiodic order \cite{MR857454,MR3136260} gained a lot of interest
since the discovery in 1982 of quasicrystals by Shechtman \cite{PhysRevLett.53.1951}
for which he was awarded the Nobel Prize in Chemistry in 2011.
The first known aperiodic structure was based on the notion of Wang tiles.
\emph{Wang tiles} 
can be represented as
unit square with colored edges,
see Figure~\ref{fig:JR-tile-set}.
\begin{figure}[h]
\begin{center}
    \includegraphics[scale=.9]{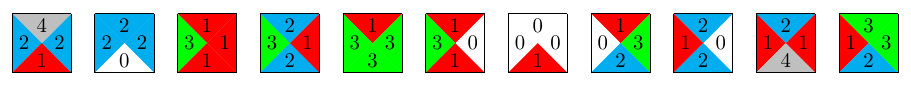}
\end{center}
    \caption{The aperiodic set $\T_0$ of 11 Wang tiles 
    discovered by Jeandel and Rao in 2015
    \cite{jeandel_aperiodic_2015}.
    }
    \label{fig:JR-tile-set}
\end{figure}
Given a finite set
of Wang tiles $\T$, we consider \emph{tilings} of the Euclidean plane using
arbitrarily many translated (but not rotated) copies of the tiles in $\T$. Tiles are placed on the integer lattice
points of the plane with their edges oriented horizontally and vertically.
The tiling is \emph{valid} if every pair of
contiguous edges have the same color.
Deciding if a set of Wang tiles admits a valid tiling of the plane is a difficult
question known as the \emph{domino problem}.
Answering a question of Wang \cite{wang_proving_1961},
Berger proved that the domino problem is undecidable \cite{MR0216954} using a
reduction to the halting problem of Turing machines.
As noticed by Wang, if every set of Wang tiles that admits a valid tiling of the plane would also
admit a periodic tiling, then the domino problem would be decidable.
As a consequence, there exist aperiodic sets of Wang tiles.
A set $\T$ of Wang tiles is called \emph{aperiodic}
if there exists a valid tiling of the plane with the tiles from $\T$
and none of the valid tilings of the plane with the tiles from $\T$ is
invariant under a nonzero translation.

Berger constructed an aperiodic set made of
20426 Wang tiles \cite{MR0216954}, later reduced to 104 
by himself \cite{MR2939561} and further reduced by others \cite{MR0286317,MR0297572}.
Small aperiodic sets of Wang tiles
include Ammann's 16 tiles \cite[p.~595]{MR857454}, Kari's 14 tiles \cite{MR1417578}  and Culik's 13 tiles \cite{MR1417576}.  
The search for the smallest aperiodic set of Wang tiles continued until
Jeandel and Rao proved 
the existence of an aperiodic set $\T_0$ of 11 Wang tiles
and that no set of Wang tiles of cardinality $\leq10$ is aperiodic
\cite{jeandel_aperiodic_2015}. 
Thus their set, shown in Figure~\ref{fig:JR-tile-set},
is a smallest possible set of aperiodic Wang tiles.
An equivalent geometric representation of their
set of 11 tiles is shown in Figure~\ref{fig:JR-tile-set-geometric}.

\begin{figure}[h]
\begin{center}
    \includegraphics[scale=.9]{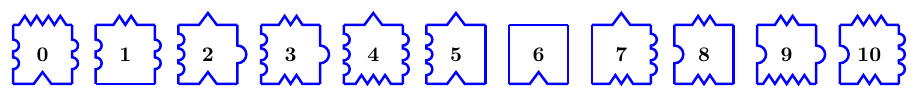}
\end{center}
    \caption{
    Jeandel-Rao tiles can be encoded into a set of equivalent geometrical
    shapes in the sense that every tiling using Jeandel-Rao tiles can be
    transformed into a unique tiling with the corresponding geometrical shapes
    and vice versa.
    }
    \label{fig:JR-tile-set-geometric}
\end{figure}

% see also last section of \cite{MR3535143} who uses Z^d rotation on 1 circle

The aperiodicity of the Jeandel-Rao set of 11 Wang tiles follows from 
the decomposition of tilings as horizontal strips of height 4 or 5.
Using the representation of Wang tiles by transducers, Jeandel and Rao proved
that the language of sequences describing the heights of consecutive horizontal strips in the
decomposition is exactly the language of the Fibonacci word on the alphabet
$\{4,5\}$. Thus it contains the same patterns as in the necklace we constructed above
where $5$ corresponds to the dark blue bead
\begin{tikzpicture}
\node[B-beads] (a) at (0,0) {};
\end{tikzpicture}
and $4$ corresponds to the light red bead
\begin{tikzpicture}
\node[R-beads] (a) at (0,0) {};
\end{tikzpicture}.
This proves the absence of any vertical period in every tiling with Jeandel-Rao tiles.
This is enough to conclude
aperiodicity in all directions, see \cite[Prop.~5.9]{MR3136260}. 
The presence of the Fibonacci word in the vertical structure of Jeandel-Rao tilings is a
first hint that Jeandel-Rao tilings are related to irrational rotations on a
torus.

\subsection{Results}

In this article, we consider Wang tilings from the point of view of symbolic
dynamics \cite{MR2078847}.
While a tiling by a set of Wang tiles $\T$ is a tiling of the plane $\R^2$
whose validity is preserved by translations of $\R^2$ (leading to the notion of \emph{hull}, see
\cite{MR3136260}), we prefer to consider maps $\Z^2\to\T$, that we call
\emph{configurations}, whose validity is preserved by
translations of $\Z^2$. The set $\Omega_\T$ 
of all valid configurations $\Z^2\to\T$
is called a \emph{Wang shift} as
it is closed under the shift $\sigma$ by integer translates.
The passage from Wang shifts ($\Z^2$-actions) to Wang tiling dynamical systems ($\R^2$-action)
can be made with the 2-dimensional suspension of the former as in the classical construction of a
``flow under a function'' in Ergodic Theory, see \cite{MR1355301}.

\begin{figure}[h]
\begin{center}
% trim={<left> <lower> <right> <upper>}
% example: [trim={0cm 0cm 8cm 0cm},clip]
    \includegraphics[width=0.99\linewidth,trim={2cm 0cm 0cm 0cm},clip]{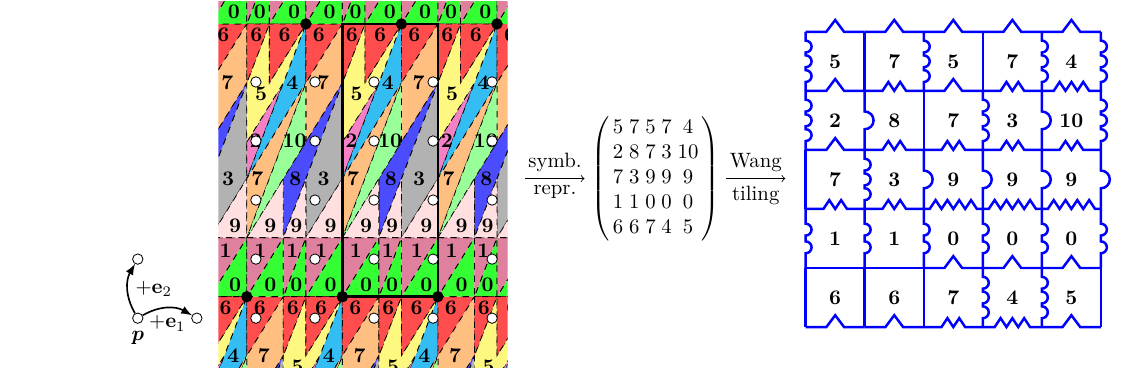}
\end{center}
    \caption{
        On the left, we illustrate the lattice $\Gamma_0=\langle(\varphi,0), (1,\varphi+3)\rangle_\Z$,
        where $\varphi=\frac{1+\sqrt{5}}{2}$,
        with black vertices,
    a rectangular fundamental domain of the flat torus $\R^2/\Gamma_0$ 
    with a black contour and 
    a polygonal partition $\Pcal_0$ of $\R^2/\Gamma_0$ 
    with indices in the set $\{0,1,\dots,10\}$.
    We show that for every starting point $\bp\in\R^2$, the coding
    of the shifted lattice $\bp+\Z^2$ under the polygonal partition 
    yields a configuration $w:\Z^2\to\{0,1,\dots,10\}$
    which is 
    a symbolic representation of $\bp$.
    The configuration $w$ corresponds to a valid tiling of the plane with
    Jeandel-Rao's set of 11 Wang tiles.}
    \label{fig:2d-walk}
\end{figure}

We may now state the main results of this article
together with an illustration.
A partition $\Pcal_0$ of the plane into well-chosen polygons
indexed by integers from the set $\{0,1,2,\dots,10\}$
is shown in Figure~\ref{fig:2d-walk} (left).
The partition $\Pcal_0$
is invariant under the group of translations $\Gamma_0=\langle(\varphi,0),
(1,\varphi+3)\rangle_{\Z}$ where $\varphi=\frac{1+\sqrt{5}}{2}$.
Equivalently, it is a partition of the torus $\R^2/\Gamma_0$ given by a
partition of the rectangular fundamental domain
$[0,\varphi)\times[0,\varphi+3)$.
On the torus $\R^2/\Gamma_0$, we consider
the continuous $\Z^2$-action defined by
$R_0^\bn(\bx):=R_0(\bn,\bx)=\bx + \bn$
for every $\bn=(n_1,n_2)\in\Z^2$
which defines a dynamical system that we denote
$(\R^2/\Gamma_0, \Z^2, R_0)$.
The symbolic dynamical system $\Xcal_{\Pcal_0,R_0}$ corresponding
to $\Pcal_0,R_0$ is
the topological closure of the set of all configurations
$w\in\{0,1,\dots,10\}^{\Z^2}$ obtained from the coding by the partition
$\Pcal_0$ of the orbit of some starting point in $\R^2/\Gamma_0$ by the
$\Z^2$-action of $R_0$ (see Lemma~\ref{lem:closure-of-tilings}).
We say that $\Xcal_{\Pcal_0,R_0}$ is a \emph{subshift} as it is also closed
under the shift $\sigma$ by integer translations.
We state the first theorem below.
The fact that $\Xcal_{\Pcal_0,R_0}\subset\Omega_0$
is illustrated in Figure~\ref{fig:2d-walk}
where $\Omega_0\subset\{0,1,\dots,10\}^{\Z^2}$ is the Jeandel-Rao Wang shift.
The definitions of the terms used in the theorem can be found in
Section~\ref{sec:preliminaries} and Section~\ref{sec:symbolic-representation}.

\begin{theorem}\label{thm:JR-max-equi-factor}
    The Jeandel-Rao Wang shift $\Omega_0$ has the following properties:
\begin{enumerate}[\rm (i)]
\item $\Xcal_{\Pcal_0,R_0}\subsetneq\Omega_0$
is a proper minimal and aperiodic subshift of $\Omega_0$,
\item the partition $\Pcal_0$ gives a symbolic representation of 
$(\R^2/\Gamma_0,\Z^2,R_0)$,
\item the dynamical system $(\R^2/\Gamma_0,\Z^2,R_0)$ is the maximal
    equicontinuous factor of $(\Xcal_{\Pcal_0,R_0},\Z^2,\sigma)$,
\item the set of fiber cardinalities of the factor map
$\Xcal_{\Pcal_0,R_0}\to\R^2/\Gamma_0$ is $\{1,2,8\}$,
\item
    the dynamical system $(\Xcal_{\Pcal_0,R_0},\Z^2,\sigma)$ is strictly
    ergodic and
    the measure-preserving dynamical system $(\Xcal_{\Pcal_0,R_0},\Z^2,\sigma,\nu)$
    is isomorphic 
    to $(\R^2/\Gamma_0,\Z^2,R_0,\lambda)$ 
    where $\nu$ is the unique shift-invariant probability measure on
    $\Xcal_{\Pcal_0,R_0}$
    and $\lambda$ is the Haar measure on $\R^2/\Gamma_0$.
\end{enumerate}
\end{theorem}

A larger picture of the partition $\Pcal_0$ is illustrated in the appendix
together with a DIY puzzle
allowing hand made construction of configurations in
$\Xcal_{\Pcal_0,R_0}\subset\Omega_0$
as the symbolic representation 
of starting points in $\R^2/\Gamma_0$.

Theorem~\ref{thm:JR-max-equi-factor}
corresponds to the easy direction in the proof of Morse-Hedlund's
theorem, namely that codings of irrational rotations have pattern complexity $n+1$.
Proving the converse, i.e., that almost every (for some shift-invariant
probability measure) configuration in the Jeandel-Rao Wang shift
is obtained as the coding of the shifted lattice $\bp+\Z^2$ for some unique %point
$\bp\in\R^2/\Gamma_0$ is harder.
This has lead to split the proof of the
converse \cite{labbe_induction_2019}.

Note that a similar result was obtained for Penrose tilings
\cite[Theorem A]{MR1355301}.  
In particular, it was shown that the set of fiber cardinalities for Penrose
tilings (with the action of $\R^2$) is $\{1,2,10\}$.
In \cite{MR3041974}, it was proved that the set of fiber cardinalities
is $\{1,2,6,12\}$ for a minimal hull among Taylor-Socolar hexagonal tilings.
We show 
in Lemma~\ref{lem:fiber-card-invariant}
that the set of fiber cardinalities of the maximal equicontinuous factor of a
minimal dynamical system is invariant under topological conjugacy.
Therefore, the Jeandel-Rao tilings, the Penrose tilings and the Taylor-Socolar tilings
are inherently different.

We also provide a stronger result on another example. We define a polygonal
partition $\Pcal_\U$ of the torus $\torus=\R^2/\Z^2$ into 19 atoms.
We consider
the continuous $\Z^2$-action $R_\U$ defined
on $\torus$ by
$R_\U^\bn(\bx)=\bx + \varphi^{-2}\bn$
for every $\bn\in\Z^2$ where $\varphi=\frac{1+\sqrt{5}}{2}$.
It defines a dynamical system that we denote
$(\torus, \Z^2, R_\U)$.
We prove that the symbolic dynamical system $\Xcal_{\Pcal_\U,R_\U}$
corresponding to $\Pcal_\U,R_\U$ is \emph{equal} to the Wang shift
$\Omega_\U$
where $\U$ is the set of 19
Wang tiles introduced by the author in \cite{MR3978536}
and discovered from the study of the Jeandel-Rao
Wang shift \cite{labbe_substitutive_2018_with_doi}.

\begin{theorem}\label{thm:OmegaU-partition}
    The Wang shift $\Omega_\U$ has the following properties:
\begin{enumerate}[\rm (i)]
\item the subshift $\Xcal_{\Pcal_\U,R_\U}$ is minimal, aperiodic and is equal to $\Omega_\U$,
% \item The partition $\Pcal_\U$ gives a symbolic representation
% of $(\torus,\Z^2,R_\U)$. 
\item $\Pcal_\U$ is a Markov partition for the dynamical system $(\torus,\Z^2,R_\U)$,
\item $(\torus,\Z^2,R_\U)$ is the maximal equicontinuous
factor of $(\Omega_\U,\Z^2,\sigma)$,
\item the set of fiber cardinalities of the factor map
$\Omega_\U\to\torus$ is $\{1,2,8\}$,
\item the dynamical system $(\Omega_\U,\Z^2,\sigma)$ is strictly
    ergodic and
    the measure-preserving dynamical system $(\Omega_\U,\Z^2,\sigma,\nu)$
    is isomorphic 
    to $(\torus,\Z^2,R_\U,\lambda)$ 
    where $\nu$ is the unique shift-invariant probability measure on
    $\Omega_\U$ and $\lambda$ is the Haar measure on $\torus$.
\end{enumerate}
\end{theorem}

Since a Wang shift is a shift of finite type, 
the equality $\Xcal_{\Pcal_\U,R_\U}=\Omega_\U$ implies that
$\Pcal_\U$ is a Markov partition (see Definition~\ref{def:Markov}) for the
$\Z^2$-action $R_\U$.
Note that Markov partitions ``\textit{remained abstract objects for a long time}''
\cite[\S 7.1]{MR1970385}.
Explicit constructions of Markov partitions were originally given for
hyperbolic automorphisms of the torus, see
\cite{MR0257315,MR1477538}.
More recent references relate Markov partitions with 
arithmetics \cite{MR1619562,Kenyon-HDR},
algebraic numbers 
\cite{MR2806687} and numeration systems \cite{MR1615950}.

The link between aperiodic order and cut and project schemes (Definition~\ref{def:CPS}) 
and model sets (Definition~\ref{def:model-set}) is not new.
In one dimension,
the fact that Sturmian sequences 
are codings of rotations
implies that
they can be seen as model
sets of cut and project schemes, see
\cite{MR2192233,MR3136260}.
Since the contribution of N.~G.~de~Bruijn \cite{MR609465}, we know
that Penrose tilings are obtained as the projection of discrete surfaces in a 5-dimensional space
onto a 2-dimensional plane.
Other typical examples include Ammann-Beenker tilings 
\cite{bedaride_ammannbeenker_2013}
and Taylor-Socolar
aperiodic hexagonal tilings for which Lee and Moody gave a description in
terms of model sets \cite{MR3041974}.
Likewise, a consequence of 
Theorem~\ref{thm:JR-max-equi-factor} and
Theorem~\ref{thm:OmegaU-partition}
is a description of the two aperiodic Wang
shifts $\Omega_0$ and $\Omega_\U$ with cut and project schemes.
More precisely, 
we show that the occurrences of patterns in the two Wang shifts
are regular model sets.
Definitions of generic and singular configurations is in Section~\ref{sec:one-to-one-map}
and definitions of regular, generic and singular models sets can be found in
Section~\ref{sec:CPS}.

\begin{theorem}\label{thm:jeandel-rao-model-set}
    % Let $\Xcal_{\Pcal_0,R_0}$ be the proper minimal subshift of Jeandel-Rao
    % Wang shift. % which was described in \cite{labbe_substitutive_2018_with_doi}.
    There exists a cut and project scheme such that
    for every Jeandel-Rao configuration $w\in
    \Xcal_{\Pcal_0,R_0}\subsetneq\Omega_0$, the set $Q\subseteq\Z^2$ of
    occurrences of a pattern in $w$
    is a regular model set.
    If $w$ is a generic (resp.~singular) configuration,
    then $Q$ is a generic (resp.~singular) model set.
    %If $w$ is generic, then $Q=\curlywedge(A)$ with acceptance set
    %$A=f_0([p])$.
\end{theorem}

We prove the same result for the Wang shift $\Omega_\U$
(see Theorem~\ref{thm:OmegaU-is-model-set}).
As opposed to the Kari-Culik Wang shift,
for which a minimal subsystem
is related to a dynamical system on
$p$-adic numbers \cite{MR3668002},
windows used for the cut and project schemes are Euclidean.

% \todo[inline]{question?:
% Cite Robinson Junior presentation
% \url{https://www.irif.fr/~steiner/num09/robinson.pdf} about Kari Culik
% and conjugacy to torus rotations.}

It was shown that the action of $\R^2$ by translation on the set of Penrose tilings
is an almost one-to-one extension of a minimal $\R^2$-action by rotations on $\torusI^4$
\cite{MR1355301} (the fact that it is $\torusI^4$ instead of $\torusI^2$ is
related to the consideration of tilings instead of shifts).
This result can also be seen as a higher dimensional generalization
of the Sturmian dynamical systems. 
Note that a shift of finite type or Wang shift can be explicitely constructed
from the Penrose tiling dynamical system, as shown in \cite{MR1971208}.
This calls for a common point of view
including Jeandel-Rao aperiodic tilings,
Penrose tilings and others. For example, we do not know if Penrose tilings
can be seen as a symbolic dynamical system associated to a Markov partition
like it is the case for the Jeandel-Rao Wang shift.
It is possible that such Markov partitions exist
only for tilings associated to some algebraic numbers, see
\cite{bedaride_canonical_2020}.

%Lattice substitution systems and model sets: \cite{MR1811757}

\subsection{Structure of the article}
This article is divided into three parts. 
In the first part, we 
construct symbolic representations of toral $\Z^2$-rotations and a factor map which provides an isomorphism between symbolic
dynamical systems and toral $\Z^2$-rotations.
In the second part, we
construct sets of Wang tiles and Wang shifts as the coding
of $\Z^2$-rotations on the 2-torus.
We illustrate the method on two examples including Jeandel-Rao aperiodic Wang
shift.
In the third part, we express occurrences of patterns in these Wang shifts in terms of
model sets of cut and project schemes.
In the appendix, we propose a do-it-yourself puzzle to explain the
construction of valid configurations in the Jeandel-Rao Wang shift as the coding
of $\Z^2$-rotations on the 2-torus.

\subsection*{Acknowledgements}

I am thankful to Jarkko Kari and Michaël Rao for their presentations at
the meeting \emph{Combinatorics on words and tilings} (CRM, Montréal, April
2017). I am grateful to Michaël Rao who provided me a text file of
$986\times2583$ characters in the alphabet $\{\texttt{a},\texttt{b}\}$
describing a rectangular pattern with Jeandel-Rao tiles.  This allowed me to
start working on aperiodic tilings.

%The fact coding of rotations on $\torusI$ was used \emph{horizontally} to prove
%of existence of Kari-Culik's aperiodic tilings and the remark made by Rao
%on its last slide about the frequency of tiles $t_0$ and $t_1$ in Jeandel-Rao
%tilings evolving \emph{vertically} like a rotation on $\torusI$ by the golden
%mean was a first hint toward the existence of double rotations on the torus.

I want to thank
Michael Baake,
Sebasti\'an Barbieri,
Vincent Delecroix,
Franz G\"ahler,
Ilya Galanov,
Maik Gröger,
Jeong-Yup Lee,
Jean-François Marckert,
Samuel Petite
and
Mathieu Sablik
for profitable discussions during the preparation of this article which
allowed me to improve my knowledge on tilings, measure theory and dynamical
systems and write the results in terms of existing concepts.
I am very thankful to the anonymous referees for their in-depth reading 
and valuable comments which lead to a great improvement of the presentation.

% MOVED to \thanks{}
% I acknowledge financial support from the Laboratoire International
% Franco-Québécois de Recherche en Combinatoire (LIRCO), the Agence Nationale de
% la Recherche through the project CODYS (ANR-18-CE40-0007) and the Horizon  2020
% European  Research  Infrastructure  project OpenDreamKit (676541).

%\clearpage
\part{Symbolic dynamics of toral $\Z^2$-rotations}\label{part:1}

This part is divided into 5 sections.
After introducing dynamical systems and subshifts,
we define the symbolic representations of toral $\Z^2$-rotations from a
topological partition of the 2-torus.
We introduce a one-to-one map from the 2-torus to symbolic representations
and a factor map from symbolic representations to the 2-torus.
We show that the factor map
provides an isomorphism between symbolic dynamical systems and toral
$\Z^2$-rotations.

\section{Dynamical systems, maximal equicontinuous factors and subshifts}
\label{sec:preliminaries}

% See also how Fuhrmann and Gröger introduce these notions in
% http://arxiv.org/abs/1812.10219
% http://arxiv.org/abs/1812.10789
% https://doi.org/10.1088%2F0951-7715%2F29%2F2%2F528

In this section, we introduce dynamical systems, maximal equicontinuous
factors, set of fiber cardinalities of a factor map, subshifts and shifts of
finite type.
We let $\Z=\{\dots,-1,0,1,2,\dots\}$ denote the integers and
$\N=\{0,1,2,\dots\}$ be the nonnegative integers.

\subsection{Topological dynamical systems}

%https://www.encyclopediaofmath.org/index.php/Topological_dynamical_system

Most of the notions introduced here can be found in \cite{MR648108}.
A \defn{dynamical system} is
a triple $(X,G,T)$, where $X$ is a topological space, $G$ is a topological
group and $T$ is a continuous function $G\times X\to X$ defining a left action
of $G$ on $X$:
if $x\in X$, $e$ is the identity element of $G$ and $g,h\in G$, then using
additive notation for the operation in $G$ we have $T(e,x)=x$
and $T(g+h,x)=T(g,T(h,x))$.
In other words, if one denotes the transformation $x\mapsto T(g,x)$
by $T^g$, then $T^{g+h}=T^g T^h$.
In this work, we consider the Abelian group $G=\Z\times\Z$.

If $Y\subset X$, let $\overline{Y}$ denote the topological closure of $Y$ and
let $T(Y):=\cup_{g\in G}T^g(Y)$ denote the $T$-closure of $Y$.
A subset $Y\subset X$ is \defn{$T$-invariant} if $T(Y)=Y$.
A dynamical system $(X,G,T)$ is called \defn{minimal} if $X$ does
not contain any nonempty, proper, closed $T$-invariant subset.
The left action of $G$ on $X$ is \defn{free}
if $g=e$ whenever there exists $x\in X$ such that $T^g(x)=x$.

%We adapt the definition of conjugacy of dynamical systems from
%\cite[p. 185]{MR1369092} to subshifts.

Let $(X,G,T)$ and $(Y,G,S)$ be two dynamical systems with
the same topological group $G$.
A \defn{homomorphism} $\theta:(X,G,T)\to(Y,G,S)$ is a continuous
function $\theta:X\to Y$ satisfying the commuting property
that $T^g\circ\theta=\theta\circ S^g$ for every $g\in G$.
A homomorphism $\theta:(X,G,T)\to(Y,G,S)$ is called an \defn{embedding}
if it is one-to-one, a \defn{factor map} if it is onto, and a \defn{topological
conjugacy} if it is both one-to-one and onto and its inverse map is continuous.
If $\theta:(X,G,T)\to(Y,G,S)$ is a factor map,
then $(Y,G,S)$ is called a \defn{factor} of $(X,G,T)$
and $(X,G,T)$ is called an \defn{extension} of $(Y,G,S)$.
Two subshifts are \defn{topologically conjugate} if there is a topological
conjugacy between them.
% Also, 
% a homomorphism $\theta:(X,G,T)\to(Y,G,S)$ is \defn{topologically
% surjective} if the range of $\theta$ is dense in $Y$, i.e.,
% $\overline{\theta(X)}=Y$, see \cite{MR2369449}.

% REMARK:
% almost one-to-one can also be defined in term of G_\delta dense, see
% http://www.scholarpedia.org/article/Minimal_dynamical_systems
% Samuel Petite: if minimal, the one-to-one at one point implies Gdelta-dense
% See also Robinson: MR1355301
% In the end, we don't need because we have much stronger.

The set of all $T$-invariant probability measures of a dynamical
system $(X,G,T)$ is denoted by $\mathcal{M}^T(X)$.
An invariant probability measure on $X$ is called \defn{ergodic} if for every set
$B\in\Bcal$ such that $T^{g}(B)=B$ for all $g\in G$, we have that $B$ has either
zero or full measure. A
dynamical system $(X,G,T)$ is \defn{uniquely ergodic}
if it has only one invariant probability measure, i.e., $|\mathcal{M}^T(X)|=1$.
A dynamical system $(X,G,T)$ is said \defn{strictly ergodic}
if it is uniquely ergodic and minimal.

%https://en.wikipedia.org/wiki/Measure-preserving_dynamical_system
% https://www.encyclopediaofmath.org/index.php/Metric_isomorphism

A \defn{measure-preserving dynamical system} is defined as a system
$(X,G,T,\mu,\Bcal)$, where $\mu$ is a probability measure defined on the
Borel $\sigma$-algebra $\Bcal$ of subsets of $X$, and $T^g:X\to X$ is a measurable map
which preserves the measure $\mu$ for all $g\in G$, that is,
$\mu(T^g(B))=\mu(B)$ for all $B\in\Bcal$. The measure $\mu$ is said to be
\defn{$T$-invariant}.
In what follows, 
$\Bcal$ is always the Borel
$\sigma$-algebra of subsets of $X$, so
we omit $\Bcal$ and write $(X,G,T,\mu)$ when it is clear from the context.

Let $(X,G,T,\mu,\Bcal)$
and $(X',G,T',\mu',\Bcal')$ be two measure-preserving dynamical systems.
We say that the two systems are
\defn{isomorphic} %\cite[p.118]{MR2371216}
if there exist measurable sets $X_0\subset X$ and $X'_0\subset X'$
of full measure (i.e., $\mu(X\setminus X_0)=0$
and $\mu'(X'\setminus X'_0)=0$) with
$T^g(X_0)\subset X_0$, $T'^g(X'_0)\subset X'_0$ for all $g\in G$
and there exists a map $\phi:X_0\to X'_0$, called an \defn{isomorphism},
that is one-to-one and onto and such that for all $A\in\Bcal'(X'_0)$,
\begin{itemize}
    \item $\phi^{-1}(A)\in\Bcal(X_0)$,
    \item $\mu(\phi^{-1}(A))=\mu'(A)$, and
    \item $\phi\circ T^g(x)=T'^g\circ\phi(x)$ for all $x\in X_0$ and $g\in G$.
\end{itemize}
The role of the set $X_0$ is to make precise the fact that the properties of
the isomorphism need to hold only on a set of full measure.

\subsection{Maximal equicontinuous factor}
In this section, we provide the definition of maximal continuous factor and of
related notions. We recall a sufficient condition for a factor to be the
maximal equicontinuous factor and we prove a result on the set of fiber cardinalities
of the maximal equicontinuous factor of a minimal dynamical system.

A metrizable dynamical system $(X,G,T)$ is called \defn{equicontinuous} if
the family of homeomorphisms $\{T^g\}_{g\in G}$ is equicontinuous, i.e., if for
all $\varepsilon>0$ there exists $\delta>0$ such that
\[
    \dist(T^g(x), T^g(y)) < \varepsilon
\]
for all $g\in G$ and all $x,y\in X$ with $\dist(x,y)<\delta$.
% 1D theorem is %\cite[Theorem 2.42]{MR2041676}
According to a well-known theorem~\cite[Theorem 3.2]{MR3381481},
equicontinuous
minimal systems defined by the action of an Abelian group
are rotations on groups.
% \todo[inline]{well-known theorem: ... are rotations as topologically or as
% measure preserving DS?}

We say that $\theta:(X,G,T)\to(Y,G,S)$ is an \defn{equicontinuous factor} if
$\theta$ is a factor map and $(Y,G,S)$ is equicontinuous.
We say that $(X_{\rm max}, G, T_{\rm max})$ is the \defn{maximal equicontinuous
factor} of $(X,G,T)$ if 
there exists an equicontinuous factor
$\pi_{\rm max}:(X,G,T)\to(X_{\rm max}, G, T_{\rm max})$,
such that for any
equicontinuous factor $\theta:(X,G,T)\to(Y,G,S)$,
there exists a unique factor map $\psi:(X_{\rm max}, G, T_{\rm max})\to(Y,G,S)$
with $\psi\circ\pi_{\rm max}=\theta$.
The maximal equicontinuous factor exists and is unique (up to topological
conjugacy), see \cite[Theorem 3.8]{MR3381481} and \cite[Theorem 2.44]{MR2041676}.

Let $\theta:(X,G,T)\to(Y,G,S)$ be a factor map.
We call the preimage set $\theta^{-1}(y)$ of a point $y\in Y$ the \defn{fiber} of $\theta$ over $y$.
The cardinality of the fiber $\theta^{-1}(y)$ for some $y\in Y$ has an important role and is related to the definition of other notions.
In particular, the factor map $\theta$ is \defn{almost one-to-one} if 
$\{y\in Y:\card(\theta^{-1}(y))=1\}$
is a $G_\delta$-dense set in $Y$.
In that case, $(X,G,T)$ is an \defn{almost one-to-one extension} of $(Y,G,S)$.
Moreover, it provides a sufficient condition to prove that an
equicontinuous factor of a minimal dynamical system is the maximal one as
stated in the next lemma from \cite{MR3381481}.

\begin{lemma}\label{lem:ABKL15-lemma3.11} {\rm\cite[Lemma 3.11]{MR3381481}}
    Let $(X,G,T)$ be a minimal dynamical system
    and $(Y,G,S)$ an equicontinuous dynamical system.
    If $(Y,G,S)$ is a factor of $(X,G,T)$ with factor map $\theta$
    and there exists $y\in Y$ such that $\card(\theta^{-1}(y))=1$,
    then $(Y,G,S)$ is the maximal equicontinuous factor.
\end{lemma}

The \defn{set of fiber cardinalities}
of a factor map $\theta:(X,G,T)\to(Y,G,S)$
is the set $\{\card(\theta^{-1}(y)) : y \in Y\}\subset\N\cup\{\infty\}$, see \cite{MR1877329}.
Note that different terminology is used in \cite{MR1355301} as
the set of fiber cardinalities of a factor map is called \emph{thickness spectrum}
and
its supremum is called \emph{thickness} whereas
the supremum is called \emph{maximum rank} in \cite{MR3381481}.
As shown in the next lemma, the set of fiber cardinalities
of the maximal equicontinuous factor of a minimal dynamical system is invariant
under topological conjugacy.

\begin{lemma}\label{lem:fiber-card-invariant}
    Let $(X,G,T)$ and $(Y,G,S)$ be a minimal dynamical systems.
    Let $f:(X,G,T)\to(X_{\rm max},G,T_{\rm max})$ 
    and $g:(Y,G,S)\to(Y_{\rm max},G,S_{\rm max})$ be two maximal equicontinuous factors.
    If $X$ and $Y$ are topologically conjugate,
    then $f$ and $g$ have the same set of fiber cardinalities.
\end{lemma}

% the definition of regionally proximal below is from:
% HIGHER ORDER REGIONALLY PROXIMAL
% EQUIVALENCE RELATIONS FOR GENERAL MINIMAL
% GROUP ACTIONS
% ELI GLASNER, YONATAN GUTMAN AND XIANGDONG YE

The maximal equicontinuous factor $f:(X,G,T)\to(X_{\rm max},G,T_{\rm max})$
defines an equivalence relation on the elements $a,b\in X$ as $a\equiv b$ if and only if
$f(a)=f(b)$.
A theorem of Auslander \cite[p.130]{MR956049} 
on the equivalence relation defined by the maximal equicontinuous factor
says that if $(X,G,T)$ is minimal, then
$f(a)=f(b)$ if and only if $a$ and $b$ are regionally proximal.
% Two elements $x,y\in X$ are said to be \defn{proximal} %, denoted $(x, y) \in P(X)$, 
% if there is a sequence of elements $g_i\in G$ such that
% $\lim_{i\to\infty}\dist(g_i x, g_i y) = 0$. 
%The system $(G, X)$ is said to be \defn{distal}, if $P(X) = \Delta = \{(x, x)|x \in X\}$.
Two elements $x, y \in X$ are said to be \defn{regionally proximal}
if there are sequences of elements $x_i,y_i\in X$ and a sequence of elements
$g_i\in G$ such that $\lim_{i\to\infty}x_i=x$, $\lim_{i\to\infty}y_i=y$ and 
$\lim_{i\to\infty} \dist(g_i x_i, g_iy_i)=0$.
%We use this property in the proof of the lemma.

\begin{proof}
    Let $\theta:(X,G,T)\to(Y,G,S)$ be a topological conjugacy.
    Let us show that the formula $\pi=g\circ\theta\circ f^{-1}$ defines a
    map $X_{\rm max}\to Y_{\rm max}$.
    Let $x\in X_{\rm max}$. Since $f$ is onto, there exists $a\in X$ such that $f(a)=x$.
    Suppose that $a,b\in f^{-1}(x)$.
    Thus $f(a)=f(b)$ and by Auslander's theorem, $a$ and $b$ are regionally proximal.
    That property depends only on the distance so it is preserved by the topological conjugacy.
    Thus $\theta(a)$ and $\theta(b)$ are regionally proximal.
    Therefore $g(\theta(a))=g(\theta(b))$ which shows that $\pi$ is well-defined.

    The map $\pi$ is one-to-one.
    Let $x,x'\in X_{\rm max}$ and suppose that $\pi(x)=\pi(x')$.
    Let $a,a'\in X$ such that $f(a)=x$ and $f(a')=x'$.
    Then $g(\theta(a))=\pi(x)=\pi(x')=g(\theta(a'))$.
    Thus $\theta(a)$ and $\theta(a')$ are regionally proximal from Auslander's theorem.
    Thus $a$ and $a'$ are regionally proximal and we obtain $x=f(a)=f(a')=x'$.

    It is sufficient to show that the fiber cardinalities of $f$ is a subset of
    the fiber cardinalities of $g$.
    Let $x\in X_{\rm max}$ such that $\pi(x)=y$. Then $g\circ\theta\circ f^{-1}(x)=y$ which means
    that
    $\theta(f^{-1}(x))\subseteq g^{-1}(y)$
    and
    $\{x\}\subseteq f(\theta^{-1}(g^{-1}(y)))$.
    Since $\pi$ is one-to-one, we deduce $\{x\}=f(\theta^{-1}(g^{-1}(y)))$.
    Thus $\theta^{-1}(g^{-1}(y))\subseteq f^{-1}(x)$
    and $g^{-1}(y)\subseteq \theta(f^{-1}(x))$.
    We conclude that $g^{-1}(y)=\theta(f^{-1}(x))$.
    In particular, $\card(f^{-1}(x))= \card(g^{-1}(y))$ and
    \[
    \{\card(f^{-1}(x)) : x \in X_{\rm max}\}
    \subseteq
    \{\card(g^{-1}(y)) : y \in Y_{\rm max}\}.
    \]
    The equality follows from the symmetry of the argument.
\end{proof}

\subsection{Subshifts and shifts of finite type}

% following KLAUS SCHMIDT

We follow the notation of \cite{MR1861953}.
Let $\A$ be a finite set, $d\geq 1$, and let $\A^{\Z^d}$ be the set of all maps
$x:\Z^d\to\A$, equipped with the compact product topology. 
An element $x\in\A^{\Z^d}$ is called \defn{configuration}
and we write it as $x=(x_\bm)=(x_\bm:\bm\in\Z^d)$,
where $x_\bm\in\A$ denotes the value of $x$ at $\bm$. 
The topology on $\A^{\Z^d}$ is compatible with the metric defined for all
configurations $x,x'\in\A^{\Z^d}$ by $\dist(x,x')=2^{-\min\left\{\Vert\bn\Vert\,:\,
x_\bn\neq x'_\bn\right\}}$
where $\Vert\bn\Vert = |n_1| + \dots + |n_d|$.
The \defn{shift action} $\sigma:\bn\mapsto
\sigma^\bn$ of $\Z^d$ on $\A^{\Z^d}$ is defined by
\begin{equation}\label{eq:shift-action}
    (\sigma^\bn(x))_\bm = x_{\bm+\bn}
\end{equation}
for every $x=(x_\bm)\in\A^{\Z^d}$ and $\bn\in\Z^d$. 
If $X\subset \A^{\Z^d}$,
let $\overline{X}$ denote the topological closure of $X$
and let $\sigma(X)=\{\sigma^\bn(x)\mid x\in X, \bn\in\Z^d\}$
denote the shift-closure of $X$.
A subset $X\subset
\A^{\Z^d}$ is \defn{shift-invariant} if 
$\sigma(X)=X$ and a closed, shift-invariant subset
$X\subset\A^{\Z^d}$ is a \defn{subshift}. 
If $X\subset\A^{\Z^d}$ is a subshift we write
$\sigma=\sigma^X$ for the restriction of the shift action
\eqref{eq:shift-action} to $X$. 
%If $X\subset\A^{\Z^d}$ is a subshift it will sometimes be helpful to specify the
%shift action of $\Z^d$ explicitly and to write $(X,\sigma)$ instead of $X$.
When $X$ is a subshift,
the triple $(X,\Z^d,\sigma)$ is a dynamical system
and the notions presented in the previous section hold.

A configuration $x\in X$ is \defn{periodic} if there is a nonzero vector
$\bn\in\Z^d\setminus\{\zero\}$ such that $x=\sigma^\bn(x)$
and otherwise it is said \defn{nonperiodic}.
We say that a nonempty subshift $X$ is \defn{aperiodic}
if the shift action $\sigma$ on $X$ is free.

For any subset $S\subset\Z^d$ let $\pi_S:\A^{\Z^d}\to\A^S$ denote the
projection map which restricts every $x\in\A^{\Z^d}$ to $S$. 
A \defn{pattern} is a function $p\in\A^S$ for some finite subset
$S\subset\Z^d$.
To every pattern $p\in\A^S$ corresponds
a subset $\pi_S^{-1}(p)\subset\A^{\Z^d}$ called \defn{cylinder}.
% A subshift $X\subset\A^{\Z^d}$ is a \emph{shift of finite type (SFT)} if there
% exists a finite set $S\subset\Z^d$ such that
% \begin{equation}\label{eq:SFT}
%     X = \{x\in\A^{\Z^d} \mid \pi_S\cdot\sigma^\bn(x)\in\pi_S(X)
%         \text{ for every } \bn\in\Z^d\}.
% \end{equation}
A subshift $X\subset\A^{\Z^d}$ is a 
\defn{shift of finite type} (SFT) if there exists a finite set $\Fcal$
of \defn{forbidden} patterns such that
\begin{equation}\label{eq:SFT}
    X = \{x\in\A^{\Z^d} \mid \pi_S\circ\sigma^\bn(x)\notin\Fcal
    \text{ for every } \bn\in\Z^d \text{ and } S\subset\Z^d\}.
\end{equation}
In this case, we write $X=\SFT(\Fcal)$.
In this article, we consider shifts of finite type on $\Z\times\Z$, that is, the case
$d=2$.

\section{Symbolic representations and Markov partitions for toral $\Z^2$-rotations}
\label{sec:symbolic-representation}

We follow the section 
\cite[\S 6.5]{MR1369092} on Markov partitions
where we adapt it to the case of invertible $\Z^2$-actions.
A \defn{topological partition} of a metric space $M$ is a finite
collection $\{P_0,P_1,...,P_{r-1}\}$ of disjoint open sets such that $M = 
\overline{P_0}\cup \overline{P_1}\cup\cdots\cup \overline{P_{r-1}}$.

% \todo[inline]{define language}

Suppose that $M$ is a compact metric space,
$(M,\Z^2,R)$ is a dynamical system
and that $\Pcal=\{P_0,P_1,...,P_{r-1}\}$ is a topological partition of $M$.
Let $\A=\{0,1,\dots,r-1\}$ and $S\subset\Z^2$ be a finite set. 
We say that a pattern $w\in\A^S$
is \defn{allowed} for $\Pcal,R$ if
\[
    \bigcap_{\bk\in S} R^{-\bk}(P_{w_\bk}) \neq \varnothing.
\]
Let $\Lcal_{\Pcal,R}$ be the collection of all allowed patterns for $\Pcal,R$.
It can be checked that $\Lcal_{\Pcal,R}$ is the language of a subshift.
Hence, using standard arguments \cite[Prop. 1.3.4]{MR1369092}, there is a
unique subshift $\Xcal_{\Pcal,R}\subset\A^{\Z^2}$ whose language is
$\Lcal_{\Pcal,R}$.  We call $\Xcal_{\Pcal,R}$ the \defn{symbolic dynamical
system corresponding to $\Pcal,R$}.
For each $w\in\Xcal_{\Pcal,R}\subset\A^{\Z^2}$ and $n\geq 0$ there is a corresponding nonempty open set
\[
    D_n(w) = \bigcap_{\Vert\bk\Vert\leq n} R^{-\bk}(P_{w_\bk}) \subseteq M.
\]
The closures $\overline{D}_n(w)$ of these sets are compact
and decrease with $n$, so that
$\overline{D}_0(w)\supseteq
\overline{D}_1(w)\supseteq
\overline{D}_2(w)\supseteq
\dots$.
It follows that $\cap_{n=0}^{\infty}\overline{D}_n(w)\neq\varnothing$.
In order for configurations in $\Xcal_{\Pcal,R}$
to correspond to points in $M$, this intersection should contain only one point.
This leads to the following definition.

\begin{definition}\label{def:symb-repr}
Let $M$ be a compact metric space and
$(M,\Z^2,R)$ be a dynamical system.
A topological partition $\Pcal$ of $M$ gives a \defn{symbolic representation}
of $(M,\Z^2,R)$ if for every configuration $w\in\Xcal_{\Pcal,R}$ the intersection
$\cap_{n=0}^{\infty}\overline{D}_n(w)$ consists of exactly one
point $m\in M$.
We call $w$ a \defn{symbolic representation of $m$}.
\end{definition}

Markov partition were originally defined for one-dimensional dynamical
systems $(M,\Z,R)$ and were extended to $\Z^d$-actions by automorphisms of
compact Abelian group in \cite{MR1632169}.
We allow ourselves to use the same terminology for
dynamical systems $(M,\Z^2,R)$ defined by higher-dimensional actions by rotations.

\begin{definition}\label{def:Markov}
A topological partition $\Pcal$ of $M$ is a \defn{Markov partition} for
$(M,\Z^2,R)$ if 
\begin{itemize}
    \item $\Pcal$ gives a symbolic representation of $(M,\Z^2,R)$ and 
    \item $\Xcal_{\Pcal,R}$ is a shift of finite type (SFT).
\end{itemize}
\end{definition}

Of course, 2-dimensional SFTs are much
different then 1-dimensional SFTs. 
For example, there exist 2-dimensional aperiodic SFTs with zero entropy. 
But this is not possible in the one-dimensional case, 
since one-dimensional infinite SFTs have positive entropy.
In this article, we consider partitions associated to
2-dimensional aperiodic Wang shifts with zero entropy. 

\input{article2_macro_jeandel_rao.tex}

The partitions we consider are partitions of the 2-dimensional torus.
Let $\Gamma$ be a \defn{lattice} in $\R^2$, i.e., a discrete subgroup of the
additive group $\R^2$ with $2$ linearly independent generators.
This defines a $2$-dimensional torus $\generictorus=\R^2/\Gamma$. 
By analogy with the rotation $x\mapsto x+\alpha$ on the circle $\R/\Z$ for
some $\alpha\in\R/\Z$, we use the terminology of \emph{rotation}
to denote the following $\Z^2$-action defined on a 2-dimensional torus.

\begin{definition}\label{def:Z2-rotation}
For some $\balpha,\bbeta\in\generictorus$, we consider
the dynamical system $(\generictorus, \Z^2, R)$ where
$R:\Z^2\times\generictorus\to\generictorus$ 
is the continuous $\Z^2$-action on $\generictorus$
defined by
\[
R^\bn(\bx):=R(\bn,\bx)=\bx + n_1\balpha + n_2\bbeta
\]
for every $\bn=(n_1,n_2)\in\Z^2$.
We say that the $\Z^2$-action $R$ is a 
    \defn{toral $\Z^2$-rotation} or a
    \defn{$\Z^2$-rotation on} $\generictorus$.
\end{definition}

From now on, 
we assume that the compact metric space $M$ is~$\generictorus$
and that $R$ is a $\Z^2$-rotation on~$\generictorus$
when
we consider dynamical systems $(M,\Z^2,R)$.

\begin{lemma}\label{lem:symbolic-representation}
    %Let $(\generictorus,\Z^2,R)$ be a dynamical system
    %on the torus $\generictorus=\R^2/\Gamma$
    %where $\Gamma$ is a lattice in $\R^2$
    %and $R^{\be_1}$ and $R^{\be_2}$ are toral translations.
    Let $(\generictorus,\Z^2,R)$ be a minimal dynamical system
    and $\Pcal=\{P_0,P_1,...,P_{r-1}\}$ be a topological partition
    of $\generictorus$.
    If there exists an atom $P_i$ which is invariant only under the trivial
    translation in $\generictorus$,
    then $\Pcal$ gives a symbolic representation of $(\generictorus,\Z^2,R)$.
\end{lemma}

\begin{proof}
    Let $\A=\{0,1,\dots,r-1\}$.
    Let $w\in\Xcal_{\Pcal,R}\subset\A^{\Z^2}$.
    As already noticed,
    the closures $\overline{D}_n(w)$ are compact
    and decrease with $n$, so that
    $\overline{D}_0(w)\supseteq
    \overline{D}_1(w)\supseteq
    \overline{D}_2(w)\supseteq
    \dots$.
    It follows that $\cap_{n=0}^{\infty}\overline{D}_n(w)\neq\varnothing$.

    We show that $\cap_{n=0}^{\infty}\overline{D}_n(w)$
    contains at most one element.
    Let $\bx,\by\in\generictorus$.
    We assume $\bx\in \cap_{n=0}^{\infty}\overline{D}_n(w)$
    and we want to show that $\by\notin \cap_{n=0}^{\infty}\overline{D}_n(w)$
    if $\bx\neq\by$.
    Let $P_i\subset\generictorus$
    for some $i\in\A$ be an atom which is invariant only under the
    trivial translation. 
    Since $\bx\neq\by$,
    $\overline{P_i}\setminus (\overline{P_i}-(\by-\bx))$ contains an open set $O$.
    Since $(\generictorus,\Z^2,R)$ is minimal,
    any orbit $\{R^\bk\bx\mid\bk\in\Z^2\}$ is dense in $\generictorus$.
    Therefore, there exists $\bk\in\Z^2$ such that
    $R^\bk\bx\in O\subset\interior{P_i}$. 
    Also $\bx\in \cap_{n=0}^{\infty}\overline{D}_n(w) \subset R^{-\bk}
    \overline{P_{w_\bk}}$ which implies
    $R^\bk\bx\in\overline{P_{w_\bk}}$.
    Thus $\overline{P_{w_\bk}}\cap\interior{P_i}\neq\varnothing$ which implies
    that $P_{w_\bk}=P_i$ and $w_\bk=i$
    since $\Pcal$ is a topological partition. Thus
    \[
        \cap_{n=0}^{\infty}\overline{D}_n(w)
        \subset R^{-\bk} \overline{P_{w_\bk}}
        =R^{-\bk} \overline{P_i}.
    \]
    The fact that $R^\bk\bx\in O$ 
    also means that 
    $R^\bk\bx\notin \overline{P_i}-(\by-\bx)$
    which can be rewritten as $R^\bk\by\notin \overline{P_i}$ or
    $\by\notin R^{-\bk}\overline{P_i}$ and we conclude that $\by\notin
    \cap_{n=0}^{\infty}\overline{D}_n(w)$.
    Thus $\Pcal$ gives a symbolic representation of
    $(\generictorus,\Z^2,R)$.
\end{proof}

\begin{remark}\label{rem:non-minimal}
    Note that minimality hypothesis in Lemma~\ref{lem:symbolic-representation}
    is not necessary.
    For example, the partition $\boxtimes$ of the torus $\torus=\R^2/\Z^2$
    gives a symbolic representation of the toral $\Z^2$-rotation
    defined by
    $R(\bn,\bx)=\bx + n_1(\sqrt{2},0) + n_2(\sqrt{3},0)$
    even if $(\torus,\Z^2,R)$ is not minimal.
\end{remark}

\section{A one-to-one map from the 2-torus to symbolic representations}
\label{sec:one-to-one-map}

The goal of this section is to express the symbolic dynamical system $\Xcal_{\Pcal,R}$
as the closure of the image of a one-to-one map defined on the $2$-torus.
First we define the map on the points of the torus having a unique symbolic
representation. Then, we extend it on all points of the torus by
approaching them from some direction $\bv\in\R^2$.
%\todo[inline]{some introduction paragraph, speak about decimal representation of $\pi$}

The set
\begin{equation*}\label{eq:boundaries}
    \Delta_{\Pcal,R}:=\bigcup_{\bn\in\Z^2}R^\bn
         \left(\bigcup_{a\in\A}\partial P_a\right)
         \subset \generictorus
\end{equation*}
is the set of points whose orbits under the toral $\Z^2$-rotation $R$ intersect
the boundary of the topological partition
$\Pcal=\{P_a\}_{a\in\A}$.
From the Baire Category Theorem \cite[Theorem 6.1.24]{MR1369092}, the set
$\generictorus\setminus\Delta_{\Pcal,R}$ is dense in $\generictorus$.

For every starting point 
$\bx\in \generictorus\setminus \Delta_{\Pcal,R}$, 
the coding of its orbit under the
toral $\Z^2$-rotation $R$ is a 2-dimensional configuration:
\[
\begin{array}{rccl}
    \scConfig^{\Pcal,R}_{\bx}:& \Z\times\Z &\to & \A \\
    &\bn &\mapsto & a \quad\text{ if and only if }\quad R^\bn(\bx)\in P_a.
\end{array}
\]
Thus it defines a map
\[
\begin{array}{rccl}
    \scSymbRep:& \generictorus\setminus \Delta_{\Pcal,R} &\to & \A^{\Z^2} \\
      &\bx &\mapsto& \scConfig^{\Pcal,R}_{\bx}.
\end{array}
\]
The map $\scSymbRep$
can not be extended continuously on $\generictorus$.
Up to some choice to be made, it can still be extended to the whole domain
$\generictorus$.
Recall that for interval exchange transformations, one way to deal with this
issue is to consider two copies $x^-$ and $x^+$ for each discontinuity point
\cite{MR0357739}.
Here we use this idea in order to
 extend $\scSymbRep$ on the whole domain $\generictorus$ by
approaching any point from a chosen direction.
Not all directions work, so we need some care to formalize this properly.
Let $\Theta^\Pcal$ with $\{\zero\}\subseteq\Theta^\Pcal\subset\R^2$ be the set of vectors
parallel to a segment included in the boundary 
of some atom $P_a\in\Pcal$.
If all atoms have curved boundaries, then $\Theta^\Pcal=\{\zero\}$.
If the atoms are polygons like in this article,
then the set $\Theta^\Pcal$ contains nonzero directions.
In any case, we assume that $\R\Theta^\Pcal=\Theta^\Pcal$.
For every $\bv\in\R^2\setminus\Theta^\Pcal$ we define 
\[
\begin{array}{rccl}
    \scSymbRep^\bv:& \generictorus &\to & \A^{\Z^2} \\
    &\bx &\mapsto& \lim_{\epsilon\to\zero}\scSymbRep (\bx+\epsilon\cdot\bv).
\end{array}
\]

We say that the configuration $\scSymbRep(\bx)= \scConfig^{\Pcal,R}_{\bx}$ is \defn{generic}
if $\bx\in\generictorus\setminus \Delta_{\Pcal,R}$
and that $\scSymbRep^\bv(\bx)$ is \defn{singular} if
$\bx\in\Delta_{\Pcal,R}$ for some $\bv\in\R^2\setminus\Theta^\Pcal$.
%These notions will be used in Part~\ref{part:3}.
The choice of direction $\bv$ is not so important since the topological closure
of the range of $\scSymbRep^\bv$ does not depend on $\bv$ as shown in the next
lemma. In other words, singular configurations are limits of generic configurations
and $\Xcal_{\Pcal,R}$ is equal to the topological closure of the range of
$\scSymbRep$.

\begin{lemma}\label{lem:closure-of-tilings}
    For every
$\bv\in\R^2\setminus\Theta^\Pcal$, the following equalities hold
\[
    \overline{\scSymbRep^\bv(\generictorus)} 
    = \overline{\scSymbRep(\generictorus\setminus \Delta_{\Pcal,R})}
    = \Xcal_{\Pcal,R}
\]
where $\Xcal_{\Pcal,R}$ is the symbolic dynamical system corresponding to
    $\Pcal,R$.
\end{lemma}

\begin{proof}
    ($\supseteq$) 
    If $\bx\in\generictorus\setminus \Delta_{\Pcal,R}$, then
    $\scSymbRep(\bx)=\scSymbRep^\bv(\bx)$.
    Thus $\scSymbRep(\generictorus\setminus \Delta_{\Pcal,R})
     =\scSymbRep^\bv(\generictorus\setminus \Delta_{\Pcal,R})$.
    Then
    \[
    \overline{\scSymbRep(\generictorus\setminus \Delta_{\Pcal,R})}
     =\overline{\scSymbRep^\bv(\generictorus\setminus \Delta_{\Pcal,R})}
     \subseteq \overline{\scSymbRep^\bv(\generictorus)}.
     \]

    ($\subseteq$)
    Let $w\in\scSymbRep^\bv(\Delta_{\Pcal,R})$.
    Then $w=\lim_{\epsilon\to\zero}\scSymbRep (\bx+\epsilon\bv)$ for some
    $\bx\in \Delta_{\Pcal,R}$. We may extract a subsequence
    $\left(\scSymbRep (\bx+\epsilon_n\bv)\right)_{n\in\N}$
    with $\epsilon_n\in\R$
    such that $\bx+\epsilon_n\bv\in\generictorus\setminus \Delta_{\Pcal,R}$ for all $n\in\N$.
    This implies that
    $w\in\overline{\scSymbRep(\generictorus\setminus \Delta_{\Pcal,R})}$.
    Therefore
    $\scSymbRep^\bv(\Delta_{\Pcal,R})\subseteq
    \overline{\scSymbRep(\generictorus\setminus \Delta_{\Pcal,R})}$.
    We obtain
    \begin{align*}
    \overline{\scSymbRep^\bv(\generictorus)} 
        &=
    \overline{\scSymbRep^\bv(\generictorus\setminus \Delta_{\Pcal,R}) \cup 
              \scSymbRep^\bv(\Delta_{\Pcal,R})} \\
        &\subseteq
    \overline{\scSymbRep(\generictorus\setminus \Delta_{\Pcal,R})}
    \end{align*}
    which proves the first equality.

    Recall that the collection of all allowed patterns for $\Pcal,R$ is the
    language $\Lcal_{\Pcal,R}$.
    The set 
    $\overline{\scSymbRep(\generictorus\setminus \Delta_{\Pcal,R})}$
    is a subshift and contains $\Lcal_{\Pcal,R}$.
    Moreover the language of
    $\overline{\scSymbRep(\generictorus\setminus \Delta_{\Pcal,R})}$
    is contained in $\Lcal_{\Pcal,R}$.
    The equality
    $\overline{\scSymbRep(\generictorus\setminus \Delta_{\Pcal,R})}
    = \Xcal_{\Pcal,R}$
    follows since
    the symbolic dynamical system $\Xcal_{\Pcal,R}$ is 
    the unique subshift whose language is $\Lcal_{\Pcal,R}$.
\end{proof}

\begin{lemma}\label{lem:tiling-one-to-one}
    Let $\Pcal$ give a symbolic representation of the dynamical system
    $(\generictorus,\Z^2,R)$
    and let $\bv\in\R^2\setminus\Theta^\Pcal$.
    Then $\scSymbRep:\generictorus\setminus\Delta_{\Pcal,R}\to\Xcal_{\Pcal,R}$ 
    and $\scSymbRep^\bv:\generictorus\to\Xcal_{\Pcal,R}$ 
    are one-to-one.
    Moreover, the following diagrams commute:
\[ 
\begin{tikzcd}
    \generictorus\setminus\Delta_{\Pcal,R} \arrow{r}{R^\bk} 
    \arrow[swap]{d}{\scSymbRep} & 
    \generictorus\setminus\Delta_{\Pcal,R} \arrow{d}{\scSymbRep} \\%
    \Xcal_{\Pcal,R} \arrow{r}{\sigma^\bk} & \Xcal_{\Pcal,R}
\end{tikzcd}
\qquad
    \text{ and }
\qquad
\begin{tikzcd}
    \generictorus \arrow{r}{R^\bk} 
    \arrow[swap]{d}{\scSymbRep^\bv} & \generictorus \arrow{d}{\scSymbRep^\bv} \\%
    \Xcal_{\Pcal,R} \arrow{r}{\sigma^\bk} & \Xcal_{\Pcal,R}
\end{tikzcd}
\]
for every $\bk\in\Z^2$.
\end{lemma}

\begin{proof}
    The fact that $w=\scSymbRep^\bv(\bx)$ implies that
    $\bx\in\cap_{n=0}^{\infty}\overline{D}_n(w)$. Therefore, if
    $\scSymbRep^\bv(\bx)=\scSymbRep^\bv(\by)=w$ 
    then $\bx,\by\in\cap_{n=0}^{\infty}\overline{D}_n(w)$.
    Since $\Pcal$ gives a symbolic representation of the dynamical system
    $(\generictorus,\Z^2,R)$,
    the set $\cap_{n=0}^{\infty}\overline{D}_n(w)$ contains at most one element,
    and
    it implies that $\bx=\by$. Thus, $\scSymbRep^\bv$ is one-to-one.
    As $\scSymbRep^\bv$ and $\scSymbRep$ agree on
    $\generictorus\setminus\Delta_{\Pcal,R}$, we also have that $\scSymbRep$ is
    one-to-one.

We now show conjugacy of $\Z^2$-actions.
Let $\bk\in\Z^2$, $\bx\in\generictorus\setminus\Delta_{\Pcal,R}$
and $\bn\in\Z^2$. 
We have
\begin{align*}
    (\sigma^\bk\circ\scSymbRep(\bx))(\bn)
    &=(\sigma^\bk\circ\scConfig^{\Pcal,R}_{\bx})(\bn)
    =\scConfig^{\Pcal,R}_{\bx}(\bn+\bk)\\
    &= \scConfig^{\Pcal,R}_{R^\bk\bx}(\bn)
    =(\scSymbRep(R^\bk\bx))(\bn)\\
    &=(\scSymbRep\circ R^\bk(\bx))(\bn).
\end{align*}
Therefore $\sigma^\bk\circ\scSymbRep =\scSymbRep\circ R^\bk$.
The conjugacy of $\Z^2$-actions by the map $\scSymbRep$ also extends to $\scSymbRep^\bv$.
\end{proof}
    
The fact that $\scSymbRep^\bv$ is one-to-one means that it admits a left-inverse
map $f:\scSymbRep^\bv(\generictorus)\to\generictorus$ such that
$f\circ\scSymbRep^\bv=\mathrm{Id}_{\generictorus}$.
But we can say more and define the map $f$ on
the closure $\overline{\scSymbRep^\bv(\generictorus)}=\Xcal_{\Pcal,R}$.
Indeed, if $\Pcal$ gives a symbolic representation of the 
dynamical system $(\generictorus,\Z^2,R)$, then there is a well-defined function
$f$ from $\Xcal_{\Pcal,R}$ to $\generictorus$ which maps a configuration
$w\in\Xcal_{\Pcal,R}\subset\A^{\Z^2}$ to the unique point
$f(w)\in\generictorus$ in the intersection
$\cap_{n=0}^{\infty}\overline{D}_n(w)$. We consider the map $f$ in the next
section.

\section{A factor map from symbolic representations to the 2-torus}

In the spirit of \cite[Prop. 6.5.8]{MR1369092},
the following result shows that there exists a continuous and onto homomorphism and
therefore a factor map from $(\Xcal_{\Pcal,R},\Z^2,\sigma)$ to
$(\generictorus,\Z^2,R)$.

% \todo[inline]{do we want to speak of semi-conjugacy?}

\begin{proposition}\label{prop:factor-map}
    Let $\Pcal$ give a symbolic representation of the dynamical system
    $(\generictorus,\Z^2,R)$.
    Let $f:\Xcal_{\Pcal,R}\to\generictorus$ be defined 
    such that $f(w)$ is the unique point
    in the intersection $\cap_{n=0}^{\infty}\overline{D}_n(w)$.
    The map $f$ is a factor map from
            $(\Xcal_{\Pcal,R},\Z^2,\sigma)$ to $(\generictorus,\Z^2,R)$
    which makes the following diagram commute
    \[ 
    \begin{tikzcd}
        \Xcal_{\Pcal,R} \arrow{r}{\sigma^\bk} 
        \arrow[swap]{d}{f} & \Xcal_{\Pcal,R} \arrow{d}{f} \\%
        \generictorus \arrow{r}{R^\bk} & \generictorus
    \end{tikzcd}
    \]
    for every $\bk\in\Z^2$.
    The map $f$ is one-to-one on
    $f^{-1}(\generictorus\setminus\Delta_{\Pcal,R})$.
\end{proposition}

\begin{proof}
    Let $\Pcal=\{P_0,P_1,\dots,P_{r-1}\}$.
    We show that the map $f$ is continuous. Let $\varepsilon>0$.
    Let $w\in\Xcal_{\Pcal,R}$.
    Since the partition gives a symbolic representation, there exists $n\in\N$
    such that the diameter of
    \[
        \overline{D}_n(w)=
        \bigcap_{\Vert\bk\Vert\leq n} R^{-\bk}
        \left(\overline{P_{w_\bk}}\right)
    \]
    is smaller than or equal to $\varepsilon$.
    That set contains $\left(\cap_{n=0}^{\infty}\overline{D}_n(w)\right)
    \cup \left(\cap_{n=0}^{\infty}\overline{D}_n(w')\right)$
    if $w'\in\Xcal_{\Pcal,R}$ is such that $\dist_{\Xcal_{\Pcal,R}}(w,w')<\frac{1}{2^n}$.
    We conclude that if $\dist_{\Xcal_{\Pcal,R}}(w,w')<\frac{1}{2^n}$, then
    $\dist_\generictorus(f(w),f(w'))<\varepsilon$ which means that $f$ is
    continuous.

    % % Slower proof as in Lind Marcus
    We show that the map $f$ is onto.
    Let $\bx\in\generictorus\setminus\Delta_{\Pcal,R}$
    and $w=\scSymbRep(\bx)$.
    Then $\bx\in\cap_{n=0}^{\infty}\overline{D}_n(w)$.
    Since $\Pcal$ gives a symbolic representation of $(\generictorus,\Z^2,R)$,
    we have that
    \[
        \{\bx\} 
        =\bigcap_{n=0}^{\infty}\overline{D}_n(w)
        =\bigcap_{n=0}^{\infty}D_n(w),
    \]
    so that $f(w)=\bx$. Thus the image of $f$ contains the dense set
    $\generictorus\setminus\Delta_{\Pcal,R}$. Since the image of a compact set
    via a continuous map is compact and therefore closed, it follows that the
    image of $f$ is all of $\generictorus$.

    An alternate proof that $f$ is onto uses $\scSymbRep^\bv$.
    Let $\bx\in\generictorus$
    and $w=\scSymbRep^\bv(\bx)$ for some $\bv\in\R^2\setminus\Theta^\Pcal$.
    We have that $\cap_{n=0}^{\infty}\overline{D}_n(w)=\{\bx\}$.
    Therefore, $f(w)=\bx$ and $f$ is onto.

    We show that the map $f$ is a homomorphism:
    \begin{align*}
    R^\bk \{f(w)\} &= R^\bk \left(\bigcap_{n=0}^{\infty}\overline{D}_n(w)\right)
        = R^\bk 
        \bigcap_{\bn\in\Z^2} R^{-\bn}
        \overline{P_{w_\bn}}
        = 
        \bigcap_{\bn\in\Z^2} R^{-(\bn-\bk)}
        \overline{P_{\sigma^\bk w_{\bn-\bk}}}\\
        &= 
        \bigcap_{\bm\in\Z^2} R^{-\bm}
        \overline{P_{\sigma^\bk w_{\bm}}}
        = \bigcap_{n=0}^{\infty}\overline{D}_n(\sigma^\bk w)
        = \{f(\sigma^\bk w)\}
    \end{align*}
    where $\bm=\bn-\bk$. Therefore $R^\bk\circ f=f\circ\sigma^\bk$ for every
    $\bk\in\Z^2$ and $f:\Xcal_{\Pcal,R}\to\generictorus$ is a factor map.

    We show that $f$ is one-to-one on $f^{-1}(\generictorus\setminus\Delta_{\Pcal,R})$.
    Let $\bx\in\generictorus\setminus\Delta_{\Pcal,R}$ and suppose
    that $w,w'\in f^{-1}(\bx)$.
    This means that $\cap_{n=0}^{\infty}\overline{D}_n(w)
                    =\cap_{n=0}^{\infty}\overline{D}_n(w')=\{\bx\}$.
    Therefore for every $\bn\in\Z^2$ we have
    \[
        \bx\in
        \left(
        R^{-\bn} \overline{P_{w_\bn}}
        \right)^\circ
        \cap
        \left(
        R^{-\bn} \overline{P_{w'_\bn}}
        \right)^\circ.
    \]
    Then $w_\bn=w'_\bn$ for every $\bn\in\Z^2$ and $w=w'$.
    Therefore 
    for every
    $\bx\in\generictorus\setminus\Delta_{\Pcal,R}$,
    $f^{-1}(\bx)$ contains exactly one element.
\end{proof}

As mentioned in Remark~\ref{rem:non-minimal},
it is possible that $\Xcal_{\Pcal,R}$ is not minimal. But 
as shown in the next lemma, it is minimal if $R$ is minimal.

\begin{lemma}\label{lem:minimal-aperiodic}
    Let $\Pcal$ give a symbolic representation of the dynamical system
    $(\generictorus,\Z^2,R)$. Then
\begin{enumerate}[\rm (i)]
    \item if $(\generictorus,\Z^2,R)$ is minimal,
        then $(\Xcal_{\Pcal,R},\Z^2,\sigma)$ is minimal,
    \item if $R$ is a free $\Z^2$-action on $\generictorus$,
        then $\Xcal_{\Pcal,R}$ aperiodic.
\end{enumerate}
\end{lemma}

\begin{proof}
    Let $f:\Xcal_{\Pcal,R}\to\generictorus$ be the factor map
    from Proposition~\ref{prop:factor-map}.

    (i)
    Let $Y\subseteq \Xcal_{\Pcal,R}$ be a nonempty subshift. Thus
    $Y$ is compact. Continuous map preserve compact sets, thus $f(Y)$ is
    compact.  The set $f(Y)$ is also $R$-invariant since
    $R^\bk f(Y)=f(\sigma^\bk Y)=f(Y)$ for every $\bk\in\Z^2$.
    Since $(\generictorus,\Z^2,R)$ is minimal, the only nonempty
    compact subset of $\generictorus$ which is invariant under $R$ is
    $\generictorus$. Thus $f(Y)=\generictorus$.

    For every $x\in\generictorus$, $f^{-1}(x)\cap Y\neq\varnothing$.
    Then $Y$ contains $\scSymbRep(x)$ for every $x\in\generictorus$
    such that $f^{-1}(x)$ is a singleton.
    Then $Y$ contains $\scSymbRep(\generictorus\setminus \Delta_{\Pcal,R})$.
    Since $Y$ is closed, it must contain
    $\overline{\scSymbRep(\generictorus\setminus \Delta_{\Pcal,R})}$.
    From Lemma~\ref{lem:closure-of-tilings}, this means that
    $\Xcal_{\Pcal,R}\subseteq Y$.
    Thus $Y=\Xcal_{\Pcal,R}$ and $\Xcal_{\Pcal,R}$ is minimal.

    (ii)
    Suppose that there exists $w\in \Xcal_{\Pcal,R}$ such that 
    $w$ is periodic, i.e., there exists $\bk\in\Z^2$ such that
    $\sigma^\bk w=w$.
    Since $f$ commutes the $\Z^2$-actions, we obtain
    \[
        R^\bk f(w)
        =f(\sigma^\bk w)
        =f(w).
    \]
    Since we assume that $R$ is a free $\Z^2$-action, this implies
    that $\bk=\zero$.
    Thus $\Xcal_{\Pcal,R}$ is aperiodic.
\end{proof}

We can now deduce a corollary of Proposition~\ref{prop:factor-map}.

\begin{corollary}\label{cor:max-equicontinuous-factor}
    If the dynamical system $(\generictorus,\Z^2,R)$ is minimal and
    $\Pcal$ gives a symbolic representation of $(\generictorus,\Z^2,R)$,
    then $(\generictorus, \Z^2,R)$ is the maximal equicontinuous factor of
    $(\Xcal_{\Pcal,R},\Z^2,\sigma)$.
\end{corollary}

\begin{proof}
    From Lemma~\ref{lem:minimal-aperiodic} (i),
    $(\Xcal_{\Pcal,R},\Z^2,\sigma)$ is minimal.
    The dynamical system $(\generictorus, \Z^2,R)$ is equicontinuous.
    We proved in Proposition~\ref{prop:factor-map} that the factor map
    $f$ is one-to-one on $f^{-1}(\generictorus\setminus\Delta_{\Pcal,R})$.
    In particular, there exists
    at least one element $\by\in\generictorus\setminus\Delta_{\Pcal,R}$
    such that $\card(f^{-1}(\by))=1$.
    From Lemma~\ref{lem:ABKL15-lemma3.11},
    $(\generictorus, \Z^2,R)$ is the maximal
    equicontinuous factor of $(\Xcal_{\Pcal,R},\Z^2,\sigma)$.
\end{proof}

\begin{remark}
    There are some more consequences.
From Proposition~\ref{prop:factor-map}, we deduce that
\[
\lambda\left(\{\bx\in\generictorus:\card(f^{-1}(\bx))>1\}\right)
\leq\lambda\left(\Delta_{\Pcal,R}\right) =0
\]
where $\lambda$ be the Haar measure on $\generictorus$.
From \cite{fuhrmann_structure_2018}, 
this implies that $(\Xcal_{\Pcal,R},\Z^2,\sigma)$
is a \emph{regular} extension of $(\generictorus,\Z^2,R)$
and that $(\Xcal_{\Pcal,R},\Z^2,\sigma)$ is \emph{mean equicontinuous}
which has more structural consequences including 
the fact of having \emph{discrete spectrum with continuous eigenfunctions}.
We refer the reader to \cite{fuhrmann_structure_2018} for the definitions
of regular extension and mean equicontinuous.
% Suppose that $f:(X,G,T)\to(Y,G,S)$ is a factor map.
% We say that $(X,G,T)$ is a \defn{regular} extension of $(Y,G,S)$ if for
% every $T$-invariant measure $\mu$ on $X$ we have that
%     \[
% f_\star(\mu)\left(\{y\in Y:\card(f^{-1}(y))>1\}\right)=0.
%     \]
% where $f_\star$ is the push-forward map.
\end{remark}

\section{An isomorphism between symbolic dynamical systems and toral $\Z^2$-rotations}

Let $X\subset\A^{\Z^2}$ be a subshift.
Recall that for any subset $S\subset\Z^2$, $\pi_S:X\to\A^S$ is
the projection map which restricts every $w\in X$ to $S$. 
To every finite pattern $p\in\A^S$ correspond
a cylinder $[p]=\pi_S^{-1}(p)\subset X$.
The set of all cylinders
\[
    \{[p] \mid p\in\A^S \text{ with }S\subset\Z^2\text{ finite}\}
\]
generates the Borel $\sigma$-algebra on $X$.

Let $\Pcal=\{P_a\}_{a\in\A}$ give a symbolic representation of the dynamical
system $(\generictorus,\Z^2,R)$
and
let $f:\Xcal_{\Pcal,R}\to\generictorus$ be the factor map
from Proposition~\ref{prop:factor-map}.
For each $a\in\A$, we have that
\[
f([a])=\overline{P_a}\subset\generictorus
\]
is a closed set.
Thus the image of a cylinder $[p]$ under $f$ for some finite pattern
$p\in\A^S$ is a closed set 
called \defn{coding region} for the pattern $p$
being the finite intersection of closed sets:
\begin{equation}\label{eq:cylindre-intersection-polygones}
    f\left([p]\right) = 
    \bigcap_{\bn\in S} R^{-\bn}
    \overline{P_{p_\bn}}\subset\generictorus.
\end{equation}

The following proposition can be seen as an explicit construction
of a strictly ergodic symbolic dynamical system isomorphic 
to the $\Z^2$-rotation $R$ on the torus $\generictorus$
as established by the Theorem of Jewett and Krieger \cite{denker_strictly_1976}
for one-dimensional dynamical systems and generalized to $\Z^2$-actions by Rosenthal
\cite{MR931867}.

\begin{proposition}\label{prop:measurable-conjugacy}
    Let $\Pcal$ give a symbolic representation of a minimal dynamical system
    $(\generictorus,\Z^2,R)$.
    Suppose that $\lambda(\partial P)=0$ for each atom $P\in\Pcal$
    where $\lambda$ is the Haar measure on $\generictorus$.
    Then the dynamical system $(\Xcal_{\Pcal,R},\Z^2,\sigma)$ is strictly
    ergodic and
    the measure-preserving dynamical system $(\Xcal_{\Pcal,R},\Z^2,\sigma,\nu)$
    is isomorphic 
    to $(\generictorus,\Z^2,R,\lambda)$ 
    where $\nu$ is the unique shift-invariant probability measure on
    $\Xcal_{\Pcal,R}$.
\end{proposition}

% https://mathoverflow.net/questions/181752/is-every-continuous-function-measurable

\begin{proof}
    We prove that the factor map
    $f:\Xcal_{\Pcal,R}\to\generictorus$ from
    Proposition~\ref{prop:factor-map} provides the isomorphism.
    The map $f$ is measurable as $f$ is continuous and $f^{-1}(K)$ is compact
    for any compact subset $K\subset\generictorus$.
    Let $\lambda$ be the Haar measure on $\generictorus$.
    By hypothesis, $(\generictorus,\Z^2,R)$ is minimal.
    It is also strictly ergodic \cite{MR648108}
    with $\lambda$ being the only $R$-invariant probability measure on
    $\generictorus$.
    % voir aussi Pytheas Fogg page 20 pour le resultat dans Walters
    % sur les rotations
    %\todo[inline]{cite Weil theorem or see
    %\url{https://en.wikipedia.org/wiki/Haar_measure}
    %\url{https://www.encyclopediaofmath.org/index.php/Haar_measure}
    %}

    Since $\sigma$ is continuous and
    $\Xcal_{\Pcal,R}$ is a compact metric space,
    the set $\mathcal{M}^\sigma(\Xcal_{\Pcal,R})$ 
    of $\sigma$-invariant
    probability measures on $\Xcal_{\Pcal,R}$ is nonempty
    \cite[Cor.~6.9.1]{MR648108}.
    Thus let $\nu\in\mathcal{M}^\sigma(\Xcal_{\Pcal,R})$.
    Let $Z=[p]\subset \Xcal_{\Pcal,R}$ be the cylinder corresponding to some
    pattern $p\in\A^S$ for some finite subset $S\subset\Z^2$.
    From Equation~\eqref{eq:cylindre-intersection-polygones}
    we know that
    $f\left(Z\right)$ is a closed set being the intersection of
    a finite number of closed sets.
    Closed sets as well as their interior are both measurable for the Haar
    measure $\lambda$. Continuity of $f$ implies that
    $f^{-1}\left( f\left(Z\right)\right)$ 
    and
    $f^{-1}\left(f\left(Z\right)^\circ\right)$ 
    are both measurable for $\nu$.

    For each letter $a\in\A$, we have 
    $f^{-1}(f([a])^\circ)\subset[a]$.
    Thus we have
    \[
    f^{-1}\left(f\left(Z\right)^\circ\right)
    \subset
    Z
    \subset
    f^{-1}\left( f\left(Z\right)\right)
    \]
    so that
    \[
    \nu(f^{-1}\left(f\left(Z\right)^\circ\right))
    \leq
    \nu(Z)
    \leq
    \nu(f^{-1}\left( f\left(Z\right)\right)).
    \]
    Let $f_\star$ be the pushforward map
    \[
    \begin{array}{rccl}
        f_\star: &\mathcal{M}^\sigma(\Xcal_{\Pcal,R}) & \to & 
                \mathcal{M}^{R}(\generictorus) \\
        &\nu & \mapsto & \nu\circ  f^{-1}
    \end{array}
    \]
    which maps shift-invariant measures on $\Xcal_{\Pcal,R}$
    to $R$-invariant measures on $\generictorus$.
    But there is only one such measure, so that $f_\star\nu=\lambda$ for every
    $\nu\in\mathcal{M}^\sigma(\Xcal_{\Pcal,R})$.  For every
    $\nu\in\mathcal{M}^\sigma(\Xcal_{\Pcal,R})$, we have for the left-hand
    side
    \[
    \nu( f^{-1}\left(f\left(Z\right)^\circ\right))
    =
    f_\star\nu\left(f\left(Z\right)^\circ\right)
    =
    \lambda\left(f\left(Z\right)^\circ\right)
    \]
    and for the right-hand side
    \[
    \nu( f^{-1}\left( f\left(Z\right)\right))
    =
    f_\star\nu\left( f\left(Z\right)\right)
    =
    \lambda\left( f\left(Z\right)\right).
    \]
    As the boundary of $f(Z)$ is a $\lambda$-null set, we obtain
    \[
    \lambda\left( f\left(Z\right)\right)
    =
    \lambda\left(f\left(Z\right)^\circ\right)
    \leq
    \nu(Z)
    \leq
    \lambda\left( f\left(Z\right)\right)
    \]
    and we conclude that
    \[
    \nu(Z)
    =
    \lambda\left( f\left(Z\right)\right).
    \]
    Since measures are defined from the measure of 
    cylinders which generate the Borel $\sigma$-algebra, we conclude that
    there is a unique shift-invariant probability measure on $\Xcal_{\Pcal,R}$.
    Thus $\Xcal_{\Pcal,R}$ is uniquely ergodic and therefore strictly ergodic
    since minimality of $\Xcal_{\Pcal,R}$ was proved in
    Lemma~\ref{lem:minimal-aperiodic}.
\end{proof}

Proposition~\ref{prop:measurable-conjugacy} implies uniform pattern frequencies for
configurations in $\Xcal_{\Pcal,R}$.
It also means that the
symbolic dynamical system $\Xcal_{\Pcal,R}$ is an almost one-to-one extension of a Kronecker
dynamical system (a rotation action on a compact Abelian group)
and from Von Neumann's Theorem \cite[Theorem 3.9]{MR2590264},
it implies that $\Xcal_{\Pcal,R}$ has discrete spectrum.
See also \cite{MR2369449} for a treatment of Von Neumann's
Theorem in the context of tiling dynamical systems.
%Voir aussi la page 23 de Pytheas Fogg.

\part{Wang shifts as codings of toral $\Z^2$-rotations}\label{part:2}

This part is divided into 5 sections.
After introducing Wang tiles and Wang shifts,
we present a generic method for constructing sets of Wang tiles
and valid configurations in the associated Wang shift as codings of
toral $\Z^2$-rotations. We illustrate the method 
on Jeandel-Rao's set of 11 Wang tiles
and on a self-similar set of 19 Wang tiles.
We expose the limitations of the method by presenting two ``non-examples''.

\section{Wang shifts}

A \defn{Wang tile} $\tau=\tile{$a$}{$b$}{$c$}{$d$}$ is a unit square with colored edges
formally represented as a tuple of four colors $(a,b,c,d)\in I\times J\times
I\times J$
where $I$, $J$ are
two finite sets (the vertical and horizontal colors respectively). 
For each Wang tile $\tau=(a,b,c,d)$, let
$\scright(\tau)=a$,
$\sctop(\tau)=b$,
$\scleft(\tau)=c$,
$\scbottom(\tau)=d$
denote the colors of the right, top, left and bottom edges of $\tau$
\cite{wang_proving_1961,MR0297572}.

Let $\T$ be a set of Wang tiles. 
A configuration $x:\Z^2\to\T$ is \defn{valid} if
it assigns tiles to each position of $\Z^2$ so that contiguous edges
have the same color, that is,
\begin{align}
    \scright(x_{\bn})&=\scleft(x_{\bn+\be_1})\label{eq:1}\\
    \sctop(x_{\bn})&=\scbottom(x_{\bn+\be_2})\label{eq:2}
\end{align}
for every $\bn\in\Z^2$ where $\be_1=(1,0)$ and $\be_2=(0,1)$.
Let $\Omega_\T\subset\T^{\Z^2}$ denote the set of all valid configurations $\Z^2\to\T$
and we call it the \defn{Wang shift} of $\T$. 
Together with the shift action $\sigma$ of $\Z^2$ on $\T^{\Z^2}$,
$\Omega_\T$ is a SFT of the form \eqref{eq:SFT}
since there exists a finite set of
forbidden patterns made of all horizontal and vertical dominoes of two tiles
that do not share an edge of the same color.

%As explained in the first page of \cite{MR0297572} 
%(see also \cite[Prop. 5.9]{MR3136260}),
%if $\T$ is periodic, then there is a tiling $x$ by $\T$ with two linearly
%independent translation vectors (in particular a tiling $x$ with vertical
%and horizontal translation vectors).

A set of Wang tiles $\T$ is \defn{periodic} if there exists a periodic configuration
$x\in\Omega_\T$. 
Originally, Wang thought that every set of Wang tiles $\T$ is periodic 
as soon as $\Omega_\T$ is nonempty \cite{wang_proving_1961}.
This statement is equivalent to the existence of an algorithm 
solving the \emph{domino problem}, that is, taking as input a set of Wang tiles
and returning \textit{yes} or \textit{no} whether there exists a valid
configuration with these tiles. 
Berger, a student of Wang, later proved that the domino problem is undecidable
and he also provided a first example of an aperiodic set of Wang tiles
\cite{MR0216954}.
A set of Wang tiles $\T$ is \defn{aperiodic} if
the Wang shift $\Omega_\T$ is a nonempty aperiodic subshift.
This means that in general one can not decide the emptiness of a Wang shift
$\Omega_\T$. This illustrates that the behavior of $d$-dimensional SFTs when $d\geq2$
is much different than the one-dimensional case where emptiness of a SFT is
decidable \cite{MR1369092}.
Note that another important difference between $d=1$ and $d\geq2$
is expressed in terms of the possible values of entropy of $d$-dimensional SFTs,
see \cite{MR2680402}.

%\section{$\Z^2$-rotations on the 2-torus coded by Wang shifts}
\section{From toral partitions and $\Z^2$-rotations to Wang shifts}

We consider the $2$-torus $\generictorus=\R^2/\Gamma$ 
where $\Gamma$ is a lattice in $\R^2$.
We suppose that $(\generictorus,\Z^2,R)$ is a dynamical system
where $R$ is a toral $\Z^2$-rotation.
Let
$\mathcal{Y}=\{Y_i\}_{i\in I}$,
$\mathcal{Z}=\{Z_j\}_{j\in J}$ 
be two finite topological partitions of $\generictorus$.
For each $(i,j,k,\ell)\in I\times J\times I\times J$
we define the intersection of 4 atoms in the following way
\[
    P_{(i,j,k,\ell)} = Y_i \cap Z_j
                  \cap R^{\be_1}(Y_k)
                  \cap R^{\be_2}(Z_\ell)
\]
where $\be_1=(1,0)$ and $\be_2=(0,1)$.
The quadruples $\tau$ for which the intersection $P_\tau$ is nonempty
define a set
\[
    \T = \left\{\tau\in I\times J\times I\times J
                  \mid P_\tau \neq\varnothing\right\}
\]
that we see as a set of Wang tiles. Naturally, this comes with
a topological partition
\[
    \Pcal = \{P_\tau\}_{\tau\in\T}
\]
of $\generictorus$ which is the refinement of the four partitions
$\mathcal{Y}$ (the right color), 
$\mathcal{Z}$ (the top color),
$R^{\be_1}(\mathcal{Y})$ (the left color) and 
$R^{\be_2}(\mathcal{Z})$ (the bottom color).
Thus to each $\bx\in \generictorus\setminus\Delta_{\Pcal,R}$
corresponds a unique Wang tile, that is, a right, a top, a left and a bottom
color according to which atom it belongs in each of the four partitions.

\begin{proposition}\label{prop:is_wang_tiling}
    Let $(\generictorus,\Z^2,R)$ be a dynamical system
where $R$ is a $\Z^2$-rotation and
let $\mathcal{Y}$ and $\mathcal{Z}$ 
be two finite topological partitions of $\generictorus$.
Let $\Pcal=
\mathcal{Y}\wedge
\mathcal{Z}\wedge
R^{\be_1}(\mathcal{Y})\wedge
R^{\be_2}(\mathcal{Z})$ be the refinement of four partitions.
Let $\T$ be the set of Wang tiles defined above as the set of quadruples $\tau$
such that $P_\tau$ is a nonempty atom of the partition $\Pcal$.
Then $\Xcal_{\Pcal,R}$ is a subshift of the Wang shift $\Omega_\T$.
\end{proposition}

\begin{proof}
    Let $\bx\in \generictorus\setminus\Delta_{\Pcal,R}$
    and $w=\scSymbRep(\bx)$.
    Let $\bn\in\Z^2$. 
    First we check that Equation~\eqref{eq:1} is satisfied.
    There exists $i\in I$ such that 
    $R^\bn(\bx)\in Y_i$.
    Equivalently, $R^{\bn+\be_1}(\bx)\in R^{\be_1}(Y_i)$.
    Thus we have
    \begin{align*}
        \scright(w_\bn)
        &=\scright(\scConfig^{\Pcal,R}_{\bx}(\bn))
        = i\\
        &=\scleft(\scConfig^{\Pcal,R}_{\bx}(\bn+\be_1))
        =\scleft(w_{\bn+\be_1}).
    \end{align*}
    Similarly we check that Equation~\eqref{eq:2} is satisfied.
    There exists $j\in J$ such that 
    $R^\bn(\bx)\in Z_j$.
    Equivalently, $R^{\bn+\be_2}(\bx)\in R^{\be_2}(Z_j)$.
    Thus we have
    \begin{align*}
        \sctop(w_\bn)
        &=\sctop(\scConfig^{\Pcal,R}_{\bx}(\bn))
        = j\\
        &=\scbottom(\scConfig^{\Pcal,R}_{\bx}(\bn+\be_2))
        =\scbottom(w_{\bn+\be_2}).
    \end{align*}
Then the configuration $w$ is valid and $w\in\Omega_\T$.
Thus $\scSymbRep(\generictorus\setminus\Delta_{\Pcal,R}) \subseteq
    \Omega_\T$.

Remark that $\Omega_\T$ is closed since it is a subshift.
Therefore the topological closure of the image of $\scSymbRep$ is in the Wang
shift $\Omega_\T$. Using Lemma~\ref{lem:closure-of-tilings}, we conclude that
    \[
    \Xcal_{\Pcal,R} 
    =
    \overline{\scSymbRep(\generictorus\setminus\Delta_{\Pcal,R})}
    \subseteq \Omega_\T.\qedhere
    \]
\end{proof}

\begin{lemma}
If the refined partition $\Pcal=
\mathcal{Y}\wedge
\mathcal{Z}\wedge
R^{\be_1}(\mathcal{Y})\wedge
R^{\be_2}(\mathcal{Z})$ 
    gives a symbolic representation of $(\generictorus,\Z^2,R)$,
    then for every $\bv\in\R^2\setminus\Theta^\Pcal$, $\scSymbRep^\bv$ is a
    one-to-one map $\generictorus\to\Omega_\T$.
\end{lemma}

\begin{proof}
    Follows from Lemma~\ref{lem:tiling-one-to-one}
    and Proposition~\ref{prop:is_wang_tiling}.
\end{proof}

\section{Example 1: Jeandel-Rao aperiodic Wang shift}

Consider the lattice
$\Gamma_0=\langle (\varphi,0), (1,\varphi+3) \rangle_\Z$
where $\varphi=\frac{1+\sqrt{5}}{2}$.
On the torus $\R^2/\Gamma_0$, we consider the $\Z^2$-rotation
$R_0:\Z^2\times\R^2/\Gamma_0\to\R^2/\Gamma_0$
defined by
\[
    R_0^\bn(\bx):=R_0(\bn,\bx)=\bx+\bn
\]
for every $\bn\in\Z^2$.
We consider the fundamental domain $\Dcal =
[0,\varphi[\,\times\,[0,\varphi+3[$
of $\R^2$ for the group of translations~$\Gamma_0$.
% The lattice $\Gamma_0$ and the fundamental domain $\Dcal$ are illustrated in
% Figure~\ref{fig:fundamental_domain}.
Let $I=\{0,1,2,3\}$ and $J=\{0,1,2,3,4\}$ be sets of colors
and consider the partitions 
$\Ycal=\{Y_i\}_{i\in I}$ and $\Zcal=\{Z_j\}_{j\in J}$
shown in Figure~\ref{fig:jeandelrao-partition}.

\begin{figure}[h]
\begin{center}
\hspace{-3mm}
\begin{tikzpicture}[auto,scale=1.5]
\tikzstyle{every node}=[font=\footnotesize]
%\tikzstyle{every node}=[font=\LARGE]

% horizontal axis
\def\horizontalaxis{
\draw (-.05,0) -- (\p+.05,0);
\foreach \x/\y in
{0/0, 1/1,
    %2*\p-3/\frac{1}{\varphi^{3}},
    2-\p  /\frac{1}{\varphi^{2}},
    \p-1  /\frac{1}{\varphi},
    \p    /\varphi}
\draw (\x,0+.05) -- (\x,0-.05) node[below] {$\y$};}
% silent vertical axis
\def\silentverticalaxis{
\draw[-latex] (0,-.1) -- (0,3+\p+.2);
\foreach \x/\y in
{0/0, 1/1, 2/2,
    \p+1  /\varphi+1,
    \p+2  /\varphi+2,
    \p+3  /\varphi+3}
\draw (.05,\x) -- (-.05,\x);}% node[left] {$\y$};}

% vertical axis
\begin{scope}[xshift=-.0cm]
    \draw[-latex] (0,-.1) -- (0,3+\p+.2);
    \draw (.05,0)    -- (-.05,0)    node[left] {$1$};
    \draw (.05,1)    -- (-.05,1)    node[left] {$2$};
    \draw (.05,2)    -- (-.05,2)    node[left] {$3$};
    \draw (.05,\p+1) -- (-.05,\p+1) node[left] {$\varphi\hspace{-1mm}+\hspace{-1mm}1$};
    \draw (.05,\p+2) -- (-.05,\p+2) node[left] {$\varphi\hspace{-1mm}+\hspace{-1mm}2$};
    \draw (.05,\p+3) -- (-.05,\p+3) node[left] {$\varphi\hspace{-1mm}+\hspace{-1mm}3$};
\end{scope}

\begin{scope}[xshift=0cm]
    \node[above,align=center] at (.8,3+\p+.1) {$\Ycal$};
    \fill[cyan] (0,0) rectangle (\p,1);
    \fill[red] (\p-1,1) -- (\p,2+\p) -- (\p,1) -- cycle;
    \fill[red] (0,1) -- (\p-1,2) -- (\p-1,1) -- cycle;
    \fill[red] (0,2) -- (0,1+\p) -- (\p-1,2+\p) -- cycle;
    \fill[green] (0,1) -- (0,2) -- (1,3+\p) -- (1,2+\p) -- (\p,3+\p) -- (\p,2+\p) -- cycle;
    \draw[thick] (0,0) rectangle (\p,3+\p);
    \draw[thick] (\p-1,2) -- (\p-1,1) -- (\p,2+\p) -- (0,1) -- (\p,1);
    \draw[thick] (0,2) -- (1,3+\p) -- (1,2+\p) -- (\p,3+\p);
    \draw[thick] (0,1+\p) -- (\p-1,2+\p);
    %%% H separations
    %\foreach \x in {1,2,1+\p,2+\p}
    %    \draw[thick,dotted] (0,\x) -- (\p,\x);
    % horizontal axis
    \horizontalaxis
    \silentverticalaxis
    % labels
    \node at (0.81,0.5) {2};
    \node at (1.26,1.5) {1};
    \node at (0.37,1.3) {1};
    \node at (0.90,2.25) {0};
    \node at (0.10,2.5) {1};
    \node at (0.95,3.1) {3};
    \node at (0.29,4.1) {0};
    \node at (1.19,4.4) {0};
\end{scope}

\begin{scope}[xshift=1.8cm]
    \node[above,align=center] at (.8,3+\p+.1) {$\Zcal$};
    \fill[lightgray]  (0,0) -- (0,1) -- (0,0) -- (\p-1,1) --
                (\p-1,0) -- (1,1) --
                (1,0) -- (\p,1) -- (\p,0) -- cycle;
    \fill[cyan] (0,1) -- (0,0) -- (\p-1,1) --
                (\p-1,0) -- (1,1) --
                (1,0) -- (\p,1) -- (\p,2+\p) -- cycle;
    \fill[green] (0,1) -- (1,2+\p) -- (1,1+\p) -- cycle;
    \fill[red] (0,1) -- (0,2+\p) -- (2-\p,3+\p) -- (2-\p,2+\p) --
                        (1,3+\p) -- (1,2+\p) -- cycle;
    \fill[red] (1,1+\p) -- (1,2+\p) -- (\p,3+\p) -- (\p,2+\p) -- cycle;
    \draw[thick] (0,1) -- (0,0) -- (\p-1,1) --
                (\p-1,0) -- (1,1) --
                (1,0) -- (\p,1) -- (\p,2+\p) -- cycle;
    \draw[thick] (0,2+\p) -- (2-\p,3+\p) -- (2-\p,2+\p) --
                        (1,3+\p) -- (1,2+\p) -- (0,1);
    \draw[thick] (1,1+\p) -- (1,2+\p) -- (\p,3+\p) -- (\p,2+\p) -- cycle;
    %contour
    \draw[thick] (0,0) rectangle (\p,3+\p);
    %% H separations
    %\foreach \x in {1,2,1+\p,2+\p}
    %    \draw[thick,dotted] (0,\x) -- (\p,\x);
    % horizontal axis
    \horizontalaxis
    \silentverticalaxis
    % labels
    \node at (1.32,.2) {4};
    \node at (0.84,.2) {4};
    \node at (0.37,.2) {4};
    \node at (0.81,1.5) {2};
    \node at (1.22,3.43) {1};
    \node at (0.45,3.21) {1};
    \node at (1.30,4.41) {0};
    \node at (0.64,4.41) {0};
    \node at (0.14,4.41) {0};
    \node at (0.64,2.32) {3};
\end{scope}

\begin{scope}[xshift=3.6cm]
    \node[above,align=center] at (.8,3+\p+.1) {$R_0^{\be_1}(\Ycal)$};
    \fill[cyan] (0,0) rectangle (\p,1);
    \fill[red] (0,1) -- (1,2+\p) -- (1,1) -- cycle;
    \fill[red] (1,1) -- (\p,2) -- (\p,1) -- cycle;
    \fill[red] (1,2) -- (1,1+\p) -- (\p,2+\p) -- cycle;
    \fill[green] (1,1) -- (1,2) -- (\p,2+\p) -- (\p,2) -- cycle;
    \fill[green] (0,2+\p) -- (2-\p,3+\p) -- (2-\p,2+\p) -- (1,3+\p) -- (1,2+\p) -- (0,2) -- cycle;
    \draw[thick] (0,2+\p) -- (2-\p,3+\p) -- (2-\p,2+\p) -- (1,3+\p)-- (1,2+\p) -- (0,2);
    \draw[thick] (0,1) -- (\p,1) -- (\p,2) -- (1,1) -- (1,2) 
    -- (\p,2+\p) -- (1,1+\p) -- (1,2+\p) -- cycle;
    %contour
    \draw[thick] (0,0) rectangle (\p,3+\p);
    %% H separations
    %\foreach \x in {1,2,1+\p,2+\p}
    %    \draw[thick,dotted] (0,\x) -- (\p,\x);
    % horizontal axis
    \horizontalaxis
    \silentverticalaxis
    % labels
    \node at (0.81,0.5) {2};
    \node at (0.64,1.5) {1};
    \node at (1.37,1.3) {1};
    \node at (1.34,2.25) {3};
    \node at (0.29,2.20) {0};
    \node at (0.34,3.1) {3};
    \node at (1.29,4.1) {0};
    \node at (0.58,4.4) {0};
    \node at (0.16,4.4) {0};
\end{scope}

\begin{scope}[xshift=5.4cm]
    \node[above,align=center] at (.8,3+\p+.1) {$R_0^{\be_2}(\Zcal)$};
    \fill[red]  (0,0) -- (0,1) -- (0,0) -- (\p-1,1) --
                (\p-1,0) -- (1,1) --
                (1,0) -- (\p,1) -- (\p,0) -- cycle;
    \fill[lightgray]  (0,1) -- (0,2) -- (0,1) -- (\p-1,2) --
                (\p-1,1) -- (1,2) --
                (1,1) -- (\p,2) -- (\p,1) -- cycle;
    \fill[cyan] (0,2) -- (0,1) -- (\p-1,2) --
                (\p-1,1) -- (1,2) --
                (1,1) -- (\p,2) -- (\p,3+\p) -- cycle;
    \fill[green] (0,2) -- (1,3+\p) -- (1,2+\p) -- cycle;
    \fill[red] (0,2) -- (0,3+\p) -- (1,3+\p) -- cycle;
    \fill[red] (1,2+\p) -- (1,3+\p) -- (\p,3+\p) -- cycle;
    \draw[thick]  (0,0) -- (0,1) -- (0,0) -- (\p-1,1) --
                (\p-1,0) -- (1,1) --
                (1,0) -- (\p,1);
    \draw[thick] (0,2) -- (0,1) -- (\p-1,2) --
                (\p-1,1) -- (1,2) --
                (1,1) -- (\p,2) -- (\p,3+\p) -- cycle;
    \draw[thick] (0,2) -- (1,3+\p) -- (1,2+\p);
    \draw[thick] (0,1) -- (\p,1);
    %contour
    \draw[thick] (0,0) rectangle (\p,3+\p);
    %% H separations
    %\foreach \x in {1,2,1+\p,2+\p}
    %    \draw[thick,dotted] (0,\x) -- (\p,\x);
    % horizontal axis
    \horizontalaxis
    \silentverticalaxis
    % labels
    \node at (1.32,0.2) {1};
    \node at (0.84,0.2) {1};
    \node at (0.37,0.2) {1};
    \node at (1.32,1.2) {4};
    \node at (0.84,1.2) {4};
    \node at (0.37,1.2) {4};
    \node at (1.13,2.8) {2};
    \node at (1.22,4.43) {1};
    \node at (0.45,4.21) {1};
    \node at (1.30,0.79) {0};
    \node at (0.79,0.79) {0};
    \node at (0.34,0.79) {0};
    \node at (0.64,3.32) {3};
\end{scope}

\begin{scope}[xshift=7.2cm]
    \node[above,align=center] at (.8,3+\p+.1) {$\Pcal_0$};

    %left
    \draw[thick] (0,2+\p) -- (2-\p,3+\p) -- (2-\p,2+\p) -- (1,3+\p)-- (1,2+\p) -- (0,2);
    \draw[thick] (0,1) -- (\p,1) -- (\p,2) -- (1,1) -- (1,2) 
    -- (\p,2+\p) -- (1,1+\p) -- (1,2+\p) -- cycle;
    % right
    \draw[thick] (0,0) rectangle (\p,3+\p);
    \draw[thick] (\p-1,2) -- (\p-1,1) -- (\p,2+\p) -- (0,1) -- (\p,1);
    \draw[thick] (0,2) -- (1,3+\p) -- (1,2+\p) -- (\p,3+\p);
    \draw[thick] (0,1+\p) -- (\p-1,2+\p);
    % top
    \draw[thick] (0,1) -- (0,0) -- (\p-1,1) --
                (\p-1,0) -- (1,1) --
                (1,0) -- (\p,1) -- (\p,2+\p) -- cycle;
    \draw[thick] (0,2+\p) -- (2-\p,3+\p) -- (2-\p,2+\p) --
                        (1,3+\p) -- (1,2+\p) -- (0,1);
    \draw[thick] (1,1+\p) -- (1,2+\p) -- (\p,3+\p) -- (\p,2+\p) -- cycle;
    % bottom
    \draw[thick]  (0,0) -- (0,1) -- (0,0) -- (\p-1,1) --
                (\p-1,0) -- (1,1) --
                (1,0) -- (\p,1);
    \draw[thick] (0,2) -- (0,1) -- (\p-1,2) --
                (\p-1,1) -- (1,2) --
                (1,1) -- (\p,2) -- (\p,3+\p) -- cycle;
    \draw[thick] (0,2) -- (1,3+\p) -- (1,2+\p);
    %contour
    \draw[thick] (0,0) rectangle (\p,3+\p);
    %% H separations
    %\foreach \x in {1,2,1+\p,2+\p}
    %    \draw[thick,dotted] (0,\x) -- (\p,\x);

    \horizontalaxis
    \silentverticalaxis

%r=2
\node at (0.14,4.36) {$t_6$}; %{\JeandelRaoVI};
\node at (0.64,4.36) {$t_6$}; %{\JeandelRaoVI};
\node at (1.22,4.36) {$t_6$}; %{\JeandelRaoVI};
\node at (1.29,3.62) {$t_7$}; %{\JeandelRaoVII};
%r=3
\node at (0.24,3.43) {$t_5$}; %{\JeandelRaoV};
\node at (0.77,3.62) {$t_4$}; %{\JeandelRaoIV};
%r=4
\node at (0.14,2.65) {$t_2$}; %{\JeandelRaoII};
\node at (0.81,2.65) {$t_{10}$}; %{\JeandelRaoX};
\node at (0.82,2.00) {$t_8$}; %{\JeandelRaoVIII};
%r=0 + var
\node at (0.18,2.00) {$t_7$}; %{\JeandelRaoVII};
\node at (1.26,2.00) {$t_3$}; %{\JeandelRaoIII};
\node at (0.37,1.25) {$t_9$}; %{\JeandelRaoIX};
\node at (0.84,1.25) {$t_9$}; %{\JeandelRaoIX};
\node at (1.42,1.25) {$t_9$}; %{\JeandelRaoIX};
%r=1 + var
\node at (0.23,0.75) {$t_1$}; %{\JeandelRaoI};
\node at (0.77,0.75) {$t_1$}; %{\JeandelRaoI};
\node at (1.26,0.75) {$t_1$}; %{\JeandelRaoI};
\node at (0.37,0.25) {$t_0$}; %{\JeandelRaoO};
\node at (0.84,0.25) {$t_0$}; %{\JeandelRaoO};
\node at (1.42,0.25) {$t_0$}; %{\JeandelRaoO};
\end{scope}
\end{tikzpicture}
\end{center}
\caption{Partitions for the 11 Jeandel-Rao Wang tiles.
    From left to right, the partition 
    $\Ycal$ for the right color,
    $\Zcal$ for the top color,
    $R_0^{\be_1}(\Ycal)$ for the left color and
    $R_0^{\be_2}(\Zcal)$ for the bottom color.
    Their refinement is the partition $\Pcal_0$ where each part is associated
    with one of Jeandel-Rao Wang tiles.}
\label{fig:jeandelrao-partition}
\end{figure}

The refined partition is $\Pcal_0=
\mathcal{Y}\wedge
\mathcal{Z}\wedge
R_0^{\be_1}(\mathcal{Y})\wedge
R_0^{\be_2}(\mathcal{Z})=\{P_t\}_{t\in I\times J\times I\times J}$.
The set of quadruples $(i,j,k,\ell)$ such that
    $P_{(i,j,k,\ell)} = Y_i \cap Z_j
                  \cap R^{\be_1}(Y_k)
                  \cap R^{\be_2}(Z_\ell)$ is nonempty is
\[
\begin{array}{c}
    \T_0= \left\{ 
(2, 4, 2, 1),
(2, 2, 2, 0),
(1, 1, 3, 1),
(1, 2, 3, 2),
(3, 1, 3, 3),
(0, 1, 3, 1),
\right.\\
\hspace{1cm}\left.
(0, 0, 0, 1),
(3, 1, 0, 2),
(0, 2, 1, 2),
(1, 2, 1, 4),
(3, 3, 1, 2)
 \right\}
\end{array}
\]
which can be seen as a set of Wang tiles
\begin{equation}\label{eq:jeandel_rao_tile_set}
\begin{array}{c}
    \T_0= \left\{ 
t_0=\hspace{-2mm}\raisebox{-3mm}{\JeandelRaoO},
t_1=\hspace{-2mm}\raisebox{-3mm}{\JeandelRaoI},
t_2=\hspace{-2mm}\raisebox{-3mm}{\JeandelRaoII},
t_3=\hspace{-2mm}\raisebox{-3mm}{\JeandelRaoIII},
t_4=\hspace{-2mm}\raisebox{-3mm}{\JeandelRaoIV},
t_5=\hspace{-2mm}\raisebox{-3mm}{\JeandelRaoV},
\right.\\
\hspace{1cm}\left.
t_6=\raisebox{-3mm}{\JeandelRaoVI},
t_7=\raisebox{-3mm}{\JeandelRaoVII},
t_8=\raisebox{-3mm}{\JeandelRaoVIII},
t_9=\raisebox{-3mm}{\JeandelRaoIX},
t_{10}=\raisebox{-3mm}{\JeandelRaoX}
 \right\}.
\end{array}
\end{equation}
We observe that $\T_0$ is Jeandel-Rao's set of 11 tiles
\cite{jeandel_aperiodic_2015}.
Let $\Omega_0=\Omega_{\T_0}$ be the Jeandel-Rao Wang shift.
We may now prove Theorem~\ref{thm:JR-max-equi-factor} which follows mostly from the work done in the Part~\ref{part:1}.

\begin{proof}[Proof of Theorem~\ref{thm:JR-max-equi-factor}]
    (i)
The dynamical system $(\R^2/\Gamma_0,\Z^2,R_0)$ is minimal.
Since $R_0^{\be_1}$ and $R_0^{\be_2}$ are linearly independent irrational
rotations on $\R^2/\Gamma_0$, we have
that $R_0$ is a free $\Z^2$-action.
Thus from Lemma~\ref{lem:minimal-aperiodic},
$\Xcal_{\Pcal_0,R_0}$ is minimal and aperiodic.
From Proposition~\ref{prop:is_wang_tiling} and
from Equation~\eqref{eq:jeandel_rao_tile_set},
we have $\Xcal_{\Pcal_0,R_0}\subseteq\Omega_{0}$.
It was proved in \cite{labbe_substitutive_2018_with_doi} that the Jeandel-Rao Wang shift
$\Omega_0$ is not minimal. 
Thus $\Omega_0\setminus\Xcal_{\Pcal_0,R_0}$ is nonempty.

(ii)
    The atom $P_{t_{10}}$ is invariant only under the trivial
translation.
Therefore, from Lemma~\ref{lem:symbolic-representation},
$\Pcal_0$ gives a symbolic representation of 
$(\R^2/\Gamma_0,\Z^2,R_0)$.

    (iii)
From Proposition~\ref{prop:factor-map}, there exists a
factor map $f_0$ from $(\Xcal_{\Pcal_0,R_0},\Z^2,\sigma)$ to
$(\R^2/\Gamma_0,\Z^2,R_0)$ and 
from Corollary~\ref{cor:max-equicontinuous-factor},
$(\R^2/\Gamma_0, \Z^2,R_0)$ is the maximal equicontinuous
factor of $(\Xcal_{\Pcal_0,R_0},\Z^2,\sigma)$.

    (iv)
From Proposition~\ref{prop:factor-map}, we have
that $f_0$ is one-to-one on $\torus\setminus\Delta_{\Pcal_0,R_0}$.
Suppose that $\bx\in\Delta_{\Pcal_0,R_0}$.
We have $\card(f_0^{-1}(\bx))\geq2$.
If $\card(f_0^{-1}(\bx))>2$, then we may show that
    there exists $\bn\in\Z^2$ such that $\bx=R_0^\bn(\zero)$.
    If $\bx=R_0^{\bn}(\zero)$ for some $\bn\in\Z^2$, then
the set $f_0^{-1}(\bx)$ contains 8 different configurations
    of the form $\scSymbRep_0^\bv(0)$ for some $\bv\in\R^2\setminus\Theta^{\Pcal_0}$
    where $\Theta^{\Pcal_0}=\R\cdot\{(1,0),(0,1),(1,\varphi),(1,\varphi^2)\}$.
If $\bx\in\Delta_{\Pcal_0,R_0}$ but is not
    in the orbit of $\zero$ under $R_0$,
    then $\card(f_0^{-1}(\bx))=2$.
    We conclude that
    $\{\card(f_0^{-1}(\bx))\mid\bx\in\R^2/\Gamma_0\}=\{1,2,8\}$.

(v)
    We have that $\lambda(\partial P)=0$ for each atom $P\in\Pcal_0$
    where $\lambda$ is the Haar measure on $\R^2/\Gamma_0$.
    The result follows from Proposition~\ref{prop:measurable-conjugacy}.
\end{proof}

The frequency of any pattern $p$ in $\Xcal_{\Pcal_0,R_0}$ is equal to
the measure of the associated cylinder $[p]$ in $\Xcal_{\Pcal_0,R_0}$
which is equal to the Haar measure of $f_0([p])$ in $\R^2/\Gamma_0$
and can be computed using Equation~\eqref{eq:cylindre-intersection-polygones}
as the area of the coding region which is the intersection of polygons. 
Here is what it gives for the frequencies of tiles in $\Xcal_{\Pcal_0,R_0}$.

\begin{proposition}
    The frequencies of each of the 11 Jeandel-Rao tiles $t_i$ for
    $i\in\{0,\dots,10\}$ in the subshift $\Xcal_{\Pcal_0,R_0}$ is
    given by the measure of the cylinders below:
\begin{align*}
    \nu([7])                      &= 5/(12\varphi+14) &\approx 0.1496,\\
    \nu([0])= \nu([1])= \nu([3])= \nu([6])= \nu([9]) 
                                  &= 1/(2\varphi+6)   &\approx 0.1083,\\
    \nu([5])                      &= 1/(5\varphi+4)   &\approx 0.0827,\\
    \nu([4])=\nu([8])=\nu([10])   &= 1/(8\varphi+2)   &\approx 0.0669,\\
    \nu([2])                      &= 1/(18\varphi+10) &\approx 0.0256.
\end{align*}
\end{proposition}

%sage: v = [1/(8*phi+2), 1/(2*phi+6), 1/(5*phi+4), 1/(18*phi+10), 5/(12*phi+14)]
%sage: v = vector(v)
%sage: v * vector((3,5,1,1,1))
%1
%sage: map(n, v)
%[0.0669152706817991,
% 0.108271182329550,
% 0.0827118232955023,
% 0.0255593590340479,
% 0.149627093977301]

\begin{proof}
Thanks to Theorem~\ref{thm:JR-max-equi-factor} (v), 
it can be computed from the area of atoms of the partition $\Pcal_0$ shown in
Figure~\ref{fig:jeandelrao-partition} and dividing by the area of the
fundamental domain which is $\varphi(\varphi+3)=4\varphi+1$.
\end{proof}

We may check that the frequency of each tile in the minimal subshift
$\Xcal_{\Pcal_0,R_0}\subset\Omega_0$ match those obtained in
\cite{labbe_substitutive_2018_with_doi} 
for some minimal subshift $X_0\subset\Omega_0$
and computed from the substitutive structure of $X_0$.
% and the Perron right-eigenvector of the incidence matrix of the
% self-similarity of $\Omega_\U$.
In fact, 
$X_0=\Xcal_{\Pcal_0,R_0}$ but we postpone the proof of this
in a later work \cite{labbe_induction_2019} in which the substitutive structure of
$\Xcal_{\Pcal_0,R_0}$ is described using Rauzy induction of toral partitions
and $\Z^2$-rotations.
In \cite{labbe_substitutive_2018_with_doi}, we also proved that $X_0$ is a shift of finite type
as it can be described by the forbidden patterns coming from $\Omega_0$ plus 
a finite number of other forbidden patterns.
The completion of the proof of the equality
$X_0=\Xcal_{\Pcal_0,R_0}$ 
will imply that the following statement holds, but we can only state it as a
conjecture for now.

\begin{conjecture}
    $\Pcal_0$ is a Markov partition for $(\R^2/\Gamma_0,\Z^2,R_0)$.
\end{conjecture}

%\section{The 19 self-similar Wang tiles example}
%\section{Self-similar Wang tilings in $\Omega_\U$ as the coding of a $\Z^2$-action on a 2-torus}
%\section{A Markov partition for Wang tilings in $\Omega_\U$}
\section{Example 2: A minimal aperiodic Wang shift defined by 19 tiles}

On the torus $\torus=\R^2/\Z^2$, we consider the $\Z^2$-rotation
$R_\U:\Z^2\times\torus\to\torus$
defined by
\[
    R_\U^\bn(\bx):=R_\U(\bn,\bx)=\bx+\varphi^{-2}\bn
\]
for every $\bn\in\Z^2$ where $\varphi=\frac{1+\sqrt{5}}{2}$.
Let $I=\{\mathrm{A,B,C,D,E,F,G,H,I,J}\}$ and $J=\{\mathrm{K,L,M,N,O,P}\}$ be sets of colors
and consider the partitions 
$\Ycal=\{Y_i\}_{i\in I}$ and $\Zcal=\{Z_j\}_{j\in J}$
shown in Figure~\ref{fig:19-selfsimilar-partitions}.

The refined partition is $\Pcal_\U=
\mathcal{Y}\wedge
\mathcal{Z}\wedge
R_\U^{\be_1}(\mathcal{Y})\wedge
R_\U^{\be_2}(\mathcal{Z})=\{P_u\}_{u\in I\times J\times I\times J}$.
Let $\U$ be the set of quadruples $(i,j,k,\ell)$ such that
    $P_{(i,j,k,\ell)} = Y_i \cap Z_j
                  \cap R_\U^{\be_1}(Y_k)
                  \cap R_\U^{\be_2}(Z_\ell)$ is nonempty.
We represent $\U$
as a set of Wang tiles, see
Figure~\ref{fig:19-wang-tile-set}. 
It corresponds to the set of 19 Wang tiles
$\U=\{u_0,u_1,\dots,u_{18}\}$ 
introduced in \cite{MR3978536}
which was derived from the substitutive structure of the Jeandel-Rao Wang shift
\cite{labbe_substitutive_2018_with_doi}.

\begin{figure}[h]
\begin{center}
\begin{tabular}{ccccc}
$\Ycal$ & 
$\Zcal$ & 
$R_\U^{\be_1}(\Ycal)$ & 
$R_\U^{\be_2}(\Zcal)$ &
$\Pcal_\U$\\
\hspace{-3mm}\includegraphics[scale=.91]{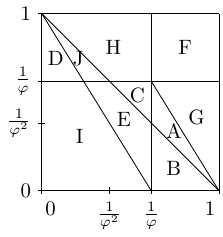}&
\hspace{-3mm}\includegraphics[scale=.91]{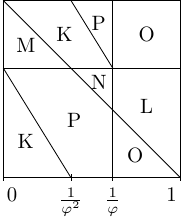}&
\hspace{-3mm}\includegraphics[scale=.91]{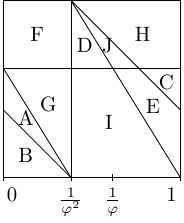}&
\hspace{-3mm}\includegraphics[scale=.91]{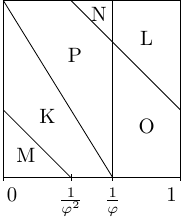}&
\hspace{-3mm}\includegraphics[scale=.91]{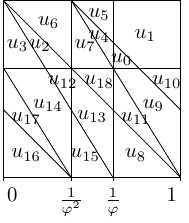}
\end{tabular}
\end{center}
    \caption{Partitions for the set $\U$ of 19 Wang tiles.
    From left to right, the partition 
    $\Ycal$ for the right color,
    $\Zcal$ for the top color,
    $R_\U^{\be_1}(\Ycal)$ for the left color and
    $R_\U^{\be_2}(\Zcal)$ for the bottom color.
    Their refinement is the partition $\Pcal_\U$ where each part is associated
    with a Wang tile.}
    \label{fig:19-selfsimilar-partitions}
\end{figure}

\begin{figure}[h]
\begin{center}
\input{article3_tile_set_U.tex}
\end{center}
    \caption{The set $\U=\{u_0,\dots,u_{18}\}$ of 19 Wang tiles.
    Each index $i\in\{0, ..., 18\}$ written in the middle of a tile corresponds
    to a tile $u_i$.}
    \label{fig:19-wang-tile-set}
\end{figure}

Let $\Omega_\U$ be the Wang shift associated with the set of Wang tiles $\U$.
We now prove the second theorem of the article.
Note that the fact that $\Omega_\U$ is minimal 
(proved in \cite{MR3978536}) allows to conclude that $\Pcal_\U$
is a Markov partition
without having to exploit the substitutive structure of
$\Xcal_{\Pcal_\U,R_\U}$ which is computed in
\cite{labbe_induction_2019}.

\begin{proof}
    [Proof of Theorem~\ref{thm:OmegaU-partition}]
(i)
The dynamical system $(\torus,\Z^2,R_\U)$ is minimal.
Since $R_\U^{\be_1}$ and $R_\U^{\be_2}$ are linearly independent irrational
rotations, we have
that $R_\U$ is a free $\Z^2$-action.
Thus from Lemma~\ref{lem:minimal-aperiodic},
$\Xcal_{\Pcal_\U,R_\U}$ is minimal and aperiodic.
From Proposition~\ref{prop:is_wang_tiling},
we have $\Xcal_{\Pcal_\U,R_\U}\subseteq\Omega_{\U}$.
It was proved in \cite{MR3978536} that $\Omega_\U$ is minimal. 
Thus $\Xcal_{\Pcal_\U,R_\U}=\Omega_\U$.

(ii)
The atom $P_{u_0}$ is invariant only under the trivial
translation.
Therefore, from Lemma~\ref{lem:symbolic-representation},
$\Pcal_\U$ gives a symbolic representation of 
$(\torus,\Z^2,R_\U)$.
Moreover $\Xcal_{\Pcal_\U,R_\U}=\Omega_\U$ is a shift of finite type.
Therefore, the conditions of Definition~\ref{def:Markov} are satisfied and
    $\Pcal_\U$ is a Markov partition for the dynamical system
$(\torus,\Z^2,R_\U)$.

(iii)
From Proposition~\ref{prop:factor-map}, there exists a
factor map $f_\U$ from $(\Xcal_{\Pcal_\U,R_\U},\Z^2,\sigma)$ to
$(\torus,\Z^2,R_\U)$ and 
from Corollary~\ref{cor:max-equicontinuous-factor},
$(\torus, \Z^2,R_\U)$ is the maximal equicontinuous
factor of $(\Xcal_{\Pcal_\U,R_\U},\Z^2,\sigma)$.

    (iv)
In Proposition~\ref{prop:factor-map}, we proved
that $f_\U$ is one-to-one on $\torus\setminus\Delta_{\Pcal_\U,R_\U}$.
Suppose that $\bx\in\Delta_{\Pcal_\U,R_\U}$.
We have $\card(f_\U^{-1}(\bx))\geq2$.
If $\card(f_\U^{-1}(\bx))>2$, then we may show that
    there exists $\bn\in\Z^2$ such that $\bx=R_\U^\bn(\zero)$.
    If $\bx=R_\U^{\bn}(\zero)$ for some $\bn\in\Z^2$, then
the set $f_\U^{-1}(\bx)$ contains 8 different configurations
    of the form $\scSymbRep_\U^\bv(0)$ for some $\bv\in\R^2\setminus\Theta^{\Pcal_\U}$
    where $\Theta^{\Pcal_\U}=\R\cdot\{(1,0),(0,1),(1,-1),(1,-\varphi)\}$.
If $\bx\in\Delta_{\Pcal_\U,R_\U}$ but not
    in the orbit of $\zero$ under $R_\U$,
    then $\card(f_\U^{-1}(\bx))=2$.
    We conclude that
    $\{\card(f_\U^{-1}(\bx))\mid\bx\in\torus\}=\{1,2,8\}$.

(v)
    We have that $\lambda(\partial P)=0$ for each atom $P\in\Pcal_\U$
    where $\lambda$ is the Haar measure on $\torus$.
    The result follows from Proposition~\ref{prop:measurable-conjugacy}.
\end{proof}

\section{Two non-examples}

In this section, we present two ``non-examples''.
The motivation for presenting those two ``non-examples'' is to illustrate that
properties of partitions presented in the previous sections are not shared by
``randomly'' chosen partitions of $\torus$ and $\Z^2$-rotations on $\torus$.

\subsection*{Example 3}

Let $\balpha,\bbeta\in\R^2$.
On the torus $\torus=\R^2/\Z^2$, we consider the $\Z^2$-rotation
$R:\Z^2\times\torus\to\torus$
defined by
\[
    R^\bn(\bx):=R(\bn,\bx)=\bx+n_1\balpha + n_2\bbeta
\]
for every $\bn=(n_1,n_2)\in\Z^2$.
Let $\Ycal=\{Y_A\}$ and $\Zcal=\{Z_B\}$ be trivial partitions
with $Y_A=Z_B=\torus$.
The refined partition is $\Pcal=
\mathcal{Y}\wedge
\mathcal{Z}\wedge
R^{\be_1}(\mathcal{Y})\wedge
R^{\be_2}(\mathcal{Z})=\{P_{(A,B,A,B)}\}$
where $P_{(A,B,A,B)}=\torus$.
The set of Wang tiles $\T=\{\tau\}$ is a singleton set with
$\tau=(A,B,A,B)$.
The associated color partitions
and the tile coding partition are shown in
Figure~\ref{fig:example0-partitions}.
\begin{figure}[h]
\begin{center}
\def\t{0.02}   % ticks length
% horizontal axis macro
\def\Haxis{
\foreach \x/\label/\xshift in
{0/0/.15cm, 1/1/-.15cm}
\draw (\x,\t) -- (\x,-\t) node[below,xshift=\xshift] {$\label$};}
% vertical axis macro
\def\Vaxis{
\foreach \x/\label/\xshift in
{0/0/.15cm, 1/1/-.15cm}
\draw (\t,\x) -- (-\t,\x) node[left,yshift=\xshift] {$\label$};}
% tabular
\begin{tabular}{ccccc}
$\Ycal$ & 
$\Zcal$ & 
$R^{\be_1}(\Ycal)$ & 
$R^{\be_2}(\Zcal)$ &
$\Pcal$\\
    \begin{tikzpicture}[scale=2.2]
        \Haxis\Vaxis
        \draw (0,0) rectangle node {A} (1,1);
    \end{tikzpicture}&
    \begin{tikzpicture}[scale=2.2]
        \Haxis
        \draw (0,0) rectangle node {B} (1,1);
    \end{tikzpicture}&
    \begin{tikzpicture}[scale=2.2]
        \Haxis
        \draw (0,0) rectangle node {A} (1,1);
    \end{tikzpicture}&
    \begin{tikzpicture}[scale=2.2]
        \Haxis
        \draw (0,0) rectangle node {B} (1,1);
    \end{tikzpicture}&
    \begin{tikzpicture}[scale=2.2]
        \Haxis
        \draw (0,0) rectangle node {
$\tau=\raisebox{-3mm}{
\begin{tikzpicture}
[scale=0.900000000000000]
% tile at position (x,y)=(0, 0)
\draw (0, 0) -- (1, 0);
\draw (0, 0) -- (0, 1);
\draw (1, 1) -- (1, 0);
\draw (1, 1) -- (0, 1);
\node[rotate=0,font=\tiny] at (0.800000000000000, 0.5) {A};
\node[rotate=0,font=\tiny] at (0.5, 0.800000000000000) {B};
\node[rotate=0,font=\tiny] at (0.200000000000000, 0.5) {A};
\node[rotate=0,font=\tiny] at (0.5, 0.200000000000000) {B};
\end{tikzpicture}}$} (1,1);
    \end{tikzpicture}
\end{tabular}
\end{center}
    \caption{Partitions for the Example 3.
    From left to right, the partition 
    $\Ycal$ for the right color,
    $\Zcal$ for the top color,
    $R^{\be_1}(\Ycal)$ for the left color and
    $R^{\be_2}(\Zcal)$ for the bottom color.
    Their refinement is the trivial partition $\Pcal$ whose single atom is
    associated with the Wang tile~$\tau$.}
    \label{fig:example0-partitions}
\end{figure}

The map $\scSymbRep:\torus\to\Omega_{\T}$ is clearly not one-to-one, but it
is onto.

\begin{lemma}
    We have $\Xcal_{\Pcal,R}=\Omega_{\T}$,
    but the partition $\Pcal$ does not give a symbolic representation
of $(\torus,\Z^2,R)$.
\end{lemma}

\begin{proof}
The set of Wang tiles $\T=\{\tau\}$ is a singleton set with
$\tau=(A,B,A,B)$. Therefore $\Omega_{\T}$ contains a unique configuration
corresponding to the constant map $(m,n)\mapsto\tau$ for all $m,n\in\Z$.
The fact that $\Xcal_{\Pcal,R}\subseteq\Omega_{\T}$ 
    follows from Proposition~\ref{prop:is_wang_tiling}.
The unique constant configuration in $\Omega_{\T}$ can be obtained as
$\scSymbRep(\bx)=\scConfig^{\Pcal,R}_{\bx}$ for any $\bx\in\torus$.
Therefore $\scSymbRep$ is onto.

    The partition $\Pcal$ does not give a symbolic representation
of $(\torus,\Z^2,R)$ as every point of $\torus$ is associated with the same
    configuration.
\end{proof}

\subsection*{Example 4}

Let $\varphi=\frac{1+\sqrt{5}}{2}$.
On the torus $\torus=\R^2/\Z^2$, we consider the $\Z^2$-rotation
$R:\Z^2\times\torus\to\torus$
defined by
\[
    R^\bn(\bx):=R(\bn,\bx)=\bx+\varphi\,\bn
\]
for every $\bn\in\Z^2$.
Let $I=\{A,B\}$ and $J=\{C,D\}$ be sets of colors.
We consider the partitions 
$\Ycal=\{Y_A,Y_B\}$ and $\Zcal=\{Z_C,Z_D\}$
shown in Figure~\ref{fig:example1-partitions} involving slopes $1$ and $-1$ in
the partition of $\torus$ into polygons.
The refined partition is $\Pcal=
\mathcal{Y}\wedge
\mathcal{Z}\wedge
R^{\be_1}(\mathcal{Y})\wedge
R^{\be_2}(\mathcal{Z})=\{P_\tau\}_{\tau\in\T}$
where
$\T$ is the set of Wang tiles
made of 20 tiles
shown in Figure~\ref{fig:boring20tiles}.
\begin{figure}[h]
\begin{center}
\begin{tabular}{ccccc}
$\Ycal$ & 
$\Zcal$ & 
$R^{\be_1}(\Ycal)$ & 
$R^{\be_2}(\Zcal)$ &
$\Pcal$\\
    \hspace{-3mm}\includegraphics[scale=.92]{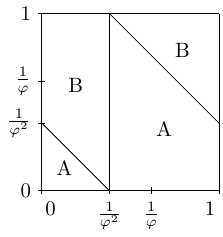}&
    \hspace{-3mm}\includegraphics[scale=.92]{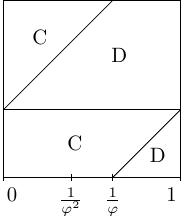}&
    \hspace{-3mm}\includegraphics[scale=.92]{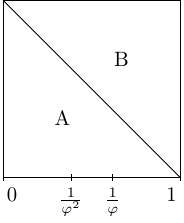}&
    \hspace{-3mm}\includegraphics[scale=.92]{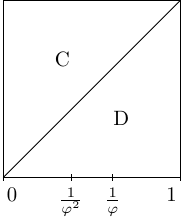}&
    \hspace{-3mm}\includegraphics[scale=.92]{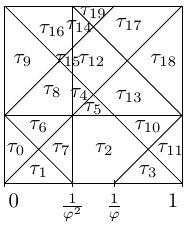}
\end{tabular}
\end{center}
    \caption{Partitions for the Example 4.
    From left to right, the partition 
    $\Ycal$ for the right color,
    $\Zcal$ for the top color,
    $R^{\be_1}(\Ycal)$ for the left color and
    $R^{\be_2}(\Zcal)$ for the bottom color.
    Their refinement is the partition $\Pcal$ where each part is associated
    with a Wang tile.}
    \label{fig:example1-partitions}
\end{figure}

\begin{figure}[h]
\begin{center}
\input{boring_tile_set.tex}
\end{center}
    \caption{The set of 20 tiles $\T=\{\tau_0,\dots,\tau_{19}\}$. 
    Each index $i\in\{0, ..., 19\}$ written in the middle of a tile corresponds
    to a tile $\tau_i$.
    The Wang shift $\Omega_{\T}$ contains periodic configurations.}
    \label{fig:boring20tiles}
\end{figure}

\begin{lemma}
The partition $\Pcal$ gives a symbolic representation
of $(\torus,\Z^2,R)$
and $(\torus, \Z^2,R)$ is the maximal equicontinuous
factor of $(\Xcal_{\Pcal,R},\Z^2,\sigma)$.
    We have that $\Xcal_{\Pcal,R}$ 
    is a strictly ergodic and 
    aperiodic subshift
    of $\Omega_{\T}$. 
    But the Wang shift $\Omega_{\T}$ contains a periodic configuration so
    $\Xcal_{\Pcal,R}\subsetneq\Omega_{\T}$.
\end{lemma}

\begin{proof}
The dynamical system $(\torus,\Z^2,R)$ is minimal.
The atom $P_{\tau_0}$ is invariant only under the trivial
translation.
Therefore, from Lemma~\ref{lem:symbolic-representation},
$\Pcal$ gives a symbolic representation of 
$(\torus,\Z^2,R)$.
From Proposition~\ref{prop:factor-map} , there exists a
factor map from $(\Xcal_{\Pcal,R},\Z^2,\sigma)$ to
$(\torus,\Z^2,R)$ and 
from Corollary~\ref{cor:max-equicontinuous-factor},
$(\torus, \Z^2,R)$ is the maximal equicontinuous
factor of $(\Xcal_{\Pcal,R},\Z^2,\sigma)$.

Since $R^{\be_1}$ and $R^{\be_2}$ are linearly independent irrational
rotations, we have
that $R$ is a free $\Z^2$-action.
Thus from Lemma~\ref{lem:minimal-aperiodic},
$\Xcal_{\Pcal,R}$ is minimal and aperiodic.
From Proposition~\ref{prop:measurable-conjugacy},
$\Xcal_{\Pcal,R}$ is uniquely ergodic thus strictly ergodic.
From Proposition~\ref{prop:is_wang_tiling},
we have $\Xcal_{\Pcal,R}\subseteq\Omega_{\T}$.
The set of Wang tiles $\T$ contains the tile $\tau_0=(A,C,A,C)$.
Let $w$ be the constant map $(m,n)\mapsto\tau_0$ for all $m,n\in\Z$.
    The configuration $w$ is valid and periodic, thus
    $w\in\Omega_{\T}\setminus\Xcal_{\Pcal,R}$.
\end{proof}

The two examples presented in this section show that
we can not expect
Theorem~\ref{thm:JR-max-equi-factor}
and
Theorem~\ref{thm:OmegaU-partition}
to hold for any given toral partition and $\Z^2$-rotation.
The characterization of toral partitions and $\Z^2$-rotations for which such results
hold is an open question.

\part{Wang shifts as model sets of cut and project schemes}\label{part:3}

This part is divided into three sections. 
Its goal is to show that occurrences of patterns in 
a minimal subshift of the Jeandel-Rao Wang shift 
and in the Wang shift $\Omega_\U$ 
are obtained as 4-to-2 cut and project schemes.

\section{Cut and project schemes and model sets}\label{sec:CPS}

In \cite{MR1456897}, the torus parametrization of three tiling dynamical systems was given.
We want to do similarly in the case of symbolic dynamical systems and in particular in the case of Wang shifts.
We recall from the more recent book \cite[\S 7.2]{MR3136260} the definition of cut and project
scheme and we reuse their notation.

\begin{definition}\label{def:CPS}
    A \defn{cut and project scheme} (CPS) is a triple $(\R^d,H,\Lcal)$ with
    a (compactly generated) locally compact Abelian group (LCAG) $H$, a lattice
    $\Lcal$ in $\R^d\times H$ and the two natural projections $\pi:\R^d\times
    H\to\R^d$ and $\pi_{\rm int}:\R^d\times H\to H$, subject to the conditions that
    $\pi|_\Lcal$ is injective and that $\pi_{\rm int}(\Lcal)$ is dense in $H$.
\end{definition}
A CPS is called \defn{Euclidean} when $H=\R^m$ for some $m\in\N$. 
A general CPS is summarized in the following diagram.
\[
\begin{tikzcd}
    \R^d & \R^d\times H \arrow[swap]{l}{\pi} \arrow{r}{\pi_{\rm int}} & H\\
    \pi(\Lcal) 
    \arrow[draw=none]{u}[sloped,auto=false]{\subset}
    \arrow[equal]{d}
    & \Lcal 
    \arrow[draw=none]{u}[sloped,auto=false]{\subset}
    \arrow[swap]{l}{1-1} \arrow{r}
    & \pi_{\rm int}(\Lcal) 
    \arrow[draw=none]{u}[sloped,auto=false]{\subset}
    \arrow[draw=none]{u}[swap]{\text{ dense}}
    \arrow[equal]{d}\\
    L \arrow{rr}{\star} && L^\star
\end{tikzcd}
\]
The image is denoted $L=\pi(\Lcal)$. Since for a given CPS, $\pi$ is a bijection between
$\Lcal$ and $L$, there is a well-defined mapping $\star:L\to H$ given by
\[
    x\mapsto x^\star:=\pi_{\rm int}\left((\pi|_\Lcal)^{-1}(x)\right)
\]
where $(\pi|_\Lcal)^{-1}(x)$ is the unique point in the set $\Lcal\cap\pi^{-1}(x)$.
This mapping is called the \defn{star map} of the CPS. The $\star$-image of $L$
is denoted by $L^\star$. The set $\Lcal$ can be viewed as a diagonal embedding of $L$ as
\[
    \Lcal=\{(x,x^\star)\mid x\in L\}.
\]
For a given CPS $(\R^d,H,\Lcal)$ and a (general) set $A\subset H$,
\[
    \curlywedge(A):=\{x\in L\mid x^\star\in A\}
\]
is the projection set within the CPS. The set $A$ is called its \defn{acceptance
set}, \defn{window} or \defn{coding set}.
\begin{definition}\label{def:model-set}
If $A\subset H$ is a relatively compact set
with non-empty interior, the projection set $\curlywedge(A)$, or any translate
$t+\curlywedge(A)$ with $t\in\R^d$, is called a \defn{model set}. 
\end{definition}
A model set is termed \defn{regular} when $\mu_H(\partial A)=0$, where $\mu_H$ is
the Haar measure of $H$. If $L^\star\cap\partial A=\varnothing$, the model set
is called \defn{generic}. If the window is not in a generic position (meaning
that $L^\star\cap\partial A\neq\varnothing$), the corresponding model set is
called \defn{singular}.

% \begin{proposition}{\rm \cite[Prop. 7.5]{MR3136260}}
%     Let $(\R^d,H,\Lcal)$ be a CPS and consider a projection set of the form
%     $\Lambda=t+\curlywedge(A)$ with $t\in\R^d$ and window $A\subset H$. 
%     If $A$
%     is relatively compact, $\Lambda$ is FLC and thus also uniformly discrete;
%     if $A^\circ\neq\varnothing$, $\Lambda$ is relatively dense. If $\Lambda$ is
%     a model set, it is also a Meyer set.
% \end{proposition}

The shape of the acceptance set $A$ is important and implies structure on the
model set $\Lambda=t+\curlywedge(A)$.
For example, if $A$ is relatively compact, $\Lambda$ has finite local
complexity and thus also is uniformly discrete; if $A^\circ\neq\varnothing$,
$\Lambda$ is relatively dense. If $\Lambda$ is a model set, it is also a Meyer
set, \cite[Prop.~7.5]{MR3136260}. For regular model set
$\Lambda=\curlywedge(A)$ with a compact window $A=\overline{A^\circ}$, it is
known \cite[Theorem 7.2]{MR3136260} that the points $\{x^\star\mid x\in\Lambda\}$
are uniformly distributed in $A$.

Linear repetitivity of model sets is an important notion.
Recall that a Delone set $Y\subseteq \R^d$ is called \defn{linearly repetitive}
if there exists a constant $C>0$ such that, for any $r\geq 1$, every patch of
size $r$ in $Y$ occurs in every ball of diameter $Cr$ in $\R^d$.
It was shown by Lagarias and Pleasants in \cite[Theorem 6.1]{MR1992666}
that linear repetitivity of a Delone set implies the existence of strict
uniform patch frequencies, equivalently the associated dynamical system on the
hull of the point set is strictly ergodic (minimal and uniquely ergodic).  As a
consequence \cite[Cor.~6.1]{MR1992666}, a linearly repetitive Delone set $X$ in
$\R^n$ has a unique autocorrelation measure $\gamma_X$. This measure $\gamma_X$
is a pure discrete measure supported on $X-X$. In particular $X$ is
diffractive.
A characterization of linearly repetitive model sets $\curlywedge(A)$
for cubical acceptance set $A$ was recently proved by Haynes, Koivusalo and
Walton \cite{MR3755878}.

%\begin{theorem} {\rm\cite{MR3755878}}
%Let $Y$ be any $k$ to $d$ cubical cut and project set, whose physical space is
%    determined by a collection $\{L_i\}_{i=1}^{k-d}$
%of linear forms on $d$ variables. Then the following are
%equivalent:
%\begin{enumerate}
%\item $Y$ is linearly repetitive;
%\item $Y$ satisfies a subadditive ergodic theorem;
%\item $\{L_i\}_{i=1}^{k-d}$ satisfies both (LR1) and (LR2):
%    \begin{enumerate}[(LR1)]
%        \item the sum of the ranks of the kernels of the maps
%            $\mathcal{L}_i:\Z^d\to\R/\Z$ defined by
%            $\mathcal{L}_i(n)=L_i(n)\mod 1$ is equal
%            to $d(k - d - 1)$;
%        \item Each $L_i$ is relatively badly approximable.
%    \end{enumerate}
%\end{enumerate}
%\end{theorem}

% \begin{theorem}\cite[Theorem 6.1]{MR1992666}
%     If $X$ is a Delone set in $\R^n$ that is either linearly repetitive or
%     densely repetitive, then $X$ has strict uniform patch frequencies.
% \end{theorem}
% 
% \begin{corollary}\cite[Cor.~6.1]{MR1992666}
% If $X$ is a linearly or densely repetitive Delone set in $\R^n$, then it has a
% unique autocorrelation measure $\gamma_X$. This measure $\gamma_X$ is a
% pure discrete measure supported on $X-X$. In particular $X$ is diffractive.
% \end{corollary}

\subsection*{Polygon exchange transformations}

We end this section with a concept that will be useful for the next two
sections. Suppose that $(\generictorus, \Z^2,R)$ is a dynamical system where
$R$ is a $\Z^2$-rotation on $\generictorus$.
The rotations $R^{\be_1}$ and $R^{\be_2}$ can be seen as polygon
exchange transformations \cite{MR3186232}
on a fundamental domain of $\generictorus$.

\begin{definition}\cite{alevy_kenyon_yi}
Let $X$ be a polygon together with
two topological partitions of $X$ into polygons
\[
    X=\bigcup_{k=0}^N P_k
     =\bigcup_{k=0}^N Q_k
\] 
such that for each $k$, $P_k$ and $Q_k$ are translation equivalent, i.e.,
there exists $v_k\in\R^2$ such that $P_k=Q_k+v_k$.
A \defn{polygon exchange transformation (PET)} is the piecewise translation
on $X$ defined for $x\in P_k$ by
$T(x) = x+v_k$.
The map is not defined for points $x\in\cup_{k=0}^N\partial P_k$.
\end{definition}

\section{A model set for the Jeandel-Rao Wang shift}

We want to describe the positions $Q\subseteq\Z^2$ of patterns in
configurations belonging to $\Xcal_{\Pcal_0,R_0}\subsetneq\Omega_0$.
Because of that, 
in the construction of a proper cut and project scheme,
we need to be careful in the choice of the locally compact Abelian group $H$
so that $\pi|_\Lcal$ is an injective map.
This is why we introduce 
the submodule $\Lambda=\langle(1,-1,0,0),(0,0,1,-1)\rangle_\Z$
and define the projections $\pi$ and $\pi_{\rm int}$ on $\R^4/\Lambda$ as:
\[
\begin{array}{rccc}
    \pi:&\R^4/\Lambda & \to & \R^2\\
        &(x_1,x_2,x_3,x_4) & \mapsto & (x_1+x_2,x_3+x_4)
\end{array}
\]
and
\[
\begin{array}{rccc}
    \pi_{\rm int}:&\R^4/\Lambda & \to & \R^2/\Gamma_0\\
        &(x_1,x_2,x_3,x_4) & \mapsto & 
        \left(x_1
             -\frac{1}{\varphi}x_2
             +\frac{1}{\varphi}x_4,
              x_3
             -(\varphi+2)x_4 \right)
\end{array}
\]
where $\varphi=\frac{1+\sqrt{5}}{2}$.
The product $\pi\times \pi_{\rm int}:\R^4/\Lambda\to\R^2\times\R^2/\Gamma_0$ 
of the projections is one-to-one and onto.
Therefore, the projections define
a Euclidean cut and project scheme
with $d=2$ and $H=\R^2/\Gamma_0$ 
on $\R^4/\Lambda \simeq \R^2\times H$.

Recall that we proved 
in Theorem~\ref{thm:JR-max-equi-factor}
that $\Xcal_{\Pcal_0,R_0}\subsetneq\Omega_0$
and that there exists a
factor map $f_0$ from $(\Xcal_{\Pcal_0,R_0},\Z^2,\sigma)$ to
$(\R^2/\Gamma_0,\Z^2,R_0)$.
Therefore any Jeandel-Rao configuration
$w\in\Xcal_{\Pcal_0,R_0}\subsetneq\Omega_0$ can be qualified as a
singular or generic according to 
whether $f_0(w)$ is in the set $\Delta_{\Pcal_0,R_0}\subset\R^2/\Gamma_0$ or not.

\begin{proof}[Proof of Theorem~\ref{thm:jeandel-rao-model-set}]
    Let $w\in \Xcal_{\Pcal_0,R_0}$.
    Let $\bx=(r,s)=r'(\varphi,0)+s'(1,\varphi+3)=f_0(w)\in\R^2/\Gamma_0$.
    We consider the lattice $\Lcal=\Z^4+(r'+s',-r'-s',
    s',-s')\subset\R^4/\Lambda$.
    We have that $\pi|_\Lcal$ is injective.
    Also $L=\pi(\Lcal)=\Z^2$ since $\pi(r'+s',-r'-s', s',-s')=\zero$.
    We also have that $\pi_{\rm int}(\Lcal)$ is dense in $H=\R^2/\Gamma_0$.
    Also $\pi_{\rm int}(r'+s',-r'-s', s',-s')=(r,s)$.

    Recall that 
    the $\Z^2$-rotation $R_0$ is defined
    on the torus $\R^2/\Gamma_0$ by $R_0^\bn(\bx) = \bx + \bn$
    for every $\bn\in\Z^2$.
    The maps $(R_0)^{\be_1}$
    and $(R_0)^{\be_2}$ can be seen as polygon exchange transformations
    on the fundamental domain $W=[0,\varphi)\times[0,\varphi+3)$ of
    $\R^2/\Gamma_0$ 
    (see Figure~\ref{fig:PET-for-JR}):
    \[
        (R_0)^{\be_1}(\bx) = 
        \begin{cases}
        \bx + v_a  & \text{ if } \bx \in P_a,\\
        \bx + v_b  & \text{ if } \bx \in P_b,
        \end{cases}
        \quad\text{ and }\quad
        (R_0)^{\be_2}(\bx) = 
        \begin{cases}
        \bx + v_c  & \text{ if } \bx \in P_c,\\
        \bx + v_d  & \text{ if } \bx \in P_d,\\
        \bx + v_e  & \text{ if } \bx \in P_e.
        \end{cases}
    \]
\begin{figure}[h]
\begin{center}
\begin{tikzpicture}[auto,scale=1.5]
% vertical axis
\def\verticalaxis{
    \draw[-latex] (0,0) -- (0,3+\p+.2);
    \foreach \x/\y in
    {0/0, 1/1, 2/2,
        \p+1  /\varphi+1,
        \p+2  /\varphi+2,
        \p+3  /\varphi+3}
    \draw (.05,\x) -- (-.05,\x) node[left] {$\y$};}
% silent vertical axis
\def\silentverticalaxis{
\draw[-latex] (0,-.1) -- (0,3+\p+.2);
\foreach \x/\y in
{0/0, 1/1, 2/2,
    \p+1  /\varphi+1,
    \p+2  /\varphi+2,
    \p+3  /\varphi+3}
\draw (.05,\x) -- (-.05,\x);}% node[left] {$\y$};}
% horizontal axis
\def\horizontalaxis{
    \draw[-latex] (0,0) -- (\p+.1,0);
    \foreach \x/\y in
    {0/0, 1/1,
        %2-\p /\frac{1}{\varphi^2},
        \p-1  /\frac{1}{\varphi},
        \p    /\varphi}
    \draw (\x,.05) -- (\x,-.05) node[below] {$\y$};}
% the domain D
%\begin{scope}
%    \draw (0,0) rectangle (1,3+\p); 
%    \node at (.5,2.31) {$D$};
%\end{scope}
\begin{scope}[xshift=0.0cm]
    \verticalaxis
    \horizontalaxis
    \draw (0,0) rectangle (\p,3+\p); 
    \draw (\p-1,0) -- (\p-1, 3+\p);
    \node at (.32, 2.31) {$P_a$};
    \node at (1.12, 2.31) {$P_b$};
\end{scope}
\begin{scope}[xshift=2.2cm]
    \silentverticalaxis
    \horizontalaxis
    \draw (0,0) rectangle (\p,3+\p); 
    \draw (1,0) -- (1, 3+\p);
    \node at (.5, 2.31) {$R_0^{\be_1}(P_b)$};
    \node[rotate=90] at (1.32, 2.31) {$R_0^{\be_1}(P_a)$};
\end{scope}
\begin{scope}[xshift=4.4cm]
    \silentverticalaxis
    \horizontalaxis
    \draw (0,0) rectangle (\p,3+\p); 
    \draw (0,\p+2) -- (\p, \p+2);
    \draw (1,\p+2) -- (1, \p+3);
    \node at (.8, 2.31) {$P_c$};
    \node at (.5, 4.1) {$P_d$};
    \node at (1.3, 4.1) {$P_e$};
\end{scope}
\begin{scope}[xshift=6.6cm]
    \silentverticalaxis
    \horizontalaxis
    \draw (0,0) rectangle (\p,3+\p); 
    \draw (0,1) -- (\p, 1);
    \draw (\p-1,0) -- (\p-1, 1);
    \node at (.8, 2.31) {$R_0^{\be_2}(P_c)$};
    \node at (1.1, .5) {$R_0^{\be_2}(P_d)$};
    \node[rotate=90] at (.3, .5) {$R_0^{\be_2}(P_e)$};
\end{scope}
\end{tikzpicture}
\end{center}
\caption{
    The maps $(R_0)^{\be_1}$
    and $(R_0)^{\be_2}$ can be seen as polygon exchange transformations
    on the fundamental domain $W=[0,\varphi)\times[0,\varphi+3)$ of
    $\R^2/\Gamma_0$.}
\label{fig:PET-for-JR}
\end{figure}
    The translations written in terms of the base of $\Z^2$ and $\Gamma_0$ and
    vice versa are:
    \begin{align*}
    v_a&=\be_1                          &\be_1&=v_a, \\
    v_b&=\be_1-(\varphi,0)              &\be_2&=v_c,\\
    v_c&=\be_2                          &(\varphi,0)&=v_a-v_b, \\
    v_d&=\be_2-(1,\varphi+3)+(\varphi,0)&(1,\varphi+3)&=v_a-v_b+v_c-v_d,\\
    v_e&=\be_2-(1,\varphi+3).
    \end{align*}
    Since $W$ is a fundamental domain for 
    $\Gamma_0=\langle(\varphi,0),(1,\varphi+3)\rangle_\Z$,
    by definition for every $\bx\in\R^2$, there exist unique $k,\ell\in\Z$ such
    that $\bx+k(\varphi,0)+\ell(1,\varphi+3)\in W$.
    Therefore, for every $(m,n)\in\Z^2$ there exist unique $k,\ell\in\Z$ such
    that the following holds
    \begin{align*}
        R_0^{(m,n)}(r,s)
            &=(r,s) + (m,n) \bmod\Gamma_0\\
            &=r'(\varphi,0)+s'(1,\varphi+3) + m\be_1 + n\be_2 +k(\varphi,0)+\ell(1,\varphi+3)\in W\\
            &=m\be_1 + n\be_2 +(r'+k)(\varphi,0)+(s'+\ell)(1,\varphi+3)\\
            &=mv_a + n v_c +(r'+k)(v_a-v_b)+(s'+\ell)(v_a-v_b+v_c-v_d)\\
            &=(m+r'+k+s'+\ell)v_a -(r'+k+s'+\ell)v_b \\
                 &\hspace{2cm}  +(n+s'+\ell)v_c- (s'+\ell) v_d\\
            &= \pi_{\rm int}((m+r'+k+s'+\ell,-r'-k-s'-\ell,\\
                  &\hspace{2cm} n+s'+\ell,-s'-\ell)+\Lambda)
            \in\pi_{\rm int}(\Lcal).
    \end{align*}
    Notice that the projection into the physical space is
    \[
    \pi((m+r'+k+s'+\ell,-r'-k-s'-\ell,
        n+s'+\ell,-s'-\ell)+\Lambda) = (m,n).
    \]
    Thus
    \[
    (m+r'+k+s'+\ell,-r'-k-s'-\ell,
        n+s'+\ell,-s'-\ell)+\Lambda
    = (\pi|_\Lcal)^{-1}(m,n)
    \]
    so that
    \begin{align*}
        (m,n)^\star &=\pi_{\rm int}\left((\pi|_\Lcal)^{-1}(m,n)\right)\\
            &= \pi_{\rm int}((m+r'+k+s'+\ell,-r'-k-s'-\ell,
                                n+s'+\ell,-s'-\ell)+\Lambda)\\
            & = R_0^{(m,n)}(r,s).
    \end{align*}

    Let $p=\pi_S(w)\in\T_0^S$
    be a pattern occurring in the configuration $w$
    for some subset $S\subset\Z^2$.
    Let $[p]$ be the cylinder associated with the pattern $p$
    and $A=f_0([p])\subset W$ be the acceptance set.
    The set $A$ is a polygon by construction, see
    Equation~\eqref{eq:cylindre-intersection-polygones}.
    Therefore the Lebesgue measure of $\partial A$ is zero.
    Assume for now that $w$ is a generic configuration.
    Since $R_0^{(m,n)}(r,s)=(m,n)^\star \notin \partial A$ for every
    $m,n\in\Z$, the set $Q\subseteq\Z^2$ of occurrences of $p$ in $w$ is
    \[
        Q = \left\{(m,n)\in\Z^2\,\middle|\, R_0^{(m,n)}(r,s)\in A\right\}
          = \{(m,n)\in L\mid (m,n)^\star\in A\}
          = \curlywedge(A)
    \]
    which is a regular and generic model set.
    If $w$ is a singular configuration, then $w=\scSymbRep_0^\bv(r,s)$ for some
    $\bv\in\R^2\setminus\Theta^{\Pcal_0}$.
    If $A=f_0([p])\subset W$,
    then we take $A'=\lim_{\epsilon\to0} A\cap(A-\epsilon \bv)$
    as acceptance set and we have
    \[
        Q = \left\{(m,n)\in\Z^2\,\middle|\, R_0^{(m,n)}(r,s)\in A'\right\}
          = \{(m,n)\in L\mid (m,n)^\star\in A'\}
          = \curlywedge(A')
    \]
    which is a regular and singular model set.
\end{proof}

\section{A model set for the Wang shift $\Omega_\U$ defined by 19 tiles}

% \[
% \begin{tikzcd}
%     \R^2 & \R^2\times\R^2 \arrow[swap]{l}{\pi} \arrow{r}{\pi_{\rm int}} &\R^2\\
%     \left(\Z[\sqrt{5}]\right)^2
%     \arrow[draw=none]{u}[sloped,auto=false]{\subset}
%     \arrow[draw=none]{u}{\text{dense }}
%     \arrow[equal]{d}
%     & \Lcal 
%     \arrow[draw=none]{u}[sloped,auto=false]{\subset}
%     \arrow[swap]{l}{1-1} \arrow{r}{1-1}
%     &\left(\Z[\sqrt{5}]\right)^2
%     \arrow[draw=none]{u}[sloped,auto=false]{\subset}
%     \arrow[draw=none]{u}[swap]{\text{ dense}}
%     \arrow[equal]{d}\\
%     L \arrow{rr}{\star} && L^\star
% \end{tikzcd}
% \]

As in the previous section we use
the submodule $\Lambda=\langle(1,-1,0,0),(0,0,1,-1)\rangle_\Z$
and define the projections on $\R^4/\Lambda$ as:
\[
\begin{array}{rccc}
    \pi:&\R^4/\Lambda & \to & \R^2\\
        &(x_1,x_2,x_3,x_4) & \mapsto & (x_1+x_2,x_3+x_4)
\end{array}
\]
and
\[
\begin{array}{rccc}
    \pi_{\rm int}:&\R^4/\Lambda & \to & \torus\\
        &(x_1,x_2,x_3,x_4) & \mapsto & 
        \left(\frac{1}{\varphi^2}x_1-\frac{1}{\varphi}x_2,
              \frac{1}{\varphi^2}x_3-\frac{1}{\varphi}x_4\right)
\end{array}
\]
where $\varphi=\frac{1+\sqrt{5}}{2}$.
The product $\pi\times \pi_{\rm int}:\R^4/\Lambda\to\R^2\times\torus$ 
of the projections is one-to-one and onto
so we may identify the domain of the projections as $\R^4/\Lambda \simeq
\R^2\times\torus$, in agreement with the definition of cut and
project scheme.

Note that if $\Lcal=\Z^4\subset\R^4/\Lambda$, then
$\pi|_\Lcal$ is injective and $L=\pi(\Lcal)=\Z^2$.
If the acceptance set is the whole cubical window $A=\torus$, we obtain
a description of the positions of patterns in a configuration
as a model set, that is, $\Z^2=\curlywedge(A)$.
%This falls in the case of \emph{cubical} window
%considered in the characterization of linearly repetitive cubical cut
%and project sets done in \cite{MR3755878}.
In the result below, noncubical acceptance sets $A\subset\torus$ are used to
describe the positions of patterns occurring in configurations.

Recall that we proved 
among other things 
in Theorem~\ref{thm:OmegaU-partition} 
that $\Xcal_{\Pcal_\U,R_\U}=\Omega_\U$
and that there exists a
factor map $f_\U$ from $(\Xcal_{\Pcal_\U,R_\U},\Z^2,\sigma)$ to
$(\torus,\Z^2,R_\U)$.
Therefore any configuration $w\in\Omega_\U$ can be qualified as
singular or generic according to 
whether $f_\U(w)$ is in the set $\Delta_{\Pcal_\U,R_\U}\subset\torus$ or not.

\begin{theorem}\label{thm:OmegaU-is-model-set}
    Let $\U$ be the self-similar set of Wang tiles shown in 
    Figure~\ref{fig:19-wang-tile-set}.
    There exists a cut and project scheme such that
    for every configuration $w\in\Omega_\U$,
    the set $Q\subseteq\Z^2$ of occurrences of a pattern in $w$
    is a regular model set.
    If $w$ is a generic (resp.~singular) configuration,
    then $Q$ is a generic (resp.~singular) model set.
    %If $w$ is a generic configuration, then $Q=\curlywedge(A)$ with acceptance set
    %$A=f_0([p])$.
\end{theorem}

\begin{proof}
    Let $w\in\Omega_\U$ be a valid configuration and 
    let $\bx=(r,s)=f_\U(w)\in\torus$.
    We consider $\Lcal=\Z^4+(r,-r,s,-s)\subset\R^4/\Lambda$.
    We have that $\pi|_\Lcal$ is injective
    and $L=\pi(\Lcal)=\Z^2$.
    We also have that $\pi_{\rm int}(\Lcal)$ is dense in $H=\torus$.
    Also $\pi_{\rm int}(r,-r,s,-s)=(r,s)$.

    Since $\pi$ is a bijection between
    $\Lcal$ and $L$, there is a well-defined mapping $\star:L\to H$ given by
    \[
        x\mapsto x^\star:=\pi_{\rm int}\left((\pi|_\Lcal)^{-1}(x)\right)
    \]
    where $(\pi|_\Lcal)^{-1}(x)$ is the unique point in the set
    $\Lcal\cap\pi^{-1}(x)$.

    Recall that 
    the $\Z^2$-rotation $R_\U$ is defined
    on the torus $\torus$ by
    $R_\U^\bn(\bx) = \bx + \varphi^{-2}\bn$
    for every $\bn\in\Z^2$.
    The maps $(R_\U)^{\be_1}$
    and $(R_\U)^{\be_2}$ can be seen as polygon exchange transformations
    on the fundamental domain $W=[0,1)^2$ of $\torus$:
    \[
        (R_\U)^{\be_1}(\bx) = 
        \begin{cases}
        \bx + v_a  & \text{ if } \bx \in P_a,\\
        \bx + v_b  & \text{ if } \bx \in P_b,
        \end{cases}
        \quad\text{ and }\quad
        (R_\U)^{\be_2}(\bx) = 
        \begin{cases}
        \bx + v_c  & \text{ if } \bx \in P_c,\\
        \bx + v_d  & \text{ if } \bx \in P_d.
        \end{cases}
    \]
    with $v_a=(\varphi^{-2},0)$, $v_b=(-\varphi^{-1},0)$,
    $v_c=(0,\varphi^{-2})$ and $v_d=(0,-\varphi^{-1})$,
    see Figure~\ref{fig:PET-for-U}.
\begin{figure}[h]
\begin{center}
\def\p{1.61803398874989}   % phi
\def\t{0.02}   % ticks length
\begin{tikzpicture}[scale=2.4]
\begin{scope}
%horizontalaxis
\draw (-\t,0) -- (1+\t,0);
\foreach \x/\label in {
    0/0,
    1/1,
    %2*\p-3/\frac{1}{\varphi^{3}},
    %2-\p  /\frac{1}{\varphi^{2}},
    \p-1  /\frac{1}{\varphi}}
\draw (\x,0+\t) -- (\x,0-\t) node[below] {$\label$};
%verticalaxis
\draw (0,-\t) -- (0,1+\t);
\foreach \y/\label in {
    0/0,
    1/1,
    %2*\p-3/\frac{1}{\varphi^{3}},
    %2-\p  /\frac{1}{\varphi^{2}},
    %\p-1  /\frac{1}{\varphi}
    }
    \draw (\t,\y) -- (-\t,\y) node[left] {$\label$};
% horizontal separation
\foreach \y in {0,1}
\draw (0,\y) -- (1,\y);
% vertical separation
\foreach \x in {0,\p-1,1}
\draw (\x,0) -- (\x,1);
% nodes
\node at (.5*\p-.5,.5) {$P_a$};
\node at (.5*\p   ,.5) {$P_b$};
\end{scope}

\begin{scope}[xshift=15mm]
%horizontalaxis
\draw (-\t,0) -- (1+\t,0);
\foreach \x/\label in {
    0/0,
    1/1,
    %2*\p-3/\frac{1}{\varphi^{3}},
    2-\p  /\frac{1}{\varphi^{2}},
    %\p-1  /\frac{1}{\varphi}
    }
\draw (\x,0+\t) -- (\x,0-\t) node[below] {$\label$};
%verticalaxis
\draw (0,-\t) -- (0,1+\t);
\foreach \y/\label in {
    0/0,
    1/1,
    %2*\p-3/\frac{1}{\varphi^{3}},
    %2-\p  /\frac{1}{\varphi^{2}},
    %\p-1  /\frac{1}{\varphi}
    }
    \draw (\t,\y) -- (-\t,\y) node[left] {$\label$};
% horizontal separation
\foreach \y in {0,1}
\draw (0,\y) -- (1,\y);
% vertical separation
\foreach \x in {0,2-\p,1}
\draw (\x,0) -- (\x,1);
% nodes
\node[rotate=90] at (1-.5*\p   ,.5) {$R_\U^{\be_1}(P_b)$};
\node at (-.5*\p+1.5,.5) {$R_\U^{\be_1}(P_a)$};
\end{scope}

\begin{scope}[xshift=30mm]
%horizontalaxis
\draw (-\t,0) -- (1+\t,0);
\foreach \x/\label in {
    0/0,
    1/1,
    %2*\p-3/\frac{1}{\varphi^{3}},
    %2-\p  /\frac{1}{\varphi^{2}},
    %\p-1  /\frac{1}{\varphi}
    }
\draw (\x,0+\t) -- (\x,0-\t) node[below] {$\label$};
%verticalaxis
\draw (0,-\t) -- (0,1+\t);
\foreach \y/\label in {
    0/0,
    1/1,
    %2*\p-3/\frac{1}{\varphi^{3}},
    %2-\p  /\frac{1}{\varphi^{2}},
    \p-1  /\frac{1}{\varphi}
    }
\draw (\t,\y) -- (-\t,\y) node[left] {$\label$};
% horizontal separation
\foreach \y in {0,\p-1,1}
\draw (0,\y) -- (1,\y);
% vertical separation
\foreach \x in {0,1}
\draw (\x,0) -- (\x,1);
% nodes
\node at (.5,.5*\p-.5) {$P_c$};
\node at (.5,.5*\p   ) {$P_d$};
\end{scope}

\begin{scope}[xshift=45mm]
%horizontalaxis
\draw (-\t,0) -- (1+\t,0);
\foreach \x/\label in {
    0/0,
    1/1,
    %2*\p-3/\frac{1}{\varphi^{3}},
    %2-\p  /\frac{1}{\varphi^{2}},
    %\p-1  /\frac{1}{\varphi}
    }
\draw (\x,0+\t) -- (\x,0-\t) node[below] {$\label$};
%verticalaxis
\draw (0,-\t) -- (0,1+\t);
\foreach \y/\label in {
    0/0,
    1/1,
    %2*\p-3/\frac{1}{\varphi^{3}},
    2-\p  /\frac{1}{\varphi^{2}},
    %\p-1  /\frac{1}{\varphi}
    }
\draw (\t,\y) -- (-\t,\y) node[left] {$\label$};
% horizontal separation
\foreach \y in {0,2-\p,1}
\draw (0,\y) -- (1,\y);
% vertical separation
\foreach \x in {0,1}
\draw (\x,0) -- (\x,1);
% nodes
\node at (.5,1-.5*\p   ) {$R_\U^{\be_2}(P_d)$};
\node at (.5,-.5*\p+1.5) {$R_\U^{\be_2}(P_c)$};
\end{scope}
\end{tikzpicture}
\end{center}
\caption{
    The maps $(R_\U)^{\be_1}$
    and $(R_\U)^{\be_2}$ can be seen as polygon exchange transformations
    on the fundamental domain $W=[0,1)^2$ of $\torus$.}
\label{fig:PET-for-U}
\end{figure}
    Notice that the base of $\Z^2$ can be written in
    terms of the translations as
    \[
        \be_1 = v_a - v_b
        \quad\text{ and }\quad
        \be_2 = v_c - v_d.
    \]
    Since $W$ is a fundamental domain for $\Z^2=\langle\be_1,\be_2\rangle_\Z$,
    for every $\bx\in\R^2$ there exist unique $k,\ell\in\Z$ such
    that $\bx+k\be_1+\ell\be_2\in W$.
    Therefore, for every $(m,n)\in\Z^2$ there exist unique $k,\ell\in\Z$ such
    that the following holds
    \begin{align*}
        R_\U^{(m,n)}(r,s)
            &=(r,s) + \frac{1}{\varphi^2}(m,n) \bmod\Z^2\\
            &=r\be_1 + s\be_2 + m v_a + n v_c \bmod\Z^2\\
            &=r\be_1 + s\be_2 + m v_a 
                    + n v_c +k\be_1+\ell\be_2 \in W\\
            &=m v_a + n v_c +(r+k)(v_a-v_b)+(s+\ell)(v_c-v_d)\\
            &=(m+r+k) v_a 
                    - (r+k) v_b
                    + (n+s+\ell) v_c
                    - (s+\ell) v_d\\
            &= \pi_{\rm int}((m+r+k,-r-k,n+s+\ell,-s-\ell)+\Lambda)
            \in\pi_{\rm int}(\Lcal).
    \end{align*}
    Notice that the projection into the physical space is
    \[
    \pi((m+r+k,-r-k,n+s+\ell,-s-\ell)+\Lambda) = (m,n).
    \]
    Thus
    \[
    (m+r+k,-r-k,n+s+\ell,-s-\ell) + \Lambda
    = (\pi|_{\Lcal})^{-1}(m,n)
    \]
    so that
    \begin{align*}
        (m,n)^\star &=\pi_{\rm int}\left((\pi|_{\Lcal})^{-1}(m,n)\right)\\
            &= \pi_{\rm int}\left((m+r+k,-r-k,n+s+\ell,-s-\ell)
                                  +\Lambda\right)\\
            & = R_\U^{(m,n)}(r,s)
            = (\{r+\varphi m\},\{s+\varphi n\})
    \end{align*}
    where $\{x\}=x-\lfloor x\rfloor$ is the fractional part of $x$.

    Let $p=\pi_S(w)\in\U^S$
    be a pattern occurring in the configuration $w$
    for some subset $S\subset\Z^2$.
    Let $[p]$ be the cylinder associated with the pattern $p$
    and $A=f_\U([p])\subset W$ be the acceptance set.
    The set $A$ is a polygon by construction, see
    Equation~\eqref{eq:cylindre-intersection-polygones}.
    Therefore the Lebesgue measure of $\partial A$ is zero.
    Assume for now that $w$ is a generic configuration.
    Since $R_\U^{(m,n)}(r,s)=(m,n)^\star \notin \partial A$ for every
    $m,n\in\Z$, the set $Q\subseteq\Z^2$ of occurrences of $p$ in $w$ is
    \[
        Q = \left\{(m,n)\in\Z^2\,\middle|\, R_\U^{(m,n)}(r,s)\in A\right\}
          = \{(m,n)\in L\mid (m,n)^\star\in A\}
          = \curlywedge(A)
    \]
    which is a regular and generic model set.
    If $w$ is a singular configuration, then $w=\scSymbRep_\U^\bv(r,s)$ for some
    $\bv\in\R^2\setminus\Theta^{\Pcal_\U}$.
    If $A=f_\U([p])\subset W$,
    then we take $A'=\lim_{\epsilon\to0} A\cap(A-\epsilon \bv)$
    as acceptance set and we have
    \[
        Q = \left\{(m,n)\in\Z^2\,\middle|\, R_\U^{(m,n)}(r,s)\in A'\right\}
          = \{(m,n)\in L\mid (m,n)^\star\in A'\}
          = \curlywedge(A')
    \]
    which is a regular and singular model set.
\end{proof}

% \begin{corollary}
%     The set of occurrences of a pattern appearing in 
%     a Wang configuration $w\in\Omega_\U$ is a Meyer set.
% \end{corollary}
% 
% \begin{proof}
%     Using the notation from the proof of
%     Theorem~\ref{thm:OmegaU-is-model-set}, 
%     the set of occurrences of a pattern appearing in a configuration in $\Omega_\U$
%     is equal to $Q=\curlywedge(A)$ where
%     the acceptance set $A$ is a non empty
%     polygon, it follows that $Q$ is a Meyer set
%     following \cite[Prop. 7.5]{MR3136260}.
% \end{proof}

In \cite{MR3978536}, $\Omega_\U$ was proved to be self-similar
being invariant under the application of an expansive and primitive
substitution. It follows that $\Omega_\U$ is linearly repetitive.
Based on \cite{MR3755878}, an alternate proof of linear repetitivity of
$\Omega_\U$ could be obtained now that $\Omega_\U$ is described as a model set.
Some more work has to be done as the characterization of linearly repetitive 
model sets provided in
\cite{MR3755878} is stated for cubical windows only.

In the present work, we made the choice of uniform $1\times 1$ size
for Wang tiles but we can make the following remark on the use of other
rectangular shapes and stone inflations.

\begin{remark}
    To use the natural size for Wang tiles in $\U$ as stone
inflation deduced from its self-similarity, see \cite[\S 7]{MR3978536}, one must use 
\[
\begin{array}{rccc}
    \pi':&\R^4 & \to & \R^2\\
        &(x_1,x_2,x_3,x_4) & \mapsto &
        (x_1+\frac{1}{\varphi}x_2,x_3+\frac{1}{\varphi}x_4).
\end{array}
\]
as projection into the physical space.
In this case, $\pi'|_\Lcal$ is injective making it a proper cut and project
scheme.
%Moreover, it is natural to use 
%\[
%\begin{array}{rccc}
%    \pi'_{\rm int}:&\R^4 & \to & \R^2\\
%        &(x_1,x_2,x_3,x_4) & \mapsto & 
%        (\frac{1}{\varphi}x_1-x_2,
%         \frac{1}{\varphi}x_3-x_4).
%\end{array}
%\]
%as the projection into the internal space which is simply a rescaling by the
%factor $\varphi$ of the projection $\pi_{\rm int}$ we have chosen to use
%above.
Another way to construct the cut and project scheme is to use
the Minkowski embedding of $\Z[\sqrt{5}]\times\Z[\sqrt{5}]$
\[
    \Lcal=\left\{(x,y,x^\star,y^\star) \,\middle|\, 
                    x,y\in\Z[\sqrt{5}]\right\}
\]
    where the star map $\star$ corresponds to the algebraic conjugation
    $(\sqrt{5})^\star=-\sqrt{5}$
    in the quadratic field $\Q(\sqrt{5})$, see \cite[\S 7]{MR3136260}.
    In this setup, the natural window to be used should be
    $W=[-1,\varphi-1)\times[-1,\varphi-1)$
    instead of $[0,1)\times[0,1)$ following known construction
    in the Fibonacci case.
    We do not provide this construction here.
\end{remark}

% \begin{corollary}
%     $Q$ is a linearly repetitive Meyer set.
% \end{corollary}
% 
% \begin{proof}
% The acceptance set $A$ is a non empty polygon, it follows
% that $Q$ is a Meyer set.
% 
% \todo[inline]{is linearly repetitive, the sum of the rank of the kernel is...}
% \end{proof}

%% file: article2_macro_jeandel_rao.tex
\newcommand\JeandelRaoO{
\begin{tikzpicture}
[scale=0.900000000000000]
\tikzstyle{every node}=[font=\footnotesize]
% tile at position (x,y)=(0, 0)
\fill[cyan] (1, 0) -- (0.5, 0.5) -- (1, 1);
\fill[lightgray] (0, 1) -- (0.5, 0.5) -- (1, 1);
\fill[cyan] (0, 0) -- (0.5, 0.5) -- (0, 1);
\fill[red] (0, 0) -- (0.5, 0.5) -- (1, 0);
\draw (1, 0) -- ++ (0,1);
\draw (0, 1) -- ++ (1,0);
\draw (0, 0) -- ++ (0,1);
\draw (0, 0) -- ++ (1,0);
\node[rotate=0,black] at (0.800000000000000, 0.5) {2};
\node[rotate=0,black] at (0.5, 0.800000000000000) {4};
\node[rotate=0,black] at (0.200000000000000, 0.5) {2};
\node[rotate=0,black] at (0.5, 0.200000000000000) {1};
\end{tikzpicture}
} % end of newcommand
\newcommand\JeandelRaoI{
\begin{tikzpicture}
[scale=0.900000000000000]
\tikzstyle{every node}=[font=\footnotesize]
% tile at position (x,y)=(0, 0)
\fill[cyan] (1, 0) -- (0.5, 0.5) -- (1, 1);
\fill[cyan] (0, 1) -- (0.5, 0.5) -- (1, 1);
\fill[cyan] (0, 0) -- (0.5, 0.5) -- (0, 1);
\fill[white] (0, 0) -- (0.5, 0.5) -- (1, 0);
\draw (1, 0) -- ++ (0,1);
\draw (0, 1) -- ++ (1,0);
\draw (0, 0) -- ++ (0,1);
\draw (0, 0) -- ++ (1,0);
\node[rotate=0,black] at (0.800000000000000, 0.5) {2};
\node[rotate=0,black] at (0.5, 0.800000000000000) {2};
\node[rotate=0,black] at (0.200000000000000, 0.5) {2};
\node[rotate=0,black] at (0.5, 0.200000000000000) {0};
\end{tikzpicture}
} % end of newcommand
\newcommand\JeandelRaoII{
\begin{tikzpicture}
[scale=0.900000000000000]
\tikzstyle{every node}=[font=\footnotesize]
% tile at position (x,y)=(0, 0)
\fill[red] (1, 0) -- (0.5, 0.5) -- (1, 1);
\fill[red] (0, 1) -- (0.5, 0.5) -- (1, 1);
\fill[green] (0, 0) -- (0.5, 0.5) -- (0, 1);
\fill[red] (0, 0) -- (0.5, 0.5) -- (1, 0);
\draw (1, 0) -- ++ (0,1);
\draw (0, 1) -- ++ (1,0);
\draw (0, 0) -- ++ (0,1);
\draw (0, 0) -- ++ (1,0);
\node[rotate=0,black] at (0.800000000000000, 0.5) {1};
\node[rotate=0,black] at (0.5, 0.800000000000000) {1};
\node[rotate=0,black] at (0.200000000000000, 0.5) {3};
\node[rotate=0,black] at (0.5, 0.200000000000000) {1};
\end{tikzpicture}
} % end of newcommand
\newcommand\JeandelRaoIII{
\begin{tikzpicture}
[scale=0.900000000000000]
\tikzstyle{every node}=[font=\footnotesize]
% tile at position (x,y)=(0, 0)
\fill[red] (1, 0) -- (0.5, 0.5) -- (1, 1);
\fill[cyan] (0, 1) -- (0.5, 0.5) -- (1, 1);
\fill[green] (0, 0) -- (0.5, 0.5) -- (0, 1);
\fill[cyan] (0, 0) -- (0.5, 0.5) -- (1, 0);
\draw (1, 0) -- ++ (0,1);
\draw (0, 1) -- ++ (1,0);
\draw (0, 0) -- ++ (0,1);
\draw (0, 0) -- ++ (1,0);
\node[rotate=0,black] at (0.800000000000000, 0.5) {1};
\node[rotate=0,black] at (0.5, 0.800000000000000) {2};
\node[rotate=0,black] at (0.200000000000000, 0.5) {3};
\node[rotate=0,black] at (0.5, 0.200000000000000) {2};
\end{tikzpicture}
} % end of newcommand
\newcommand\JeandelRaoIV{
\begin{tikzpicture}
[scale=0.900000000000000]
\tikzstyle{every node}=[font=\footnotesize]
% tile at position (x,y)=(0, 0)
\fill[green] (1, 0) -- (0.5, 0.5) -- (1, 1);
\fill[red] (0, 1) -- (0.5, 0.5) -- (1, 1);
\fill[green] (0, 0) -- (0.5, 0.5) -- (0, 1);
\fill[green] (0, 0) -- (0.5, 0.5) -- (1, 0);
\draw (1, 0) -- ++ (0,1);
\draw (0, 1) -- ++ (1,0);
\draw (0, 0) -- ++ (0,1);
\draw (0, 0) -- ++ (1,0);
\node[rotate=0,black] at (0.800000000000000, 0.5) {3};
\node[rotate=0,black] at (0.5, 0.800000000000000) {1};
\node[rotate=0,black] at (0.200000000000000, 0.5) {3};
\node[rotate=0,black] at (0.5, 0.200000000000000) {3};
\end{tikzpicture}
} % end of newcommand
\newcommand\JeandelRaoV{
\begin{tikzpicture}
[scale=0.900000000000000]
\tikzstyle{every node}=[font=\footnotesize]
% tile at position (x,y)=(0, 0)
\fill[white] (1, 0) -- (0.5, 0.5) -- (1, 1);
\fill[red] (0, 1) -- (0.5, 0.5) -- (1, 1);
\fill[green] (0, 0) -- (0.5, 0.5) -- (0, 1);
\fill[red] (0, 0) -- (0.5, 0.5) -- (1, 0);
\draw (1, 0) -- ++ (0,1);
\draw (0, 1) -- ++ (1,0);
\draw (0, 0) -- ++ (0,1);
\draw (0, 0) -- ++ (1,0);
\node[rotate=0,black] at (0.800000000000000, 0.5) {0};
\node[rotate=0,black] at (0.5, 0.800000000000000) {1};
\node[rotate=0,black] at (0.200000000000000, 0.5) {3};
\node[rotate=0,black] at (0.5, 0.200000000000000) {1};
\end{tikzpicture}
} % end of newcommand
\newcommand\JeandelRaoVI{
\begin{tikzpicture}
[scale=0.900000000000000]
\tikzstyle{every node}=[font=\footnotesize]
% tile at position (x,y)=(0, 0)
\fill[white] (1, 0) -- (0.5, 0.5) -- (1, 1);
\fill[white] (0, 1) -- (0.5, 0.5) -- (1, 1);
\fill[white] (0, 0) -- (0.5, 0.5) -- (0, 1);
\fill[red] (0, 0) -- (0.5, 0.5) -- (1, 0);
\draw (1, 0) -- ++ (0,1);
\draw (0, 1) -- ++ (1,0);
\draw (0, 0) -- ++ (0,1);
\draw (0, 0) -- ++ (1,0);
\node[rotate=0,black] at (0.800000000000000, 0.5) {0};
\node[rotate=0,black] at (0.5, 0.800000000000000) {0};
\node[rotate=0,black] at (0.200000000000000, 0.5) {0};
\node[rotate=0,black] at (0.5, 0.200000000000000) {1};
\end{tikzpicture}
} % end of newcommand
\newcommand\JeandelRaoVII{
\begin{tikzpicture}
[scale=0.900000000000000]
\tikzstyle{every node}=[font=\footnotesize]
% tile at position (x,y)=(0, 0)
\fill[green] (1, 0) -- (0.5, 0.5) -- (1, 1);
\fill[red] (0, 1) -- (0.5, 0.5) -- (1, 1);
\fill[white] (0, 0) -- (0.5, 0.5) -- (0, 1);
\fill[cyan] (0, 0) -- (0.5, 0.5) -- (1, 0);
\draw (1, 0) -- ++ (0,1);
\draw (0, 1) -- ++ (1,0);
\draw (0, 0) -- ++ (0,1);
\draw (0, 0) -- ++ (1,0);
\node[rotate=0,black] at (0.800000000000000, 0.5) {3};
\node[rotate=0,black] at (0.5, 0.800000000000000) {1};
\node[rotate=0,black] at (0.200000000000000, 0.5) {0};
\node[rotate=0,black] at (0.5, 0.200000000000000) {2};
\end{tikzpicture}
} % end of newcommand
\newcommand\JeandelRaoVIII{
\begin{tikzpicture}
[scale=0.900000000000000]
\tikzstyle{every node}=[font=\footnotesize]
% tile at position (x,y)=(0, 0)
\fill[white] (1, 0) -- (0.5, 0.5) -- (1, 1);
\fill[cyan] (0, 1) -- (0.5, 0.5) -- (1, 1);
\fill[red] (0, 0) -- (0.5, 0.5) -- (0, 1);
\fill[cyan] (0, 0) -- (0.5, 0.5) -- (1, 0);
\draw (1, 0) -- ++ (0,1);
\draw (0, 1) -- ++ (1,0);
\draw (0, 0) -- ++ (0,1);
\draw (0, 0) -- ++ (1,0);
\node[rotate=0,black] at (0.800000000000000, 0.5) {0};
\node[rotate=0,black] at (0.5, 0.800000000000000) {2};
\node[rotate=0,black] at (0.200000000000000, 0.5) {1};
\node[rotate=0,black] at (0.5, 0.200000000000000) {2};
\end{tikzpicture}
} % end of newcommand
\newcommand\JeandelRaoIX{
\begin{tikzpicture}
[scale=0.900000000000000]
\tikzstyle{every node}=[font=\footnotesize]
% tile at position (x,y)=(0, 0)
\fill[red] (1, 0) -- (0.5, 0.5) -- (1, 1);
\fill[cyan] (0, 1) -- (0.5, 0.5) -- (1, 1);
\fill[red] (0, 0) -- (0.5, 0.5) -- (0, 1);
\fill[lightgray] (0, 0) -- (0.5, 0.5) -- (1, 0);
\draw (1, 0) -- ++ (0,1);
\draw (0, 1) -- ++ (1,0);
\draw (0, 0) -- ++ (0,1);
\draw (0, 0) -- ++ (1,0);
\node[rotate=0,black] at (0.800000000000000, 0.5) {1};
\node[rotate=0,black] at (0.5, 0.800000000000000) {2};
\node[rotate=0,black] at (0.200000000000000, 0.5) {1};
\node[rotate=0,black] at (0.5, 0.200000000000000) {4};
\end{tikzpicture}
} % end of newcommand
\newcommand\JeandelRaoX{
\begin{tikzpicture}
[scale=0.900000000000000]
\tikzstyle{every node}=[font=\footnotesize]
% tile at position (x,y)=(0, 0)
\fill[green] (1, 0) -- (0.5, 0.5) -- (1, 1);
\fill[green] (0, 1) -- (0.5, 0.5) -- (1, 1);
\fill[red] (0, 0) -- (0.5, 0.5) -- (0, 1);
\fill[cyan] (0, 0) -- (0.5, 0.5) -- (1, 0);
\draw (1, 0) -- ++ (0,1);
\draw (0, 1) -- ++ (1,0);
\draw (0, 0) -- ++ (0,1);
\draw (0, 0) -- ++ (1,0);
\node[rotate=0,black] at (0.800000000000000, 0.5) {3};
\node[rotate=0,black] at (0.5, 0.800000000000000) {3};
\node[rotate=0,black] at (0.200000000000000, 0.5) {1};
\node[rotate=0,black] at (0.5, 0.200000000000000) {2};
\end{tikzpicture}
} % end of newcommand

%% file: article3_tile_set_U.tex
\begin{tikzpicture}
[scale=1]
\tikzstyle{every node}=[font=\tiny]
% tile at position (x,y)=(0.0, 0.0)
\node[] at (0.5, 0.5) {0};
\draw (1.0, 0.0) -- ++ (0,1);
\draw (0.0, 1.0) -- ++ (1,0);
\draw (0.0, 0.0) -- ++ (0,1);
\draw (0.0, 0.0) -- ++ (1,0);
\node[rotate=0,black] at (0.8, 0.5) {F};
\node[rotate=0,black] at (0.5, 0.8) {O};
\node[rotate=0,black] at (0.2, 0.5) {J};
\node[rotate=0,black] at (0.5, 0.2) {O};
% tile at position (x,y)=(1.1, 0.0)
\node[] at (1.6, 0.5) {1};
\draw (2.1, 0.0) -- ++ (0,1);
\draw (1.1, 1.0) -- ++ (1,0);
\draw (1.1, 0.0) -- ++ (0,1);
\draw (1.1, 0.0) -- ++ (1,0);
\node[rotate=0,black] at (1.9000000000000001, 0.5) {F};
\node[rotate=0,black] at (1.6, 0.8) {O};
\node[rotate=0,black] at (1.3, 0.5) {H};
\node[rotate=0,black] at (1.6, 0.2) {L};
% tile at position (x,y)=(2.2, 0.0)
\node[] at (2.7, 0.5) {2};
\draw (3.2, 0.0) -- ++ (0,1);
\draw (2.2, 1.0) -- ++ (1,0);
\draw (2.2, 0.0) -- ++ (0,1);
\draw (2.2, 0.0) -- ++ (1,0);
\node[rotate=0,black] at (3.0, 0.5) {J};
\node[rotate=0,black] at (2.7, 0.8) {M};
\node[rotate=0,black] at (2.4000000000000004, 0.5) {F};
\node[rotate=0,black] at (2.7, 0.2) {P};
% tile at position (x,y)=(3.3000000000000003, 0.0)
\node[] at (3.8000000000000003, 0.5) {3};
\draw (4.300000000000001, 0.0) -- ++ (0,1);
\draw (3.3000000000000003, 1.0) -- ++ (1,0);
\draw (3.3000000000000003, 0.0) -- ++ (0,1);
\draw (3.3000000000000003, 0.0) -- ++ (1,0);
\node[rotate=0,black] at (4.1000000000000005, 0.5) {D};
\node[rotate=0,black] at (3.8000000000000003, 0.8) {M};
\node[rotate=0,black] at (3.5000000000000004, 0.5) {F};
\node[rotate=0,black] at (3.8000000000000003, 0.2) {K};
% tile at position (x,y)=(4.4, 0.0)
\node[] at (4.9, 0.5) {4};
\draw (5.4, 0.0) -- ++ (0,1);
\draw (4.4, 1.0) -- ++ (1,0);
\draw (4.4, 0.0) -- ++ (0,1);
\draw (4.4, 0.0) -- ++ (1,0);
\node[rotate=0,black] at (5.2, 0.5) {H};
\node[rotate=0,black] at (4.9, 0.8) {P};
\node[rotate=0,black] at (4.6000000000000005, 0.5) {J};
\node[rotate=0,black] at (4.9, 0.2) {P};
% tile at position (x,y)=(5.5, 0.0)
\node[] at (6.0, 0.5) {5};
\draw (6.5, 0.0) -- ++ (0,1);
\draw (5.5, 1.0) -- ++ (1,0);
\draw (5.5, 0.0) -- ++ (0,1);
\draw (5.5, 0.0) -- ++ (1,0);
\node[rotate=0,black] at (6.3, 0.5) {H};
\node[rotate=0,black] at (6.0, 0.8) {P};
\node[rotate=0,black] at (5.7, 0.5) {H};
\node[rotate=0,black] at (6.0, 0.2) {N};
% tile at position (x,y)=(6.6000000000000005, 0.0)
\node[] at (7.1000000000000005, 0.5) {6};
\draw (7.6000000000000005, 0.0) -- ++ (0,1);
\draw (6.6000000000000005, 1.0) -- ++ (1,0);
\draw (6.6000000000000005, 0.0) -- ++ (0,1);
\draw (6.6000000000000005, 0.0) -- ++ (1,0);
\node[rotate=0,black] at (7.4, 0.5) {H};
\node[rotate=0,black] at (7.1000000000000005, 0.8) {K};
\node[rotate=0,black] at (6.800000000000001, 0.5) {F};
\node[rotate=0,black] at (7.1000000000000005, 0.2) {P};
% tile at position (x,y)=(7.700000000000001, 0.0)
\node[] at (8.200000000000001, 0.5) {7};
\draw (8.700000000000001, 0.0) -- ++ (0,1);
\draw (7.700000000000001, 1.0) -- ++ (1,0);
\draw (7.700000000000001, 0.0) -- ++ (0,1);
\draw (7.700000000000001, 0.0) -- ++ (1,0);
\node[rotate=0,black] at (8.500000000000002, 0.5) {H};
\node[rotate=0,black] at (8.200000000000001, 0.8) {K};
\node[rotate=0,black] at (7.900000000000001, 0.5) {D};
\node[rotate=0,black] at (8.200000000000001, 0.2) {P};
% tile at position (x,y)=(8.8, 0.0)
\node[] at (9.3, 0.5) {8};
\draw (9.8, 0.0) -- ++ (0,1);
\draw (8.8, 1.0) -- ++ (1,0);
\draw (8.8, 0.0) -- ++ (0,1);
\draw (8.8, 0.0) -- ++ (1,0);
\node[rotate=0,black] at (9.600000000000001, 0.5) {B};
\node[rotate=0,black] at (9.3, 0.8) {O};
\node[rotate=0,black] at (9.0, 0.5) {I};
\node[rotate=0,black] at (9.3, 0.2) {O};
% tile at position (x,y)=(9.9, 0.0)
\node[] at (10.4, 0.5) {9};
\draw (10.9, 0.0) -- ++ (0,1);
\draw (9.9, 1.0) -- ++ (1,0);
\draw (9.9, 0.0) -- ++ (0,1);
\draw (9.9, 0.0) -- ++ (1,0);
\node[rotate=0,black] at (10.700000000000001, 0.5) {G};
\node[rotate=0,black] at (10.4, 0.8) {L};
\node[rotate=0,black] at (10.1, 0.5) {E};
\node[rotate=0,black] at (10.4, 0.2) {O};
% tile at position (x,y)=(0.0, -1.1)
\node[] at (0.5, -0.6000000000000001) {10};
\draw (1.0, -1.1) -- ++ (0,1);
\draw (0.0, -0.10000000000000009) -- ++ (1,0);
\draw (0.0, -1.1) -- ++ (0,1);
\draw (0.0, -1.1) -- ++ (1,0);
\node[rotate=0,black] at (0.8, -0.6000000000000001) {G};
\node[rotate=0,black] at (0.5, -0.3000000000000001) {L};
\node[rotate=0,black] at (0.2, -0.6000000000000001) {C};
\node[rotate=0,black] at (0.5, -0.9000000000000001) {L};
% tile at position (x,y)=(1.1, -1.1)
\node[] at (1.6, -0.6000000000000001) {11};
\draw (2.1, -1.1) -- ++ (0,1);
\draw (1.1, -0.10000000000000009) -- ++ (1,0);
\draw (1.1, -1.1) -- ++ (0,1);
\draw (1.1, -1.1) -- ++ (1,0);
\node[rotate=0,black] at (1.9000000000000001, -0.6000000000000001) {A};
\node[rotate=0,black] at (1.6, -0.3000000000000001) {L};
\node[rotate=0,black] at (1.3, -0.6000000000000001) {I};
\node[rotate=0,black] at (1.6, -0.9000000000000001) {O};
% tile at position (x,y)=(2.2, -1.1)
\node[] at (2.7, -0.6000000000000001) {12};
\draw (3.2, -1.1) -- ++ (0,1);
\draw (2.2, -0.10000000000000009) -- ++ (1,0);
\draw (2.2, -1.1) -- ++ (0,1);
\draw (2.2, -1.1) -- ++ (1,0);
\node[rotate=0,black] at (3.0, -0.6000000000000001) {E};
\node[rotate=0,black] at (2.7, -0.3000000000000001) {P};
\node[rotate=0,black] at (2.4000000000000004, -0.6000000000000001) {G};
\node[rotate=0,black] at (2.7, -0.9000000000000001) {P};
% tile at position (x,y)=(3.3000000000000003, -1.1)
\node[] at (3.8000000000000003, -0.6000000000000001) {13};
\draw (4.300000000000001, -1.1) -- ++ (0,1);
\draw (3.3000000000000003, -0.10000000000000009) -- ++ (1,0);
\draw (3.3000000000000003, -1.1) -- ++ (0,1);
\draw (3.3000000000000003, -1.1) -- ++ (1,0);
\node[rotate=0,black] at (4.1000000000000005, -0.6000000000000001) {E};
\node[rotate=0,black] at (3.8000000000000003, -0.3000000000000001) {P};
\node[rotate=0,black] at (3.5000000000000004, -0.6000000000000001) {I};
\node[rotate=0,black] at (3.8000000000000003, -0.9000000000000001) {P};
% tile at position (x,y)=(4.4, -1.1)
\node[] at (4.9, -0.6000000000000001) {14};
\draw (5.4, -1.1) -- ++ (0,1);
\draw (4.4, -0.10000000000000009) -- ++ (1,0);
\draw (4.4, -1.1) -- ++ (0,1);
\draw (4.4, -1.1) -- ++ (1,0);
\node[rotate=0,black] at (5.2, -0.6000000000000001) {I};
\node[rotate=0,black] at (4.9, -0.3000000000000001) {P};
\node[rotate=0,black] at (4.6000000000000005, -0.6000000000000001) {G};
\node[rotate=0,black] at (4.9, -0.9000000000000001) {K};
% tile at position (x,y)=(5.5, -1.1)
\node[] at (6.0, -0.6000000000000001) {15};
\draw (6.5, -1.1) -- ++ (0,1);
\draw (5.5, -0.10000000000000009) -- ++ (1,0);
\draw (5.5, -1.1) -- ++ (0,1);
\draw (5.5, -1.1) -- ++ (1,0);
\node[rotate=0,black] at (6.3, -0.6000000000000001) {I};
\node[rotate=0,black] at (6.0, -0.3000000000000001) {P};
\node[rotate=0,black] at (5.7, -0.6000000000000001) {I};
\node[rotate=0,black] at (6.0, -0.9000000000000001) {K};
% tile at position (x,y)=(6.6000000000000005, -1.1)
\node[] at (7.1000000000000005, -0.6000000000000001) {16};
\draw (7.6000000000000005, -1.1) -- ++ (0,1);
\draw (6.6000000000000005, -0.10000000000000009) -- ++ (1,0);
\draw (6.6000000000000005, -1.1) -- ++ (0,1);
\draw (6.6000000000000005, -1.1) -- ++ (1,0);
\node[rotate=0,black] at (7.4, -0.6000000000000001) {I};
\node[rotate=0,black] at (7.1000000000000005, -0.3000000000000001) {K};
\node[rotate=0,black] at (6.800000000000001, -0.6000000000000001) {B};
\node[rotate=0,black] at (7.1000000000000005, -0.9000000000000001) {M};
% tile at position (x,y)=(7.700000000000001, -1.1)
\node[] at (8.200000000000001, -0.6000000000000001) {17};
\draw (8.700000000000001, -1.1) -- ++ (0,1);
\draw (7.700000000000001, -0.10000000000000009) -- ++ (1,0);
\draw (7.700000000000001, -1.1) -- ++ (0,1);
\draw (7.700000000000001, -1.1) -- ++ (1,0);
\node[rotate=0,black] at (8.500000000000002, -0.6000000000000001) {I};
\node[rotate=0,black] at (8.200000000000001, -0.3000000000000001) {K};
\node[rotate=0,black] at (7.900000000000001, -0.6000000000000001) {A};
\node[rotate=0,black] at (8.200000000000001, -0.9000000000000001) {K};
% tile at position (x,y)=(8.8, -1.1)
\node[] at (9.3, -0.6000000000000001) {18};
\draw (9.8, -1.1) -- ++ (0,1);
\draw (8.8, -0.10000000000000009) -- ++ (1,0);
\draw (8.8, -1.1) -- ++ (0,1);
\draw (8.8, -1.1) -- ++ (1,0);
\node[rotate=0,black] at (9.600000000000001, -0.6000000000000001) {C};
\node[rotate=0,black] at (9.3, -0.3000000000000001) {N};
\node[rotate=0,black] at (9.0, -0.6000000000000001) {I};
\node[rotate=0,black] at (9.3, -0.9000000000000001) {P};
\end{tikzpicture}

%% file: boring_tile_set.tex
\begin{tikzpicture}
[scale=1]
\tikzstyle{every node}=[font=\tiny]
% tile at position (x,y)=(0.0, 0.0)
\node[] at (0.5, 0.5) {0};
\draw (1.0, 0.0) -- ++ (0,1);
\draw (0.0, 1.0) -- ++ (1,0);
\draw (0.0, 0.0) -- ++ (0,1);
\draw (0.0, 0.0) -- ++ (1,0);
\node[rotate=0,black] at (0.8, 0.5) {A};
\node[rotate=0,black] at (0.5, 0.8) {C};
\node[rotate=0,black] at (0.2, 0.5) {A};
\node[rotate=0,black] at (0.5, 0.2) {C};
% tile at position (x,y)=(1.1, 0.0)
\node[] at (1.6, 0.5) {1};
\draw (2.1, 0.0) -- ++ (0,1);
\draw (1.1, 1.0) -- ++ (1,0);
\draw (1.1, 0.0) -- ++ (0,1);
\draw (1.1, 0.0) -- ++ (1,0);
\node[rotate=0,black] at (1.9000000000000001, 0.5) {A};
\node[rotate=0,black] at (1.6, 0.8) {C};
\node[rotate=0,black] at (1.3, 0.5) {A};
\node[rotate=0,black] at (1.6, 0.2) {D};
% tile at position (x,y)=(2.2, 0.0)
\node[] at (2.7, 0.5) {2};
\draw (3.2, 0.0) -- ++ (0,1);
\draw (2.2, 1.0) -- ++ (1,0);
\draw (2.2, 0.0) -- ++ (0,1);
\draw (2.2, 0.0) -- ++ (1,0);
\node[rotate=0,black] at (3.0, 0.5) {A};
\node[rotate=0,black] at (2.7, 0.8) {C};
\node[rotate=0,black] at (2.4000000000000004, 0.5) {A};
\node[rotate=0,black] at (2.7, 0.2) {D};
% tile at position (x,y)=(3.3000000000000003, 0.0)
\node[] at (3.8000000000000003, 0.5) {3};
\draw (4.300000000000001, 0.0) -- ++ (0,1);
\draw (3.3000000000000003, 1.0) -- ++ (1,0);
\draw (3.3000000000000003, 0.0) -- ++ (0,1);
\draw (3.3000000000000003, 0.0) -- ++ (1,0);
\node[rotate=0,black] at (4.1000000000000005, 0.5) {A};
\node[rotate=0,black] at (3.8000000000000003, 0.8) {D};
\node[rotate=0,black] at (3.5000000000000004, 0.5) {A};
\node[rotate=0,black] at (3.8000000000000003, 0.2) {D};
% tile at position (x,y)=(4.4, 0.0)
\node[] at (4.9, 0.5) {4};
\draw (5.4, 0.0) -- ++ (0,1);
\draw (4.4, 1.0) -- ++ (1,0);
\draw (4.4, 0.0) -- ++ (0,1);
\draw (4.4, 0.0) -- ++ (1,0);
\node[rotate=0,black] at (5.2, 0.5) {A};
\node[rotate=0,black] at (4.9, 0.8) {D};
\node[rotate=0,black] at (4.6000000000000005, 0.5) {A};
\node[rotate=0,black] at (4.9, 0.2) {C};
% tile at position (x,y)=(5.5, 0.0)
\node[] at (6.0, 0.5) {5};
\draw (6.5, 0.0) -- ++ (0,1);
\draw (5.5, 1.0) -- ++ (1,0);
\draw (5.5, 0.0) -- ++ (0,1);
\draw (5.5, 0.0) -- ++ (1,0);
\node[rotate=0,black] at (6.3, 0.5) {A};
\node[rotate=0,black] at (6.0, 0.8) {D};
\node[rotate=0,black] at (5.7, 0.5) {A};
\node[rotate=0,black] at (6.0, 0.2) {D};
% tile at position (x,y)=(6.6000000000000005, 0.0)
\node[] at (7.1000000000000005, 0.5) {6};
\draw (7.6000000000000005, 0.0) -- ++ (0,1);
\draw (6.6000000000000005, 1.0) -- ++ (1,0);
\draw (6.6000000000000005, 0.0) -- ++ (0,1);
\draw (6.6000000000000005, 0.0) -- ++ (1,0);
\node[rotate=0,black] at (7.4, 0.5) {B};
\node[rotate=0,black] at (7.1000000000000005, 0.8) {C};
\node[rotate=0,black] at (6.800000000000001, 0.5) {A};
\node[rotate=0,black] at (7.1000000000000005, 0.2) {C};
% tile at position (x,y)=(7.700000000000001, 0.0)
\node[] at (8.200000000000001, 0.5) {7};
\draw (8.700000000000001, 0.0) -- ++ (0,1);
\draw (7.700000000000001, 1.0) -- ++ (1,0);
\draw (7.700000000000001, 0.0) -- ++ (0,1);
\draw (7.700000000000001, 0.0) -- ++ (1,0);
\node[rotate=0,black] at (8.500000000000002, 0.5) {B};
\node[rotate=0,black] at (8.200000000000001, 0.8) {C};
\node[rotate=0,black] at (7.900000000000001, 0.5) {A};
\node[rotate=0,black] at (8.200000000000001, 0.2) {D};
% tile at position (x,y)=(8.8, 0.0)
\node[] at (9.3, 0.5) {8};
\draw (9.8, 0.0) -- ++ (0,1);
\draw (8.8, 1.0) -- ++ (1,0);
\draw (8.8, 0.0) -- ++ (0,1);
\draw (8.8, 0.0) -- ++ (1,0);
\node[rotate=0,black] at (9.600000000000001, 0.5) {B};
\node[rotate=0,black] at (9.3, 0.8) {D};
\node[rotate=0,black] at (9.0, 0.5) {A};
\node[rotate=0,black] at (9.3, 0.2) {C};
% tile at position (x,y)=(9.9, 0.0)
\node[] at (10.4, 0.5) {9};
\draw (10.9, 0.0) -- ++ (0,1);
\draw (9.9, 1.0) -- ++ (1,0);
\draw (9.9, 0.0) -- ++ (0,1);
\draw (9.9, 0.0) -- ++ (1,0);
\node[rotate=0,black] at (10.700000000000001, 0.5) {B};
\node[rotate=0,black] at (10.4, 0.8) {C};
\node[rotate=0,black] at (10.1, 0.5) {A};
\node[rotate=0,black] at (10.4, 0.2) {C};
% tile at position (x,y)=(0.0, -1.1)
\node[] at (0.5, -0.6000000000000001) {10};
\draw (1.0, -1.1) -- ++ (0,1);
\draw (0.0, -0.10000000000000009) -- ++ (1,0);
\draw (0.0, -1.1) -- ++ (0,1);
\draw (0.0, -1.1) -- ++ (1,0);
\node[rotate=0,black] at (0.8, -0.6000000000000001) {A};
\node[rotate=0,black] at (0.5, -0.3000000000000001) {C};
\node[rotate=0,black] at (0.2, -0.6000000000000001) {B};
\node[rotate=0,black] at (0.5, -0.9000000000000001) {D};
% tile at position (x,y)=(1.1, -1.1)
\node[] at (1.6, -0.6000000000000001) {11};
\draw (2.1, -1.1) -- ++ (0,1);
\draw (1.1, -0.10000000000000009) -- ++ (1,0);
\draw (1.1, -1.1) -- ++ (0,1);
\draw (1.1, -1.1) -- ++ (1,0);
\node[rotate=0,black] at (1.9000000000000001, -0.6000000000000001) {A};
\node[rotate=0,black] at (1.6, -0.3000000000000001) {D};
\node[rotate=0,black] at (1.3, -0.6000000000000001) {B};
\node[rotate=0,black] at (1.6, -0.9000000000000001) {D};
% tile at position (x,y)=(2.2, -1.1)
\node[] at (2.7, -0.6000000000000001) {12};
\draw (3.2, -1.1) -- ++ (0,1);
\draw (2.2, -0.10000000000000009) -- ++ (1,0);
\draw (2.2, -1.1) -- ++ (0,1);
\draw (2.2, -1.1) -- ++ (1,0);
\node[rotate=0,black] at (3.0, -0.6000000000000001) {A};
\node[rotate=0,black] at (2.7, -0.3000000000000001) {D};
\node[rotate=0,black] at (2.4000000000000004, -0.6000000000000001) {B};
\node[rotate=0,black] at (2.7, -0.9000000000000001) {C};
% tile at position (x,y)=(3.3000000000000003, -1.1)
\node[] at (3.8000000000000003, -0.6000000000000001) {13};
\draw (4.300000000000001, -1.1) -- ++ (0,1);
\draw (3.3000000000000003, -0.10000000000000009) -- ++ (1,0);
\draw (3.3000000000000003, -1.1) -- ++ (0,1);
\draw (3.3000000000000003, -1.1) -- ++ (1,0);
\node[rotate=0,black] at (4.1000000000000005, -0.6000000000000001) {A};
\node[rotate=0,black] at (3.8000000000000003, -0.3000000000000001) {D};
\node[rotate=0,black] at (3.5000000000000004, -0.6000000000000001) {B};
\node[rotate=0,black] at (3.8000000000000003, -0.9000000000000001) {D};
% tile at position (x,y)=(4.4, -1.1)
\node[] at (4.9, -0.6000000000000001) {14};
\draw (5.4, -1.1) -- ++ (0,1);
\draw (4.4, -0.10000000000000009) -- ++ (1,0);
\draw (4.4, -1.1) -- ++ (0,1);
\draw (4.4, -1.1) -- ++ (1,0);
\node[rotate=0,black] at (5.2, -0.6000000000000001) {A};
\node[rotate=0,black] at (4.9, -0.3000000000000001) {C};
\node[rotate=0,black] at (4.6000000000000005, -0.6000000000000001) {B};
\node[rotate=0,black] at (4.9, -0.9000000000000001) {C};
% tile at position (x,y)=(5.5, -1.1)
\node[] at (6.0, -0.6000000000000001) {15};
\draw (6.5, -1.1) -- ++ (0,1);
\draw (5.5, -0.10000000000000009) -- ++ (1,0);
\draw (5.5, -1.1) -- ++ (0,1);
\draw (5.5, -1.1) -- ++ (1,0);
\node[rotate=0,black] at (6.3, -0.6000000000000001) {B};
\node[rotate=0,black] at (6.0, -0.3000000000000001) {D};
\node[rotate=0,black] at (5.7, -0.6000000000000001) {B};
\node[rotate=0,black] at (6.0, -0.9000000000000001) {C};
% tile at position (x,y)=(6.6000000000000005, -1.1)
\node[] at (7.1000000000000005, -0.6000000000000001) {16};
\draw (7.6000000000000005, -1.1) -- ++ (0,1);
\draw (6.6000000000000005, -0.10000000000000009) -- ++ (1,0);
\draw (6.6000000000000005, -1.1) -- ++ (0,1);
\draw (6.6000000000000005, -1.1) -- ++ (1,0);
\node[rotate=0,black] at (7.4, -0.6000000000000001) {B};
\node[rotate=0,black] at (7.1000000000000005, -0.3000000000000001) {C};
\node[rotate=0,black] at (6.800000000000001, -0.6000000000000001) {B};
\node[rotate=0,black] at (7.1000000000000005, -0.9000000000000001) {C};
% tile at position (x,y)=(7.700000000000001, -1.1)
\node[] at (8.200000000000001, -0.6000000000000001) {17};
\draw (8.700000000000001, -1.1) -- ++ (0,1);
\draw (7.700000000000001, -0.10000000000000009) -- ++ (1,0);
\draw (7.700000000000001, -1.1) -- ++ (0,1);
\draw (7.700000000000001, -1.1) -- ++ (1,0);
\node[rotate=0,black] at (8.500000000000002, -0.6000000000000001) {B};
\node[rotate=0,black] at (8.200000000000001, -0.3000000000000001) {D};
\node[rotate=0,black] at (7.900000000000001, -0.6000000000000001) {B};
\node[rotate=0,black] at (8.200000000000001, -0.9000000000000001) {C};
% tile at position (x,y)=(8.8, -1.1)
\node[] at (9.3, -0.6000000000000001) {18};
\draw (9.8, -1.1) -- ++ (0,1);
\draw (8.8, -0.10000000000000009) -- ++ (1,0);
\draw (8.8, -1.1) -- ++ (0,1);
\draw (8.8, -1.1) -- ++ (1,0);
\node[rotate=0,black] at (9.600000000000001, -0.6000000000000001) {B};
\node[rotate=0,black] at (9.3, -0.3000000000000001) {D};
\node[rotate=0,black] at (9.0, -0.6000000000000001) {B};
\node[rotate=0,black] at (9.3, -0.9000000000000001) {D};
% tile at position (x,y)=(9.9, -1.1)
\node[] at (10.4, -0.6000000000000001) {19};
\draw (10.9, -1.1) -- ++ (0,1);
\draw (9.9, -0.10000000000000009) -- ++ (1,0);
\draw (9.9, -1.1) -- ++ (0,1);
\draw (9.9, -1.1) -- ++ (1,0);
\node[rotate=0,black] at (10.700000000000001, -0.6000000000000001) {B};
\node[rotate=0,black] at (10.4, -0.3000000000000001) {C};
\node[rotate=0,black] at (10.1, -0.6000000000000001) {B};
\node[rotate=0,black] at (10.4, -0.9000000000000001) {C};
\end{tikzpicture}

%% file: solution5x5_61725_scale3.tex
\begin{tikzpicture}
[scale=3,ultra thick]
\tikzstyle{every node}=[font=\bfseries\LARGE]
% tile at position (x,y)=(0, 0)
\node[black] at (0.5, 0.5) {6};
\draw[blue] (0, 0) -- ++ (0,1);
\draw[blue] (0, 0) -- ++ (.35,0) -- ++ (.15,.20) -- ++ (.15,-.20) -- ++ (.35,0);
% tile at position (x,y)=(0, 1)
\node[black] at (0.5, 1.5) {1};
\draw[blue] (0, 1) -- ++ (0,.25) arc (-90:90:.1) -- ++ (0,.1) arc (-90:90:.1) -- ++ (0,.25);
\draw[blue] (0, 1) -- ++ (1,0);
% tile at position (x,y)=(0, 2)
\node[black] at (0.5, 2.5) {7};
\draw[blue] (0, 2) -- ++ (0,1);
\draw[blue] (0, 2) -- ++ (.3,0) -- ++ (.1,.15) -- ++ (.1,-.15) -- ++ (.1,.15) -- ++ (.1,-.15) -- ++ (.3,0);
% tile at position (x,y)=(0, 3)
\node[black] at (0.5, 3.5) {2};
\draw[blue] (0, 3) -- ++ (0,.15) arc (-90:90:.1) -- ++ (0,.05) arc (-90:90:.1) -- ++ (0,.05) arc (-90:90:.1) -- ++ (0,.15);
\draw[blue] (0, 3) -- ++ (.35,0) -- ++ (.15,.20) -- ++ (.15,-.20) -- ++ (.35,0);
% tile at position (x,y)=(0, 4)
\node[black] at (0.5, 4.5) {5};
\draw[blue] (0, 5) -- ++ (.35,0) -- ++ (.15,.20) -- ++ (.15,-.20) -- ++ (.35,0);
\draw[blue] (0, 4) -- ++ (0,.15) arc (-90:90:.1) -- ++ (0,.05) arc (-90:90:.1) -- ++ (0,.05) arc (-90:90:.1) -- ++ (0,.15);
\draw[blue] (0, 4) -- ++ (.35,0) -- ++ (.15,.20) -- ++ (.15,-.20) -- ++ (.35,0);
% tile at position (x,y)=(1, 0)
\node[black] at (1.5, 0.5) {6};
\draw[blue] (1, 0) -- ++ (0,1);
\draw[blue] (1, 0) -- ++ (.35,0) -- ++ (.15,.20) -- ++ (.15,-.20) -- ++ (.35,0);
% tile at position (x,y)=(1, 1)
\node[black] at (1.5, 1.5) {1};
\draw[blue] (1, 1) -- ++ (0,.25) arc (-90:90:.1) -- ++ (0,.1) arc (-90:90:.1) -- ++ (0,.25);
\draw[blue] (1, 1) -- ++ (1,0);
% tile at position (x,y)=(1, 2)
\node[black] at (1.5, 2.5) {3};
\draw[blue] (1, 2) -- ++ (0,.15) arc (-90:90:.1) -- ++ (0,.05) arc (-90:90:.1) -- ++ (0,.05) arc (-90:90:.1) -- ++ (0,.15);
\draw[blue] (1, 2) -- ++ (.3,0) -- ++ (.1,.15) -- ++ (.1,-.15) -- ++ (.1,.15) -- ++ (.1,-.15) -- ++ (.3,0);
% tile at position (x,y)=(1, 3)
\node[black] at (1.5, 3.5) {8};
\draw[blue] (1, 3) -- ++ (0,.35) arc (-90:90:.15) -- ++ (0,.35);
\draw[blue] (1, 3) -- ++ (.3,0) -- ++ (.1,.15) -- ++ (.1,-.15) -- ++ (.1,.15) -- ++ (.1,-.15) -- ++ (.3,0);
% tile at position (x,y)=(1, 4)
\node[black] at (1.5, 4.5) {7};
\draw[blue] (1, 5) -- ++ (.35,0) -- ++ (.15,.20) -- ++ (.15,-.20) -- ++ (.35,0);
\draw[blue] (1, 4) -- ++ (0,1);
\draw[blue] (1, 4) -- ++ (.3,0) -- ++ (.1,.15) -- ++ (.1,-.15) -- ++ (.1,.15) -- ++ (.1,-.15) -- ++ (.3,0);
% tile at position (x,y)=(2, 0)
\node[black] at (2.5, 0.5) {7};
\draw[blue] (2, 0) -- ++ (0,1);
\draw[blue] (2, 0) -- ++ (.3,0) -- ++ (.1,.15) -- ++ (.1,-.15) -- ++ (.1,.15) -- ++ (.1,-.15) -- ++ (.3,0);
% tile at position (x,y)=(2, 1)
\node[black] at (2.5, 1.5) {0};
\draw[blue] (2, 1) -- ++ (0,.25) arc (-90:90:.1) -- ++ (0,.1) arc (-90:90:.1) -- ++ (0,.25);
\draw[blue] (2, 1) -- ++ (.35,0) -- ++ (.15,.20) -- ++ (.15,-.20) -- ++ (.35,0);
% tile at position (x,y)=(2, 2)
\node[black] at (2.5, 2.5) {9};
\draw[blue] (2, 2) -- ++ (0,.35) arc (-90:90:.15) -- ++ (0,.35);
\draw[blue] (2, 2) -- ++ (.1,0) -- ++ (.1,.15) -- ++ (.1,-.15) -- ++ (.1,.15) -- ++ (.1,-.15) -- ++ (.1,.15) -- ++ (.1,-.15) -- ++ (.1,.15) -- ++ (.1,-.15) -- ++ (.1,0);
% tile at position (x,y)=(2, 3)
\node[black] at (2.5, 3.5) {7};
\draw[blue] (2, 3) -- ++ (0,1);
\draw[blue] (2, 3) -- ++ (.3,0) -- ++ (.1,.15) -- ++ (.1,-.15) -- ++ (.1,.15) -- ++ (.1,-.15) -- ++ (.3,0);
% tile at position (x,y)=(2, 4)
\node[black] at (2.5, 4.5) {5};
\draw[blue] (2, 5) -- ++ (.35,0) -- ++ (.15,.20) -- ++ (.15,-.20) -- ++ (.35,0);
\draw[blue] (2, 4) -- ++ (0,.15) arc (-90:90:.1) -- ++ (0,.05) arc (-90:90:.1) -- ++ (0,.05) arc (-90:90:.1) -- ++ (0,.15);
\draw[blue] (2, 4) -- ++ (.35,0) -- ++ (.15,.20) -- ++ (.15,-.20) -- ++ (.35,0);
% tile at position (x,y)=(3, 0)
\node[black] at (3.5, 0.5) {4};
\draw[blue] (3, 0) -- ++ (0,.15) arc (-90:90:.1) -- ++ (0,.05) arc (-90:90:.1) -- ++ (0,.05) arc (-90:90:.1) -- ++ (0,.15);
\draw[blue] (3, 0) -- ++ (.2,0) -- ++ (.1,.15) -- ++ (.1,-.15) -- ++ (.1,.15) -- ++ (.1,-.15) -- ++ (.1,.15) -- ++ (.1,-.15) -- ++ (.2,0);
% tile at position (x,y)=(3, 1)
\node[black] at (3.5, 1.5) {0};
\draw[blue] (3, 1) -- ++ (0,.25) arc (-90:90:.1) -- ++ (0,.1) arc (-90:90:.1) -- ++ (0,.25);
\draw[blue] (3, 1) -- ++ (.35,0) -- ++ (.15,.20) -- ++ (.15,-.20) -- ++ (.35,0);
% tile at position (x,y)=(3, 2)
\node[black] at (3.5, 2.5) {9};
\draw[blue] (3, 2) -- ++ (0,.35) arc (-90:90:.15) -- ++ (0,.35);
\draw[blue] (3, 2) -- ++ (.1,0) -- ++ (.1,.15) -- ++ (.1,-.15) -- ++ (.1,.15) -- ++ (.1,-.15) -- ++ (.1,.15) -- ++ (.1,-.15) -- ++ (.1,.15) -- ++ (.1,-.15) -- ++ (.1,0);
% tile at position (x,y)=(3, 3)
\node[black] at (3.5, 3.5) {3};
\draw[blue] (3, 3) -- ++ (0,.15) arc (-90:90:.1) -- ++ (0,.05) arc (-90:90:.1) -- ++ (0,.05) arc (-90:90:.1) -- ++ (0,.15);
\draw[blue] (3, 3) -- ++ (.3,0) -- ++ (.1,.15) -- ++ (.1,-.15) -- ++ (.1,.15) -- ++ (.1,-.15) -- ++ (.3,0);
% tile at position (x,y)=(3, 4)
\node[black] at (3.5, 4.5) {7};
\draw[blue] (3, 5) -- ++ (.35,0) -- ++ (.15,.20) -- ++ (.15,-.20) -- ++ (.35,0);
\draw[blue] (3, 4) -- ++ (0,1);
\draw[blue] (3, 4) -- ++ (.3,0) -- ++ (.1,.15) -- ++ (.1,-.15) -- ++ (.1,.15) -- ++ (.1,-.15) -- ++ (.3,0);
% tile at position (x,y)=(4, 0)
\node[black] at (4.5, 0.5) {5};
\draw[blue] (5, 0) -- ++ (0,1);
\draw[blue] (4, 0) -- ++ (0,.15) arc (-90:90:.1) -- ++ (0,.05) arc (-90:90:.1) -- ++ (0,.05) arc (-90:90:.1) -- ++ (0,.15);
\draw[blue] (4, 0) -- ++ (.35,0) -- ++ (.15,.20) -- ++ (.15,-.20) -- ++ (.35,0);
% tile at position (x,y)=(4, 1)
\node[black] at (4.5, 1.5) {0};
\draw[blue] (5, 1) -- ++ (0,.25) arc (-90:90:.1) -- ++ (0,.1) arc (-90:90:.1) -- ++ (0,.25);
\draw[blue] (4, 1) -- ++ (0,.25) arc (-90:90:.1) -- ++ (0,.1) arc (-90:90:.1) -- ++ (0,.25);
\draw[blue] (4, 1) -- ++ (.35,0) -- ++ (.15,.20) -- ++ (.15,-.20) -- ++ (.35,0);
% tile at position (x,y)=(4, 2)
\node[black] at (4.5, 2.5) {9};
\draw[blue] (5, 2) -- ++ (0,.35) arc (-90:90:.15) -- ++ (0,.35);
\draw[blue] (4, 2) -- ++ (0,.35) arc (-90:90:.15) -- ++ (0,.35);
\draw[blue] (4, 2) -- ++ (.1,0) -- ++ (.1,.15) -- ++ (.1,-.15) -- ++ (.1,.15) -- ++ (.1,-.15) -- ++ (.1,.15) -- ++ (.1,-.15) -- ++ (.1,.15) -- ++ (.1,-.15) -- ++ (.1,0);
% tile at position (x,y)=(4, 3)
\node[black] at (4.5, 3.5) {10};
\draw[blue] (5, 3) -- ++ (0,.15) arc (-90:90:.1) -- ++ (0,.05) arc (-90:90:.1) -- ++ (0,.05) arc (-90:90:.1) -- ++ (0,.15);
\draw[blue] (4, 3) -- ++ (0,.35) arc (-90:90:.15) -- ++ (0,.35);
\draw[blue] (4, 3) -- ++ (.3,0) -- ++ (.1,.15) -- ++ (.1,-.15) -- ++ (.1,.15) -- ++ (.1,-.15) -- ++ (.3,0);
% tile at position (x,y)=(4, 4)
\node[black] at (4.5, 4.5) {4};
\draw[blue] (5, 4) -- ++ (0,.15) arc (-90:90:.1) -- ++ (0,.05) arc (-90:90:.1) -- ++ (0,.05) arc (-90:90:.1) -- ++ (0,.15);
\draw[blue] (4, 5) -- ++ (.35,0) -- ++ (.15,.20) -- ++ (.15,-.20) -- ++ (.35,0);
\draw[blue] (4, 4) -- ++ (0,.15) arc (-90:90:.1) -- ++ (0,.05) arc (-90:90:.1) -- ++ (0,.05) arc (-90:90:.1) -- ++ (0,.15);
\draw[blue] (4, 4) -- ++ (.2,0) -- ++ (.1,.15) -- ++ (.1,-.15) -- ++ (.1,.15) -- ++ (.1,-.15) -- ++ (.1,.15) -- ++ (.1,-.15) -- ++ (.2,0);
\node[yshift=2.5mm] at (1,5) {\Huge\ScissorRightBrokenBottom};
\end{tikzpicture}

%% file: JR_partition_appendix.tex
% Golden numbers
\def\p{1.61803398874989}   % phi
\def\phiPLUStrois{4.61803398874989}

\newlength\myunit % definition
\setlength{\myunit}{1cm} % first setting

%\begin{center}
%    \textsc{The Universal Jeandel-Rao tiling solver}\\
%    \textcircled{c} Sébastien Labbé, November 2018.
%\end{center}

%\title{The Universal Jeandel-Rao tiling solver}
%\author{Sébastien Labbé}
%\maketitle

%\input{poster_macro_jeandel_rao.tex}
\def\JeandelRaoO{0}
\def\JeandelRaoI{1}
\def\JeandelRaoII{2}
\def\JeandelRaoIII{3}
\def\JeandelRaoIV{4}
\def\JeandelRaoV{5}
\def\JeandelRaoVI{6}
\def\JeandelRaoVII{7}
\def\JeandelRaoVIII{8}
\def\JeandelRaoIX{9}
\def\JeandelRaoX{10}
\def\JeandelRaoXI{11}

\newcommand\thepartition{
    % Tile 0
    \draw[dashed,fill=green!80] (0,0) -- (\p,0) -- (\p,1) -- (1,0) -- (1,1) --
    (\p-1,0) -- (\p-1,1) -- cycle;
    % Tile 1
    \draw[dashed,fill=purple!50] (0,0) -- (0,1) -- (\p,1) -- (1,0) -- (1,1) --
    (\p-1,0) -- (\p-1,1) -- cycle;
    % Tile 2
    \draw[dashed,fill=magenta!50] (0,2) -- (\p-1,\p+2) -- (0,\p+1) -- cycle;
    % Tile 3
    \draw[dashed,fill=black!30] (1,1) -- (\p,2) -- (\p,\p+2) -- (1,2) -- cycle;
    % Tile 4
    \draw[dashed,fill=cyan!80] (0,2) -- (1,\p+2) -- (1,\p+3) -- cycle;
    % Tile 5
    \draw[dashed,fill=yellow!50] (0,\p+1) -- (\p-1,\p+2) -- (1,\p+3) --
    (2-\p,\p+2) -- (2-\p,\p+3) -- (0,\p+2) -- cycle;
    % Tile 6
    \draw[dashed,fill=red!70] (0,3+\p) -- (0,2+\p) -- (2-\p,3+\p) --
    (2-\p,2+\p) -- (1,3+\p)-- (1,2+\p) -- (\p, 3+\p) -- cycle;
    % Tile 7 (part 1)
    \draw[dashed,fill=orange!50] (0,1) -- (1,\p+2) -- (0,2) -- cycle;
    % Tile 7 (part 2)
    \draw[dashed,fill=orange!50] (1,1+\p) -- (1,2+\p) -- (\p,3+\p) -- (\p,2+\p)
    -- cycle;
    % Tile 8
    \draw[dashed,fill=blue!70] (\p-1,1) -- (\p,\p+2) -- (\p-1,2) -- cycle;
    % Tile 9
    \draw[dashed,fill=pink!50] (0,1) -- (\p,1) -- (\p,2) -- (1,1) -- (1,2) --
    (\p-1,1) -- (\p-1,2) -- cycle;
    % Tile 10
    \draw[dashed,fill=green!40] (0,1) -- (1,\p+1) -- (1,\p+2) -- cycle;

    %%contour
    %\draw[very thick] (0,0) rectangle (\p,3+\p);

%r=2
\node at (0.16,4.44) {\JeandelRaoVI};
\node at (0.64,4.44) {\JeandelRaoVI};
\node at (1.22,4.44) {\JeandelRaoVI};
\node at (1.29,3.62) {\JeandelRaoVII};
%r=3
\node at (0.24,3.43) {\JeandelRaoV};
\node at (0.77,3.62) {\JeandelRaoIV};
%r=4
\node at (0.15,2.65) {\JeandelRaoII};
\node at (0.81,2.65) {\JeandelRaoX};
\node at (0.82,2.00) {\JeandelRaoVIII};
%r=0 + var
\node at (0.18,2.00) {\JeandelRaoVII};
\node at (1.30,2.00) {\JeandelRaoIII};
\node at (0.40,1.20) {\JeandelRaoIX};
\node at (0.84,1.20) {\JeandelRaoIX};
\node at (1.42,1.20) {\JeandelRaoIX};
%r=1 + var
\node at (0.23,0.78) {\JeandelRaoI};
\node at (0.77,0.78) {\JeandelRaoI};
\node at (1.26,0.78) {\JeandelRaoI};
\node at (0.40,0.20) {\JeandelRaoO};
\node at (0.84,0.20) {\JeandelRaoO};
\node at (1.42,0.20) {\JeandelRaoO};
} % end of \thepartition

% vertical axis
\def\verticalaxis{
    \draw[-latex] (0,0) -- (0,3+\p+.15);
    \foreach \x/\y in
    {0/0, 1/1, 2/2,
        \p+1  /\varphi+1,
        \p+2  /\varphi+2,
        \p+3  /\varphi+3}
    \draw (.00,\x) -- (-.05,\x) node[left] {$\y$};
}

% horizontal axis
\def\horizontalaxis{
\draw[-latex] (-.1,0) -- (\p+0.2,0);
\foreach \x/\y in
{0/0, 1/1,
    %2*\p-3/\frac{1}{\varphi^{3}},
    2-\p  /\frac{1}{\varphi^{2}},
    \p-1  /\frac{1}{\varphi},
    \p    /\varphi}
\draw (\x,0+.05) -- (\x,0-.05) node[below] {$\y$};
}

%FOR TESTING
%\begin{tikzpicture} [scale=3]
%\horizontalaxis
%\verticalaxis
%\thepartition
%\end{tikzpicture}

% \begin{center}
%     \includegraphics[width=\linewidth]{Decoupe_laser/T0_shapes.pdf}
% \end{center}

\begin{tikzpicture}[scale=3]
%\tikzstyle{every node}=[font=\footnotesize]
\tikzstyle{every node}=[font=\bfseries\LARGE]

    % % A2 Format Parameters
    % \def\MINX{-7}
    % \def\MAXX{6}
    % \def\MINY{-2}
    % \def\MAXY{1}
    % \clip (-9,-6) rectangle (9,6);

    % % A3 Format Parameters
    % \def\MINX{-5}
    % \def\MAXX{3}
    % \def\MINY{-1}
    % \def\MAXY{2}
    % \clip (-4.5,-1.0) rectangle (4.5,10.5);

    % A4 Format Parameters (portrait)
    \def\MINX{-3}
    \def\MAXX{2}
    \def\MINY{-1}
    \def\MAXY{1}
    \clip (-2.5,-1.5) rectangle (3,6);

    % % A4 Format Parameters
    % \def\MINX{-5}
    % \def\MAXX{3}
    % \def\MINY{-1}
    % \def\MAXY{1}
    % \clip (-4.5,-1) rectangle (4.5,5.5);

    % The unfolded partition of the flat torus
    \foreach \x in {\MINX,...,\MAXX}
    \foreach \y in {\MINY,...,\MAXY}
    {
        \begin{scope}[xshift=\p*\x*\myunit+\y*\myunit, 
                      yshift=\phiPLUStrois*\y*\myunit]
        \thepartition
        \end{scope}
    }

    % The fundamental domain contour
    \draw[ultra thick,double] (0,0) rectangle (\p,\p+3);

    % The lattice Gamma
    \foreach \x in {\MINX,...,\MAXX}
    \foreach \y in {\MINY,...,\MAXY}
    {
        \begin{scope}[xshift=\p*\x*\myunit+\y*\myunit, 
                      yshift=\phiPLUStrois*\y*\myunit]
        \node[circle,draw=black, fill=black, inner sep=0pt,
                minimum size=10pt] at (0,0) {};
        \end{scope}
    }

    % ZERO the origin
    %\node[circle,draw=black, fill=black, inner sep=0pt,
    %        minimum size=15pt] at (0,0) {};

\end{tikzpicture}